\documentclass{amsart}

\usepackage{amsmath}
\usepackage{amsfonts}
\usepackage{amssymb}
\usepackage{graphicx}
\usepackage{wrapfig}
\usepackage{algorithmic}
\usepackage{todonotes}
\usepackage{xcolor}
\usepackage{array}
\usepackage{extarrows}
\usepackage{wasysym}
\usepackage{dutchcal}
\usepackage{tikz}
\usepackage{pgfplots}
\usetikzlibrary{plotmarks,arrows}
\tikzset{arrow/.style={-latex, shorten >= 1ex, shorten <=1ex}}
\usepackage{pgfplotstable}
\pgfplotsset{compat=1.12}
\usetikzlibrary{external}
\usetikzlibrary{shapes.arrows}
\usepackage[algo2e,ruled, linesnumbered]{algorithm2e}
\usepackage{footmisc}
\usepackage{bbm}

\usepackage[skip=0pt, font=footnotesize, width = \linewidth]{caption}
\numberwithin{equation}{section}

\usepackage{hyperref}
\hypersetup{colorlinks=true,allcolors=blue}

\usepackage[capitalize,nameinlink]{cleveref}
\crefname{section}{section}{sections}
\crefname{subsection}{subsection}{subsections}
\Crefname{section}{Section}{Sections}
\Crefname{subsection}{Subsection}{Subsections}
\crefname{algo}{Algorithm}{Algorithms}
\crefname{table}{Table}{Tables}
\crefformat{equation}{\textup{#2(#1)#3}}
\crefrangeformat{equation}{\textup{#3(#1)#4--#5(#2)#6}}
\crefmultiformat{equation}{\textup{#2(#1)#3}}{ and \textup{#2(#1)#3}}
{, \textup{#2(#1)#3}}{, and \textup{#2(#1)#3}}
\crefrangemultiformat{equation}{\textup{#3(#1)#4--#5(#2)#6}}%
{ and \textup{#3(#1)#4--#5(#2)#6}}{, \textup{#3(#1)#4--#5(#2)#6}}{, and \textup{#3(#1)#4--#5(#2)#6}}
\Crefformat{equation}{#2Equation~\textup{(#1)}#3}
\Crefrangeformat{equation}{Equations~\textup{#3(#1)#4--#5(#2)#6}}
\Crefmultiformat{equation}{Equations~\textup{#2(#1)#3}}{ and \textup{#2(#1)#3}}
{, \textup{#2(#1)#3}}{, and \textup{#2(#1)#3}}
\Crefrangemultiformat{equation}{Equations~\textup{#3(#1)#4--#5(#2)#6}}%
{ and \textup{#3(#1)#4--#5(#2)#6}}{, \textup{#3(#1)#4--#5(#2)#6}}{, and \textup{#3(#1)#4--#5(#2)#6}}
\crefdefaultlabelformat{#2\textup{#1}#3}

\newtheorem{theorem}{Theorem}[section]
\newtheorem{definition}[theorem]{Definition}
\newtheorem{proposition}[theorem]{Proposition}

\newtheorem{remark}[theorem]{Remark}

\newtheorem{assumption}{Assumption}

\DeclareMathOperator{\Div}{div}

\DeclareMathOperator{\spann}{span}
\newcommand{\la}{\langle} 
\newcommand{\ra}{\rangle}
\newcommand{\lb}{\lbrace} 
\newcommand{\rb}{\rbrace}
\newcommand{\R}{\mathbb{R}} 
\newcommand{\N}{\mathbb{N}}
\newcommand{\HH}{\mathcal{H}} 
 
\newcommand{\tol}{\texttt{tol}}
\newcommand{\nrand}{n_\mathrm{rand}}
\newcommand{\nt}{n_\mathrm{t}}
\newcommand{\proj}{P}

\ifpdf
\hypersetup{
	pdftitle={Randomized quasi-optimal local approximation spaces in time},
	pdfauthor={J. Schleu{\ss}, K. Smetana, and L. ter Maat}
}
\fi

\title{Randomized quasi-optimal local approximation spaces in time}

\author{Julia Schleu{\ss}}
\address{Faculty of Mathematics and Computer Science, University of M\"unster, Einsteinstr. 62, 48149 M\"unster, Germany, julia.schleuss@uni-muenster.de.}
	
\author{Kathrin Smetana}
\address{Department of Mathematical Sciences, Stevens Institute of Technology, 1 Castle Point Terrace, Hoboken, NJ 07030, United States of America, ksmetana@stevens.edu.}

\author{Lukas ter Maat}
\address{Currently Master student at the University of Twente, The Netherlands, private address: Kieftsbeeklaan 33, 7607TA Almelo, The Netherlands, lukas.tm@hotmail.nl.}

\date{\today}

\thanks{The work of Julia Schleu{\ss} was funded by the Deutsche Forschungsgemeinschaft (DFG, German Research Foundation) under Germany's Excellence Strategy EXC 2044-390685587, Mathematics M\"unster: Dynamics-Geometry-Structure.}



\begin{document}
	
	\begin{abstract}
	We target time-dependent partial differential equations (PDEs) with heterogeneous coefficients in space and time. To tackle these problems, we construct reduced basis/ multiscale ansatz functions defined in space that can be combined with time stepping schemes within model order reduction or multiscale methods. To that end, we propose to perform several simulations of the PDE for few time steps in parallel starting at different, randomly drawn start points, prescribing random initial conditions; applying a singular value decomposition to a subset of the so obtained snapshots yields the reduced basis/ multiscale ansatz functions. This facilitates constructing the reduced basis/ multiscale ansatz functions in an embarrassingly parallel manner. In detail, we suggest using a data-dependent probability distribution based on the data functions of the PDE to select the start points. Each local in time simulation of the PDE with random initial conditions approximates a local approximation space in one time point that is optimal in the sense of Kolmogorov. The derivation of these optimal local approximation spaces which are spanned by the left singular vectors of a compact transfer operator that maps arbitrary initial conditions to the solution of the PDE in a later point of time, is one other main contribution of this paper. By solving the PDE locally in time with random initial conditions, we construct local ansatz spaces in time that converge provably at a quasi-optimal rate and allow for local error control. Numerical experiments demonstrate that the proposed method can outperform existing methods like the proper orthogonal decomposition even in a sequential setting and is well capable of approximating advection-dominated problems.
	\end{abstract}

	\maketitle

%%%%%%%%%%%%%%

\section{Introduction}
\label{introduction}

Applications that require repeated simulations for different parameters or a real-time simulation response of complex systems of partial differential equations (PDEs) or dynamical systems are ubiquitous. Moreover, heterogeneous problems that exhibit multiscale features or include rough data functions are particularly challenging. A direct numerical simulation using standard techniques such as the finite element (FE) method can be prohibitively expensive for such tasks. Well-known strategies to tackle these (heterogeneous) problems comprise multiscale methods which are based on local ansatz functions that incorporate the local behavior of the (numerical) solution of the PDE and model order reduction methods that exploit a carefully chosen set of problem-adapted basis functions to reduce the high-dimensional problem. 

In this paper, we consider heterogeneous time-dependent PDEs and propose reduced basis/ multiscale ansatz functions defined in space that can be combined with time stepping schemes within model order reduction or multiscale methods. We provide one of the first contributions that facilitates constructing the reduced basis/ multiscale ansatz functions in an embarrassingly parallel manner in time. As a major new contribution in this paper, we select important points in time and only perform local simulations of the PDE on the corresponding local time intervals instead of decomposing the global time interval into consecutive subintervals \cite{ChuEf18,LjuMai21}. As the numerical experiments show, this can result in a reduced total number of computed time steps, whereas the existing approaches \cite{ChuEf18,LjuMai21} require (local) computations everywhere in the entire time interval. To choose relevant points in time, we employ data-dependent sampling strategies from randomized numerical linear algebra (NLA) \cite{DriMah16,DerMah21} in a completely new context since they are usually used to construct low-rank matrix decompositions (cf., e.g., \cite{MahDri09}). As another key contribution we derive for the first time local approximation spaces in time that are optimal in the sense of Kolmogorov. Moreover, we provide for the first time a rigorous local a priori error analysis in time as one major new contribution.

A well-established tool for compressing and reducing time trajectories is the proper orthogonal decomposition (POD) \cite{BerHol93,KunVol01,Sir87}, which is based on a singular value decomposition (SVD) of the functions evaluated in the time grid points and allows for error control. However, in order to perform a POD on simulation data, the (global) solution trajectory of the considered problem has to be computed sequentially prior to reducing. 

In contrast, the approach we propose in this paper enables, as one major contribution, to construct reduced basis/ multiscale ansatz functions in parallel in time. To facilitate a time-parallel procedure, we propose to perform several simulations of the PDE for only few time steps in parallel. To this end, we start the simulations at different start time points that are randomly drawn from a data-dependent sampling distribution and prescribe random initial conditions. Subsequently, we apply an SVD to a subset of the computed snapshots to obtain the reduced basis/ multiscale ansatz functions. The proposed method is thus well-suited to be used on modern computer architectures allowing for many parallel computations and on each single compute unit a simulation for only few time steps has to be performed. Moreover, as another major contribution, the approach is especially tailored to time-dependent problems with heterogeneous time-dependent data functions. To draw start time points for the temporally local PDE simulations, we employ uniform, squared norm \cite{FrKaVe04}, or leverage score \cite{DrMaMu08} sampling, which are standard sampling techniques from randomized NLA \cite{DerMah21,DriMah16}. In particular, both squared norm and leverage score sampling take into account the time-dependent data functions of the PDE and are commonly used in a variety of applications \cite{DerMah21}, for instance, to construct CUR\footnote{A CUR decomposition of a matrix A consists of three matrices C, U, and R, where C (R) contains columns (rows) of A and U is constructed such that the product CUR approximates A. As the matrices C and R are constructed from actual elements of A, the decomposition is usually more interpretable with respect to the original data compared to, e.g., a truncated SVD (cf., e.g., \cite{MahDri09}).} or similar matrix decompositions by approximating the matrix via its columns or rows. As one major contribution of this paper, we employ these methods in a completely new context for the purpose of time point selection.

To this end, the time-dependent data functions are discretized and represented by a matrix, where each column of the matrix corresponds to one time point. Moreover, other randomized subset selection technique, as proposed, for instance, in \cite{AlaMah15,DesRad06,DesRad10,DriMag12}, may also be used to choose time points.

The key observation, motivating a localized construction in time, is that for certain time-dependent problems the solution exhibits a very rapid, exponential decay of energy in time. To detect the functions that still persist at a point of time and are thus relevant for approximation, we introduce as a key new contribution a compact transfer operator in time that maps arbitrary initial conditions to the solution of the PDE in a later point of time. Spanning the local space by the leading left singular vectors of the transfer operator results in an approximation space that is optimal in the sense of Kolmogorov \cite{Kol36} and hence minimizes the approximation error among all spaces of the same dimension. While there are many methods that exploit localization in space \cite{BabLip11,GraGre12,MaaPet14,MaaPer18,OwhZha11,OwhZha17,OwhZha14,SmePat16,EftPat13,HuyKne13,IapQua12,MadRon02}, we provide with this paper one of the first contributions that exploit localization in time \cite{ChuEf18,LjuMai21}. Whereas existing methods \cite{ChuEf18,LjuMai21} decompose the entire time interval into consecutive local subintervals, we select relevant time points in a data-dependent manner and only perform local computations on the corresponding local time intervals. Moreover, we provide for the first time an a priori error bound for the local approximation error in time.

As a direct calculation of the leading left singular vectors of the transfer operator can become computationally expensive, we employ random sampling as proposed in \cite{BuhSme18} for elliptic PDEs to efficiently approximate the optimal local spaces and facilitate an embarrassingly parallel construction of the reduced ansatz functions even for a single point of time. To this end, we solve the PDE locally in time with random initial conditions. We show that the resulting local space yields an approximation that converges at a quasi-optimal rate, allowing for local error control. While we only provide a local a priori error bound in this paper, we conjecture that it might be possible to also derive a global error bound.

Preliminary numerical experiments can be found in one of the authors bachelor thesis \cite{terMaat19}, where the proposed algorithm has been tested for problems with time-dependent source terms by uniformly sampling time points.

Optimal spatially local approximation spaces have been introduced for elliptic \cite{BabLip11,MaSch21,SmePat16} and parabolic \cite{SchSme20} problems, and random sampling has been employed to efficiently approximate the optimal local spaces in the elliptic setting in \cite{BuhSme18,KeLi20}. Further spatially localizable multiscale methods for parabolic problems have been proposed in \cite{ChuEf18,MaaPer18,OwhZha11,OwhZha17,OwhZha14}. In \cite{LjuMai21} a space-time multiscale method for the linear heat equation is introduced, where for each coarse space-time node a corrector function that is localized in both space and time is computed to capture (local) fine-scale features. While in \cite{LjuMai21} a global a posteriori error bound is proved assuming that certain localization parameters are large enough to guarantee sufficiently small localization errors, but no rigorous a priori error analysis is performed, we provide in this paper for the first time a local a priori error bound in time. Moreover, we refer to \cite{Betal20} for an overview of methods to construct local reduced spaces.

In system and control theory balanced truncation is a well-known method to reduce the complexity of input-output systems \cite{Moo81,Row05}. Balanced truncation for systems including time-dependent data functions has been introduced in \cite{ShoSil83,VerKai83} and its error analysis has been first studied in \cite{LalBec03,SanRan04}. Nevertheless, solving matrix differential equations or matrix inequalities is required which is prohibitively expensive for high dimensional problems. In \cite{ChaDoo05} the authors propose an iterative procedure that is computationally more appealing as it is based on operations which exploit sparsity of the model. However, the required computations are global in time and have to be carried out in a sequential manner, leading to a complexity that depends on the global time discretization. In contrast, the approach we propose here requires only local computations in time that are in addition parallelizable.

In \cite{UngGug19} it is shown that for linear time-invariant systems the concepts of Kolmogorov $n$-widths and Hankel singular values are directly connected and that the right singular vectors of the Hankel operator restricted to the unit ball span the optimal reduced input space in the sense of Kolmogorov and can be linked to active subspaces. We suggest that the leading right singular vectors of the transfer operator introduced here span the optimal reduced input space and can thus be used to regularize inverse problems and data assimilation procedures.

Furthermore, dynamic mode decomposition \cite{Sch10,JonCla14} fits an operator that maps the solution from one time point to the next to simulation data. The fitted operator is thus similar to the transfer operator we introduce in this paper. However, similar to the POD and in contrast to the approach we propose here, access to the (global) solution trajectory is required to compress the entire dynamics. Recently, probabilistic numerical methods that yield a probability distribution over the (unknown) solution of an ordinary or partial differential equation have been proposed, for instance, in \cite{SchDuv14,SchSar19,ConGir17,Owh15}. Moreover, randomized subset selection techniques are used in \cite{Sai20} for the purpose of hyperreduction. In comparison, in this paper data-dependent probability distributions and randomized subset selection techniques are exploited to select both start time points and initial conditions for the temporally local PDE simulations.

The remainder of this paper is organized as follows. In \cref{sec_problem_setting} we introduce the general time-dependent model problem together with an exemplary test case and its numerical approximation. Subsequently, we first sketch the key new contributions of this paper along with some motivation in \cref{subsection_motivation} and then develop the main contributions in \cref{sec_opt_spaces,sec_random_spaces}. We propose optimal local spaces in time in \cref{sec_opt_spaces} and address their approximation via random sampling in \cref{sec_random_spaces}. Moreover, in \cref{sec_random_spaces} we propose a randomized algorithm to construct one reduced space for the global approximation by solving several local problems in time in parallel. We discuss both its basic properties and the choice of the probability distribution used for drawing time points. Finally, we present numerical experiments in \cref{numerical_experiments} to demonstrate the approximation properties of the proposed algorithm and draw conclusions in \cref{conclusions}.

\section{Problem setting}
\label{sec_problem_setting}

We consider a time-dependent linear PDE, say, an advec-tion-diffusion-reaction problem, that may include heterogeneous time-dependent coefficients. To that end, let $D \subseteq \R^n$ denote a bounded Lipschitz domain of dimension $n\in\lb1,2,3\rb$ and let $I=(0,T)\subset \R$ be a time interval with $0<T<\infty$. We assume that on $I\times D$ a Bochner space $\mathcal{V}$ and a reflexive Bochner space $\mathcal{W}$ with dual space $\mathcal{W}^*$ are given that will serve as ansatz and test space, respectively. Furthermore, we denote by $\partial_t + \mathcal{A}: \mathcal{V} \rightarrow \mathcal{W}^*$ a linear, continuous, surjective, and inf-sup stable operator. Let $\mathcal{F}\in \mathcal{W}^*$ be a bounded linear functional. We assume that $\mathcal{F}$ accounts for both source terms and boundary data. Moreover, we denote the initial values by $u_0 \in L^2(D)$. Here, we assume that for any function $u\in \mathcal{V}$ we have that $u(t,\cdot)\in L^2(D)$ for every point in time $t\in I$. Then, we consider the following variational problem: Find $u\in \mathcal{V}$ such that $u(0,\cdot)=u_0$ in $L^2(D)$ and
\begin{align}\label{PDE}
\partial_t u + \mathcal{A} u = \mathcal{F} \quad \text{in } \mathcal{W}^*.
\end{align}
The Banach-Ne\v{c}as-Babu\v{s}ka theorem (e.g. \cite[Theorem 2.6]{Ern2004}) and the assumptions above guarantee the existence of a unique solution of \cref{PDE}.

\textbf{Exemplary model problem: advection-diffusion-reaction problem.} We consider an advection-diffusion-reaction problem as a representative model problem of \cref{PDE}. To this end, we assume that $\partial D = \Sigma_D \cup \Sigma_N$ with $\vert \Sigma_D\vert >0$ and denote by $f:I\times D\rightarrow \R$ source terms, $u_0: D\times\R$ initial conditions, $g_D:I\times\Sigma_D\rightarrow \R$ Dirichlet boundary conditions, $g_N:I\times\Sigma_N\rightarrow \R$ Neumann boundary conditions, and $n$ the outer unit normal, respectively. In its weak form the conductivity coefficient is assumed to satisfy $\kappa \in L^\infty(I\times D)^{n\times n}$ with $\kappa_0(t,x) \vert v \vert^2 \leq v^\top \kappa(t,x) v \leq \kappa_1(t,x) \vert v \vert^2$ for every $v\in\R^{n}$, $0<\kappa_0\leq\kappa_0(t,x)\leq\kappa_1(t,x)\leq\kappa_1<\infty$ for almost every $(t,x)\in I\times D$, where we do not consider the case $\kappa_0\rightarrow 0$ such that we can assume that a weak solution with spatial regularity of $H^1(D)$ exists. Furthermore, the advection field and the reaction coefficient are assumed to satisfy $b \in L^\infty(I\times D)^n$, $\nabla \cdot b \in L^\infty(I\times D)$, and $c\in L^\infty(I\times D)$ with $\vert b(t,x) \cdot v \vert^2 \leq b_1(t,x) \vert v \vert^2$ for every $v\in\R^{n}$, $0\leq b_1(t,x)\leq b_1<\infty$ for almost every $(t,x)\in I\times D$, and $\vert c(t,x)\vert\leq c_1 < \infty$ for almost every $(t,x)\in I\times D$, respectively. To ensure well-posedness, we moreover assume that $c(t,x) -\frac{1}{2}\nabla \cdot b(t,x) \geq 0$ for almost every $(t,x)\in I\times D$. We then seek the solution $u: I\times D \rightarrow\R$ such that
\begin{align}\label{PDE_ex}
\begin{split}
u_t(t,x)-\Div(\kappa(t,x)\nabla u(t,x)) + b(t,x) \cdot \nabla u(t,x) + c(t,x)\, u (t,x)=f(t,x)  \\ \text{ for every }(t,x) \in I\times D, \\
\begin{array}{rlll}
u(t,x)&=&g_D(t,x) \quad &\text{ for every } (t,x) \in I \times\Sigma_D,\\
\kappa(t,x) \nabla u(t,x) \cdot n(x) &=&g_N(t,x) \quad &\text{ for every } (t,x) \in I \times\Sigma_N,\\
u(0,x)&=&u_0(x) \quad &\text{ for every }x \in  D.
\end{array}
\end{split}
\end{align}

\textbf{Discretization.} Here, we consider $f\in L^2(I,L^2(D))$ and $g_N \in L^2(I,L^2(\Sigma_N))$. To simplify notation, we assume that $f$ also accounts for Dirichlet boundary conditions $g_D\in L^2(I,H^{1/2}(\Sigma_D))$. For the numerical approximation of \cref{PDE_ex}, we employ the implicit Euler method. To this end, we assume that the time interval $I$ is partitioned via $N_I$ equidistant time points $0=t_0<t_1<\ldots< t_{N_I-1}=T$ of distance $\Delta_T = T/(N_I-1)$. However, the approximation can be done completely analogously for other time stepping schemes. Moreover, we consider a piecewise linear conforming FE space $X$ of dimension $N_D\in \N$ with basis functions $\phi_i\in X$, $i=1,\ldots,N_D$. Then, the mass and stiffness matrices $\mathbf{M},\,\mathbf{A}_l\in \R^{N_D\times N_D}$ and the right-hand side vectors $\mathbf{F}_l\in \R^{N_D}$ are given as
\begin{align}\label{FE_matrices}
\begin{split}
\mathbf{M}_{ij} &:= (\phi_j,\phi_i)_{L^2(D)}, \quad 1\leq i,j \leq N_D,\\
(\mathbf{A}_l)_{ij} &:= (\kappa(t_l) \nabla \phi_j, \nabla \phi_i)_{L^2(D)} + (b(t_l) \cdot \nabla \phi_j, \phi_i)_{L^2(D)} + (c(t_l) \phi_j, \phi_i)_{L^2(D)},\\
&\mspace{314mu} 1\leq i,j \leq N_D,\;1\leq l \leq N_I-1,\\
(\mathbf{F}_l)_{i} &:= (f(t_l),\phi_i)_{L^2(D)} + (g_N(t_l), \phi_i)_{L^2(\Sigma_N)}, \quad 1 \leq i \leq N_D,\;1\leq l \leq N_I-1.
\end{split}
\end{align}
To simplify the presentation, we assume here that this discretization is stable, meaning that the advection should not be too dominant.
Given a discrete representation $\mathbf{u}_0\in \R^{N_D}$ of the initial values $u_0$, we approximate the solution of \cref{PDE_ex} by computing $\mathbf{u}_{l} \in \R^{N_D}$ at time point $t_l$ for $l=1,\ldots, N_I-1$ via
\begin{align} \label{eq_time_stepping}
(\mathbf{M}+\Delta_T \mathbf{A}_l)\,  \mathbf{u}_l = \Delta_T \mathbf{F}_l + \mathbf{M} \mathbf{u}_{l-1}.
\end{align}

In the following, we assume that the discretization in space and time is chosen sufficiently fine such that the discretization error between the exact solution and the discrete solution of \cref{eq_time_stepping} is negligibly small compared to the multiscale or model order reduction error. To reduce the possibly very high-dimensional discrete problem \cref{eq_time_stepping} (due to, for instance, fine-scale features in the coefficient functions that need to be well resolved) we apply multiscale or model order reduction methods.

\textbf{Reduced approximation.} We assume for now that suitable reduced ansatz functions $\varphi_1,\ldots,\varphi_N \in X$ are given that will be determined below (cf. \cref{sec_random_spaces}). The matrix $\mathbf{U}_\text{red} = [\boldsymbol{\varphi}_1 \ldots \boldsymbol{\varphi}_N]\in \R^{N_D\times N}$ contains the corresponding FE coefficient vectors. We then compute a reduced approximation of \cref{eq_time_stepping} via Galerkin projection of the FE space onto the space spanned by the reduced basis: For $l=1,\ldots,N_I-1$ find $\mathbf{u}_{\text{red},l}\in\R^N$ such that 
\begin{align} \label{eq_reduced_global_approximation}
\mathbf{u}_{\text{red},l} & = (\mathbf{M}_\text{red}+\Delta_T \mathbf{A}_{\text{red},l})^{-1} (\Delta_T \mathbf{F}_{\text{red},l} + \mathbf{M}_\text{red} \mathbf{u}_{\text{red},l-1}),
\end{align}
where $\mathbf{M}_\text{red} = \mathbf{U}_\text{red}^\top \mathbf{M} \mathbf{U}_\text{red}$, $\mathbf{A}_{\text{red},l} = \mathbf{U}_\text{red}^\top \mathbf{A}_{l} \mathbf{U}_\text{red}$, $\mathbf{F}_{\text{red},l} = \mathbf{U}_\text{red}^\top \mathbf{F}_{l}$, and $\mathbf{u}_{\text{red},0} = \mathbf{M}_\text{red}^{-1}  \mathbf{U}_\text{red}^\top \mathbf{M} \mathbf{u}_0$.

\section{Motivation and key new contributions of this paper}
\label{subsection_motivation}

To construct reduced ansatz functions $\varphi_1,\ldots,\varphi_N \in X$ as given above, a well-established strategy is to perform a POD on (the first part of) the global solution trajectory. For this purpose, prior to reducing, the global solution in time has to be computed in a sequential manner (cf. \cref{fig_transfer_operator} (top row)). In contrast, we propose in this paper, as a major new contribution, to generate a reduced basis in an embarrassingly time-parallel manner which enables to split the computational budget and distribute it over the entire time interval (cf. \cref{fig_transfer_operator} (middle row)). The approach is thus well-suited to be used on modern computer architectures allowing for many parallel computations. Moreover, we conjecture that the proposed ideas may in general contribute to reduce computational costs and exploit parallelization in applications where the computation of classical full order solutions is extremely expensive, for instance, by combining the proposed approach with the parareal methodology \cite{LioMad01}.

To motivate the proposed approximation strategy, we consider the toy model problem visualized in \cref{1D_motivation_plot} and observe that the time-dependent source terms $f_1$ and $f_2$ clearly determine the behavior of the corresponding solution in time. Therefore, we suggest that the time-dependent data functions may help to determine time points that are relevant for approximation. To this end, we represent the time-dependent data functions as matrices, where each column corresponds to one time point in the time grid, see \cref{1D_motivation_plot} (top right). To detect and select significant time points, we then employ column subset selection techniques from randomized NLA \cite{DerMah21,DriMah16}. These methods are generally used to construct low-rank matrix decompositions (cf., e.g., \cite{MahDri09}) by approximating (the range of) the matrix from selected columns or rows. The decompositions are thus more interpretable with respect to the original data compared to an SVD and error bounds are available (cf., e.g., \cite{DriMah16,MahDri09}). As one key contribution, we exploit column subset selection techniques in a completely new context for the purpose of time point selection.
	
Subsequently, we generate reduced ansatz functions corresponding to the selected time points in an embarrassingly parallel manner by solving independent local problems in time (cf. \cref{fig_transfer_operator} (middle row)). To motivate a localized construction of the reduced ansatz functions in time, we recall the following well-known property that many time-dependent problems share via the example of the linear heat equation: If $\mathcal{F} \equiv 0$ it is straightforward to show that $\|u( t,\cdot) \|_{L^{2}(D)} \leq e^{-Ct} \| u_0\|_{L^{2}(D)}$ for any $t \in I$, where the constant $C$ depends only on the shape of $D$ and the heat conductivity coefficient. To detect the functions that still persist at a selected time point $t$ and are thus relevant for approximation purposes, we introduce, as a key new contribution, a transfer operator $\mathcal{T}_{s \rightarrow t}$ in time that takes arbitrary initial conditions in $L^2(D)$ at time $s\in I$ with $s<t$, solves the PDE locally in $(s,t)$, and evaluates the solution at target time $t$; see \cref{fig_transfer_operator} (bottom row) for an illustration. In this way, the transfer operator captures the decay behavior of solutions of the PDE in time.

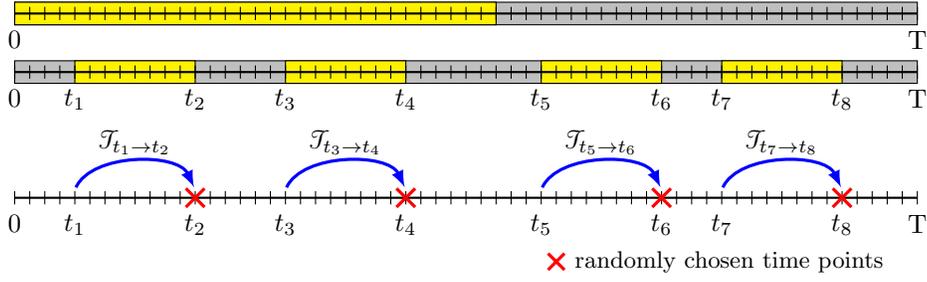
\begin{figure}
	\centering
	\begin{tikzpicture}
	\draw[fill=lightgray] (0,0.15) rectangle (12,-0.15);
	\draw[fill=yellow] (0,0.15) rectangle (6.4,-0.15);
	\draw[thick, -] (0,0) -- (12,0);
	\foreach \x in {0,0.2,...,12} \draw (\x cm, 2.5pt) -- (\x cm, -2.5pt); 
	\draw (0,0) node[below=3pt] {$0$};
	\draw (12,0) node[below=3pt] {T};
	\end{tikzpicture}
	\begin{tikzpicture}
	\draw[fill=lightgray] (0,0.15) rectangle (12,-0.15);
	\draw[fill=yellow] (0.8,0.15) rectangle (2.4,-0.15);
	\draw[fill=yellow] (3.6,0.15) rectangle (5.2,-0.15);
	\draw[fill=yellow] (7,0.15) rectangle (8.6,-0.15);
	\draw[fill=yellow] (9.4,0.15) rectangle (11,-0.15);	
	\draw[thick, -] (0,0) -- (12,0);
	\foreach \x in {0,0.2,...,12} \draw (\x cm, 2.5pt) -- (\x cm, -2.5pt); 
	\draw (0,0) node[below=3pt] {$0$};
	\draw (12,0) node[below=3pt] {T};
	\draw (0.8,0) node[below=3pt] {$t_1$};
	\draw (2.4,0) node[below=3pt] {$t_2$};
	\draw (3.6,0) node[below=3pt] {$t_3$};
	\draw (5.2,0) node[below=3pt] {$t_4$};
	\draw (7.0,0) node[below=3pt] {$t_5$};
	\draw (8.6,0) node[below=3pt] {$t_6$};
	\draw (9.4,0) node[below=3pt] {$t_7$};
	\draw (11,0) node[below=3pt] {$t_8$};
	\end{tikzpicture}
	\begin{tikzpicture}
	\draw[thick, -] (0,0) -- (12,0);
	\foreach \x in {0,0.2,...,12} \draw (\x cm, 2.5pt) -- (\x cm, -2.5pt); 
	\draw (0,0) node[below=3pt] {$0$};
	\draw (12,0) node[below=3pt] {T};
	\draw (0.8,0) node[below=3pt] {$t_1$};
	\draw (2.4,0) node[below=3pt] {$t_2$};
	\draw (3.6,0) node[below=3pt] {$t_3$};
	\draw (5.2,0) node[below=3pt] {$t_4$};
	\draw (7.0,0) node[below=3pt] {$t_5$};
	\draw (8.6,0) node[below=3pt] {$t_6$};
	\draw (9.4,0) node[below=3pt] {$t_7$};
	\draw (11,0) node[below=3pt] {$t_8$};
	\draw [arrow, bend angle = 65, bend left, very thick, blue] (0.75,0) to (2.45,0);
	\draw [arrow, bend angle = 65, bend left, very thick, blue] (3.55,0) to (5.25,0);
	\draw [arrow, bend angle = 65, bend left, very thick, blue] (6.95,0) to (8.65,0);
	\draw [arrow, bend angle = 65, bend left, very thick, blue] (9.35,0) to (11.05,0);
	\draw (1.6,0.45) node[above] {$\mathcal{T}_{t_1\rightarrow t_2} $}; 
	\draw (4.4,0.45) node[above] {$\mathcal{T}_{t_3\rightarrow t_4}$}; 
	\draw (7.8,0.45) node[above] {$\mathcal{T}_{t_5\rightarrow t_6}$};
	\draw (10.2,0.45) node[above] {$\mathcal{T}_{t_7\rightarrow t_8}$};
	\draw [red, very thick] plot [only marks,mark=x, mark options={scale=2.5}] coordinates{(2.4,0)};
	\draw [red, very thick] plot [only marks,mark=x, mark options={scale=2.5}] coordinates{(5.2,0)};
	\draw [red, very thick] plot [only marks,mark=x, mark options={scale=2.5}] coordinates{(8.6,0)};
	\draw [red, very thick] plot [only marks,mark=x, mark options={scale=2.5}] coordinates{(11,0)};
	\draw [red, very thick] plot [only marks,mark=x, mark options={scale=2.2}] coordinates{(7.2,-0.85)};
	\node[label=right:{\small randomly chosen time points}] at (7.2,-0.85) {};
	\end{tikzpicture}
	\caption{Computational budget in time: Sequential for POD (top row) vs. split and distributed for randomized approach proposed in this paper (middle row). Transfer operators corresponding to randomly chosen time points (bottom row).}
	\label{fig_transfer_operator}
\end{figure}

\begin{figure}
	\begin{tikzpicture}
	\begin{axis}[
	domain = 0:1, 
	xmin = 0, xmax = 1, ymin=0, ymax= 17, 
	samples = 100, 
	width=4.5cm, height= 3.4cm,
	xtick style = {white},
	ytick style = {white},
	label style={font=\footnotesize},
	tick label style={font=\scriptsize}
	]
	\addplot+[mark=none,thick, red] {-1500*(x-0.2)*(x-0.4)};
	\addplot+[mark=none,thick, blue] {-1000*(x-0.6)*(x-0.8)};
	\end{axis}
	\draw[->, thick] (0.1,-0.5) -- (2.9,-0.5);
	\node at (1.2,-0.5) [below right] {\footnotesize time};
	\node at (1.2,1.3) [above, red] {\footnotesize $f_1$};
	\node at (2.4,0.75) [above, blue] {\footnotesize $f_2$};
	\end{tikzpicture}
	\begin{tikzpicture}
	\node at (0,1.15) [above] {\large \bf{+}};
	\node at (0,0) {};
	\end{tikzpicture}
	\begin{tikzpicture}
	\draw[-] (0.1,1.5) -- (2.95,1.5);
	\draw[-, ultra thick, red] (0.8125,1.51) -- (1.24,1.51);
	\node at (1.025,1.51) [above, red] {\footnotesize $f_1$};
	\draw[-, ultra thick, blue] (1.81,1.51) -- (2.38,1.51);
	\node at (2.1,1.51) [above, blue] {\footnotesize $f_2$};
	\node at (-0.075,1.5) [below right] {\scriptsize 0\;\; 0.2\;\; 0.4\;\; 0.6\;\; 0.8\;\; 1};
	\draw[->, thick] (0.1,1) -- (2.95,1);
	\node at (1.2,1) [below right] {\footnotesize space};
	\node at (0,0) {};
	\end{tikzpicture}
	\hspace{0.02cm}
	\begin{tikzpicture}
	\node [single arrow,thick, draw,inner sep=0.04cm,minimum height=0.8cm, single arrow head extend=0.1cm,] at (0,1.4) {};
	\node at (0,0) {};
	\end{tikzpicture}
	\hspace{0.02cm}
	\begin{tikzpicture}
	\node (matrix) at (0,0)
	{\includegraphics[width=3cm]{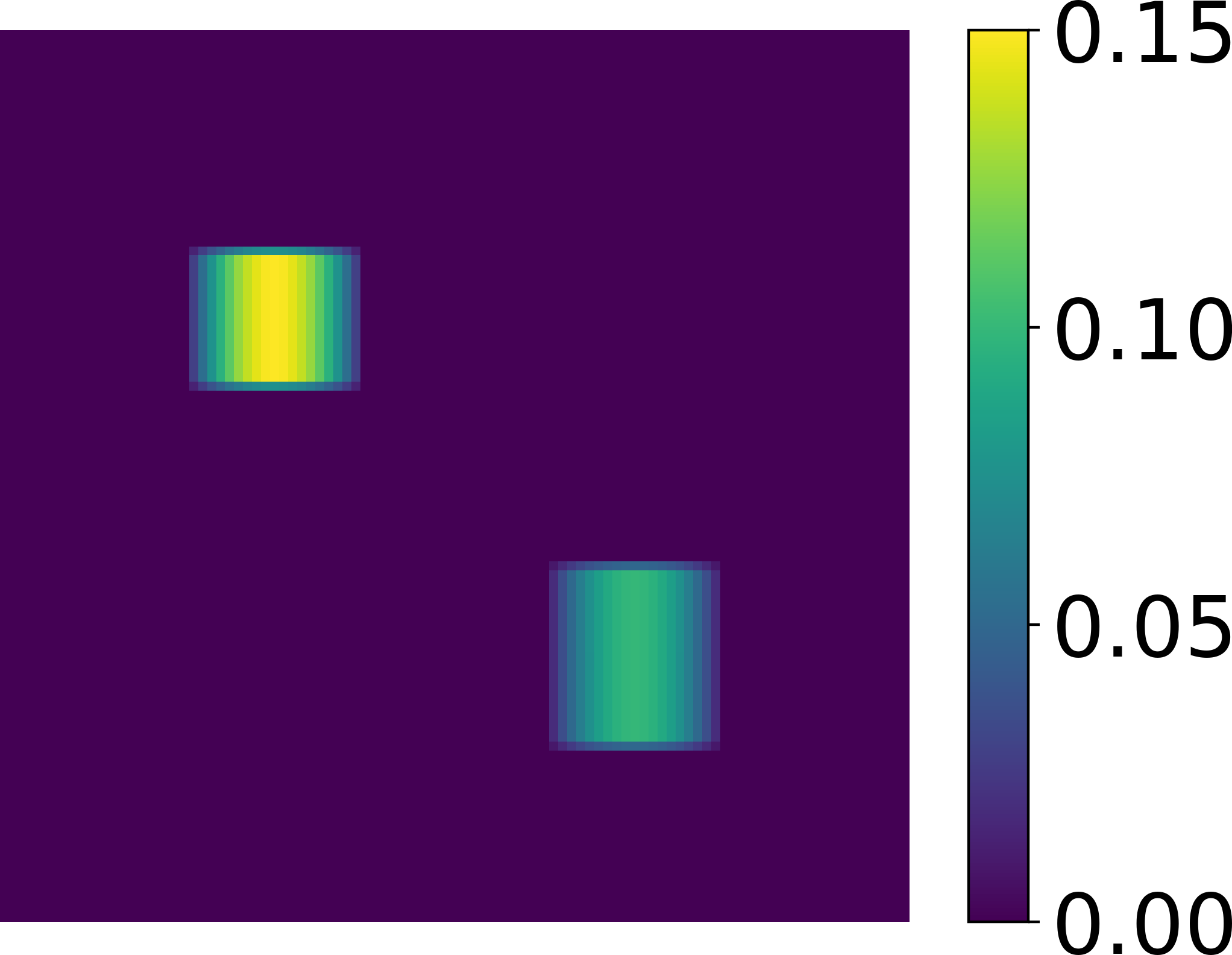}};
	\draw[->, thick] (-1.5,1.25) -- (0.7,1.25);
	\node at (-1.2,1.25) [above right] {\footnotesize time grid};
	\draw[->, thick] (-1.7,1.05) -- (-1.7,-1.05);
	\node at (-1.7,0.95) [above left, rotate=90] {\footnotesize spatial grid};
	\end{tikzpicture}
	\vspace{-0.1cm}\\
	\centering
	\begin{tikzpicture}
	\begin{axis}[
	width=4.8cm,
	height=3cm,
	xmin=0,
	xmax= 1,
	ymin=0,
	ymax=0.048,
	xtick style = {white},
	ytick style = {white},
	ytick = {0,0.02,0.04},
	yticklabel style = {white},
	yticklabels = {\textcolor{black}{0},\textcolor{black}{0.02},\textcolor{black}{0.04}},
	label style={font=\footnotesize},
	tick label style={font=\scriptsize}  
	]
	\addplot[solid, red, thick] table[x index=0, y index=1] {data_figures/1D_solutions_motivation.dat};
	\end{axis}
	\node [red, align = center] at (1.3,0.5) {\scriptsize solution\\ \scriptsize at $t=0.3$};
	\end{tikzpicture}
	\hspace{-0.5cm}
	\begin{tikzpicture}
	\begin{axis}[
	width=4.8cm,
	height=3cm,
	xmin=0,
	xmax= 1,
	ymin=0,
	ymax=0.048,
	xtick style = {white},
	ytick style = {white},
	yticklabel style = {white},
	label style={font=\footnotesize},
	tick label style={font=\scriptsize}  
	]
	\addplot[solid,line width = 2pt, cyan] table[x index=0, y index=2] {data_figures/1D_solutions_motivation.dat};
	\end{axis}
	\node [cyan, align = center] at (1.5,0.5) {\scriptsize solution\\ \scriptsize at $t=0.5$};
	\end{tikzpicture}
	\hspace{-0.5cm}
	\begin{tikzpicture}
	\begin{axis}[
	width=4.8cm,
	height=3cm,
	xmin=0,
	xmax= 1,
	ymin=0,
	ymax=0.048,
	xtick style = {white},
	ytick style = {white},
	yticklabel style = {white},
	label style={font=\footnotesize},
	tick label style={font=\scriptsize}  
	]
	\addplot[solid, blue, thick] table[x index=0, y index=3] {data_figures/1D_solutions_motivation.dat};
	\end{axis}
	\node [blue, align = center] at (2,0.4) {\scriptsize solution\\ \scriptsize at $t=0.7$};
	\end{tikzpicture}
	\caption{\footnotesize Data-dependent sampling based on $f$ helps to determine time points for detecting shapes of the solution. Right-hand side data functions $f_1$ and $f_2$ in time (top left) corresponding to spatially disjoint sources in one spatial dimension (top middle) and associated right-hand side data matrix (top right). Solution of heat equation for $f_1$, $f_2$, homogeneous initial and Dirichlet boundary conditions evaluated at different points in time (bottom).}
	\label{1D_motivation_plot}
\end{figure}
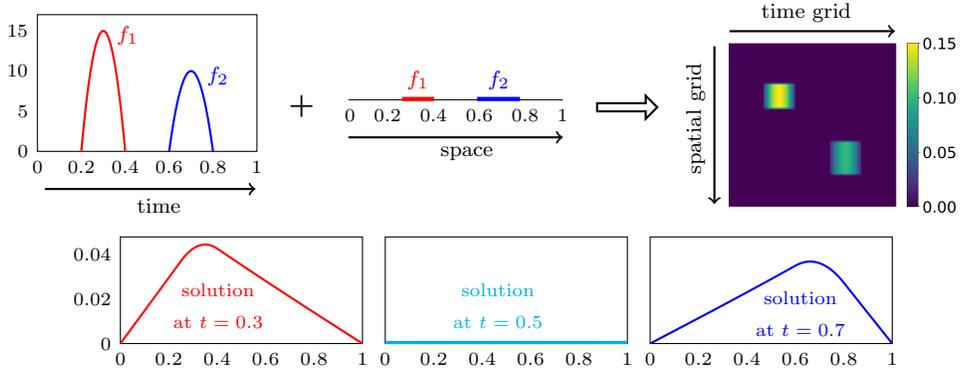

After discretization, say, with the FE method, the transfer operator can be represented by a matrix. According to the well-known Eckart-Young theorem the range of this matrix can be optimally approximated by its leading left singular vectors. While the \emph{discrete} transfer operator is trivially compact thanks to its finite rank, in the \emph{continuous} setting we need to prove compactness of the transfer operator to facilitate its singular value decomposition (SVD) via the Hilbert-Schmidt theorem  \cite[Theorem 8.94]{RR04}. Then, the space spanned by the leading left singular vectors provides an optimal approximation in the sense of Kolmogorov, meaning that it minimizes the approximation error among all linear spaces of the same dimension (see \cite[Theorem 2.2 in Chapter 4]{Pinkus85}). Moreover, in the elliptic setting it has been shown that the optimal local approximation spaces outperform other approaches (based on, for instance, Legendre-type functions or empirical modes) numerically (see \cite{SmePat16}).

To further facilitate an efficient parallel computation, we approximate the optimal spaces via random sampling \cite{BuhSme18,HaMaTr11}, i.e. we solve the PDE locally in time with random initial conditions, which results in a provably nearly optimal local approximation.

\section{Optimal local approximation spaces in time}
\label{sec_opt_spaces}

First, we introduce in \cref{subsection_optimal_spaces}, as a key new contribution, a transfer operator in time for which one can prove compactness and thus obtain local ansatz spaces in time, which are optimal in the sense of Kolmogorov for the approximation of the solution of $\cref{PDE}$ in a point of time.
Subsequently, we show how to compute an approximation of the optimal local spaces and discuss its practical realization via Krylov subspace methods and random sampling in \cref{matrix_representation}.

\subsection{Constructing optimal local approximation spaces via a transfer operator}
\label{subsection_optimal_spaces}

We first observe that for any local subinterval $(s,t)\subseteq I$ of the global time interval the solution $u$ of \cref{PDE} solves the PDE locally in time with (unknown) initial conditions given by $u(s,\cdot)\in L^2(D)$. Therefore, we consider all local solutions $u_\text{loc} \in \mathcal{V}\vert_{(s,t)\times  D}$ with arbitrary initial conditions $u_\text{loc}(s,\cdot) \in L^2(D)$ that satisfy
\begin{align}\label{local_PDE}
\partial_t u_\text{loc} + \mathcal{A}_\text{loc}\, u_\text{loc} = \mathcal{F}_\text{loc} \quad \text{in } (\mathcal{W}\vert_{(s,t)\times  D})^*.
\end{align}
Here, $\mathcal{A}_\text{loc}$ and $\mathcal{F}_\text{loc}$ denote the respective local operator and functional associated with $(s,t)\times D$.
All solutions of \cref{local_PDE} can be split into a function that solves \cref{local_PDE} for $\mathcal{F}_\text{loc}$ and homogeneous initial conditions and a function that solves \cref{local_PDE} for $\mathcal{F}_\text{loc}\equiv 0$ and arbitrary initial conditions. In the following, we will first address the case where $\mathcal{F}_\text{loc}\equiv 0$ and discuss the general case at the end of the subsection.

As we want to approximate the evaluation of the global solution at time point $t$, we consider the space $\HH_{t}$ of all solutions of \cref{local_PDE} evaluated at time $t$:
\begin{equation}\label{space_of_local_solutions}
\begin{split}
\HH_{t}:=\lb w(t,\cdot) \in L^2(D) \mid w \in \mathcal{V}\vert_{(s,t)\times  D} \text{ solves } \cref{local_PDE},\; w(s,\cdot) \in L^2(D),\; \mathcal{F}_\text{loc}\equiv 0 \,\rb.\hspace{-5pt}
\end{split}
\end{equation}
We equip the space $\HH_{t}$ with the $L^2(D)$-inner product and -norm. 

As one major contribution of this paper, we next introduce a transfer operator in time that can be proven to be compact and thus facilitates to derive optimal local ansatz spaces for the approximation of the solution space at time $t$ via its SVD (see \cref{fig_one_transfer_operator} for an illustration).\\

\vspace{-2pt}
\hspace{-15pt}
\begin{minipage}{0.6\textwidth}
		\begin{definition}
			For $s<t$, $s,t\in I$, the\\ transfer operator $\mathcal{T}_{s \rightarrow t}:L^2(D) \rightarrow \mathcal{H}_{t}$ is given by
			\begin{equation}\label{transfer_operator}
			\mathcal{T}_{s \rightarrow t}\, w(s,\cdot)=w(t,\cdot)
			\end{equation}
			 for $w\in \mathcal{V}\vert_{(s,t)\times  D}$ that solves \cref{local_PDE} for  $\mathcal{F}_\text{loc}\equiv 0$.
		\end{definition}
\end{minipage}
\hfill
\begin{minipage}{0.35\textwidth}
	\centering
	\begin{tikzpicture}
	\draw[thick, ->] (-0.2,0) -- (3.6,0);
	\foreach \x in {-0.2,0,...,3.4} \draw (\x cm, 2.5pt) -- (\x cm, -2.5pt); 
	\draw (0.8,0) node[below=3pt] {$s$};
	\draw (2.4,0) node[below=3pt] {$t$};
	\draw [arrow, bend angle = 65, bend left, very thick, blue] (0.75,0) to (2.45,0);
	\draw (1.6,0.5) node[above] {\large $\mathcal{T}_{s\rightarrow t} $}; 
	\end{tikzpicture}
	
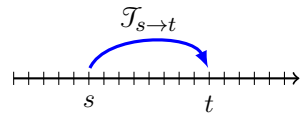
\captionof{figure}{Transfer operator $\mathcal{T}_{s\rightarrow t} $.}
	\label{fig_one_transfer_operator}
\end{minipage}
\vspace{2pt}

\begin{assumption} \label{assumption_compactness}
	We assume that the transfer operator $\mathcal{T}_{s \rightarrow t}$ introduced in \cref{transfer_operator} is compact for arbitrary $0\leq s < t \leq T$.
\end{assumption}

\begin{remark}\label{remark_compactness_ex}
		In \cref{supp_section_ex} we exemplarily prove compactness of the transfer operator for an advection-diffusion-reaction problem as introduced in \cref{sec_problem_setting}.
\end{remark}

Compactness of the transfer operator guarantees the existence of its SVD via the Hilbert-Schmidt theorem \cite[Theorem 8.94]{RR04}. Similar to \cite{BabLip11,SchSme20,SmePat16} it can then be shown that the leading left singular vectors of $\mathcal{T}_{s \rightarrow t}$ span an optimal approximation space in the local solution space $\mathcal{H}_{t}$. Here, we use the concept of optimality in the sense of Kolmogorov \cite{Kol36}: A subspace $\mathcal{H}^{n}_{t} \subset \mathcal{H}_{t}$ of dimension at most $n$ for which holds $
d_{n}(\mathcal{T}_{s \rightarrow t}( L^2(D));\mathcal{H}_{t}) =  \|\mathcal{T}_{s \rightarrow t} - \mathcal{P}_{\mathcal{H}^n_t}\mathcal{T}_{s \rightarrow t} \|$
is called an optimal subspace for $d_{n}(\mathcal{T}_{s \rightarrow t}( L^2(D));\mathcal{H}_{t})$, where the Kolmogorov $n$-width $d_{n}(\mathcal{T}_{s\rightarrow t}( L^2(D));\mathcal{H}_{t})$ is defined as $
d_{n}(\mathcal{T}_{s\rightarrow t}( L^2(D));\mathcal{H}_{t}) :=\inf_{\mathcal{H}_{t}^{n}\subset \mathcal{H}_{t};\;\dim(\mathcal{H}_{t}^{n})=n} \|\mathcal{T}_{s \rightarrow t} - \mathcal{P}_{\mathcal{H}^n_t}\mathcal{T}_{s \rightarrow t} \|$ and $\mathcal{P}_{\mathcal{H}^n_t}$ denotes the orthogonal projection onto ${\mathcal{H}^n_t}$.

\begin{theorem}[Optimal local approximation spaces in time]\label{optimal_spaces}
	Let $\sigma^{(i)}_{s\rightarrow t} \in \R^+$ and $\varphi^{(i)}_{s\rightarrow t} \in \HH_{t}$, $i=1,\ldots,\infty$, denote the singular values and left singular vectors of the transfer operator $\mathcal{T}_{s \rightarrow t}$ defined in \cref{transfer_operator}. Then the optimal approximation space for $d_n(\mathcal{T}_{s \rightarrow t} ( L^2(D));\HH_{t})$ is given by
	\begin{align}\label{eq_optimal_space}
	\HH^n_t:= \spann \lb \varphi^{(1)}_{s\rightarrow t},\ldots,\varphi^{(n)}_{s\rightarrow t}\rb
	\end{align}
	and the Kolmogorov n-width satisfies
	\begin{align*}
	&\;d_n(\mathcal{T}_{s \rightarrow t} ( L^2(D));\HH_{t})\\
	= &\sup_{\psi \in  L^2(D)} \inf_{\zeta \in \mathcal{H}_{t}^{n}} \frac{\|\mathcal{T}_{s \rightarrow t}\psi - \zeta\|_{L^2(D)}}{\|\psi\|_{L^2(D)}} = \|\mathcal{T}_{s \rightarrow t} - \mathcal{P}_{\mathcal{H}^n_t}\mathcal{T}_{s \rightarrow t} \| = \sigma^{(n+1)}_{s\rightarrow t}.
	\end{align*}
\end{theorem}

\begin{proof}
	The assertion directly follows from the Hilbert-Schmidt theorem \cite[Theorem 8.94]{RR04} and \cite[Theorem 2.2 in Chapter 4]{Pinkus85}.
\end{proof}

	\begin{remark}[Discussion of \cref{optimal_spaces}]
		While a compact transfer operator has already been used to construct optimal spatially local approximation spaces for elliptic \cite{BabLip11,MaSch21,SmePat16} and parabolic \cite{SchSme20} problems, the key new contribution in this paper is the introduction of a compact transfer operator in time that enables to generate an optimal ansatz space for the approximation of the solution space at a point of time. Moreover, we provide, to the best of our knowledge for the first time, a priori error analysis for the local approximation error in time (cf. also \cref{probabilistic_a_priori}). We conjecture that it is also possible to derive a global error bound by using ideas as employed in \cite{BonCohDeV17}.
	\end{remark}

To address non-homogeneous data $\mathcal{F}_\text{loc}$, we define $u^{f}_{s\rightarrow t}\in \mathcal{V}\vert_{(s,t)\times  D}$ as the solution of \cref{local_PDE} with homogeneous initial conditions at time $s$. Finally, the optimal local approximation space at time $t$ is given by
\begin{align}
\label{optimal_space_data}
\mathcal{H}_t^{n,\text{data}}:= \spann \lb \varphi^{(1)}_{s\rightarrow t},\ldots,\varphi^{(n)}_{s\rightarrow t}, u^{f}_{s\rightarrow t}(t,\cdot)\rb.
\end{align}

\subsection{Approximation of the optimal local approximation spaces in time}
\label{matrix_representation}
In this subsection we describe how to compute an approximation of the optimal local space $\mathcal{H}_t^{n,\text{data}}$ in \cref{optimal_space_data} for the example of the advection-diffusion-reaction problem (cf. \cref{sec_problem_setting}). In the following, we use the notation introduced in \cref{sec_problem_setting} (see e.g. \cref{FE_matrices}). Assuming that $s=t_i$ and $t=t_j$ for $0\leq i < j \leq N_I-1$ and given a discrete version $\mathbf{u}_{\text{loc},i}\in\R^{N_D}$ of the arbitrary local initial values $u_\text{loc}(s,\cdot)$, we compute a local solution $\mathbf{u}_{\text{loc},l}\in\R^{N_D}$ at time point $t_l$ for $l={i+1},\ldots,j$ via (cf. \cref{eq_time_stepping,local_PDE})
\begin{align} \label{eq_local_time_stepping}
\mathbf{u}_{\text{loc},l} & = (\mathbf{M}+\Delta_T \mathbf{A}_l)^{-1} (\Delta_T \mathbf{F}_l + \mathbf{M} \mathbf{u}_{\text{loc},l-1}).
\end{align}
As the discrete transfer operator $T_{t_i\rightarrow t_j}$ acts on the space of local solutions with $\mathbf{F}_l = 0$ ($i\leq l \leq j$), the matrix version $\mathbf{T}_{t_i\rightarrow t_j} \in \R^{N_D \times N_D}$ of $T_{t_i\rightarrow t_j}$ is given by (cf. \cref{transfer_operator})
\begin{align}\label{matrix_trans_op}
\mathbf{T}_{t_i \rightarrow t_j} \,\boldsymbol{\xi} = [(\mathbf{M}\mspace{-3mu}+\mspace{-3mu}\Delta_T \mathbf{A}_j)^{-1}  \mathbf{M}]\, [(\mathbf{M}\mspace{-3mu}+\mspace{-3mu}\Delta_T \mathbf{A}_{j-1})^{-1}  \mathbf{M}] \ldots [(\mathbf{M}\mspace{-3mu}+\mspace{-3mu}\Delta_T \mathbf{A}_{i+1})^{-1} \mathbf{M}] \,\boldsymbol{\xi}.\mspace{-10mu}
\end{align}
Finally, we compute the $n$ leading left singular vectors $\boldsymbol{\varphi}^{(1)}_{t_i \rightarrow t_j},\ldots,\boldsymbol{\varphi}^{(n)}_{t_i \rightarrow t_j} \in \R^{N_D}$ of $\mathbf{T}_{t_i \rightarrow t_j}$ to approximate the optimal local space $\mathcal{H}^n_{t_j}$ (cf. \cref{eq_optimal_space}) and define
\begin{align}
\label{optimal_discrete_space}
H^n_{t_j} := \spann \lb {\varphi}^{(1)}_{t_i \rightarrow t_j},\ldots,{\varphi}^{(n)}_{t_i \rightarrow t_j} \rb,
\end{align}
where ${\varphi}^{(k)}_{t_i \rightarrow t_j}$ is the FE function corresponding to the coefficient vector $\boldsymbol{\varphi}^{(k)}_{t_i \rightarrow t_j}$ for $1\leq k \leq n$. Consequently, we have that $\Vert T_{t_i\rightarrow t_j} - \proj_{H^n_{t_j}} T_{t_i\rightarrow t_j}\Vert = \sigma^{(n+1)}_{t_i \rightarrow t_j}$ (Eckart-Young theorem e.g. in \cite{GV13}), where $\sigma^{(n+1)}_{t_i \rightarrow t_j}$ is the $n+1$-st singular value of $T_{t_i\rightarrow t_j}$ (listed in non-increasing order of magnitude) and $\proj_{H^n_{t_j}}$ denotes the orthogonal projection onto $H^n_{t_j}$ (cf. \cref{optimal_spaces}). We use the same notation for continuous and discrete singular values and vectors expecting that the respective meaning is clear from the context.

To address non-homogeneous data $\mathbf{F}_l$ ($l={i+1},\ldots,j$), we compute the solution of \cref{eq_local_time_stepping} for homogeneous initial conditions $\mathbf{u}^\mathbf{F}_{\text{loc},i}\equiv 0$, add the resulting solution $\mathbf{u}^\mathbf{F}_{\text{loc},j}$ at time $t_j$ to the FE basis, and define $H^{n,\text{data}}_{t_j}$ as the span of the FE functions associated with the coefficient vectors $\boldsymbol{\varphi}^{(1)}_{t_i \rightarrow t_j},\ldots,\boldsymbol{\varphi}^{(n)}_{t_i \rightarrow t_j}, \mathbf{u}^\mathbf{F}_{\text{loc},j}$.

\begin{remark}[Comparison of computational approaches to approximate the optimal local spaces]
	\label{remark_comparison_comput_appr}
	In practice, the left singular vectors of $\mathbf{T}_{t_i \rightarrow t_j}$ can be computed via the eigenvectors of $\mathbf{T}_{t_i \rightarrow t_j}^\top \mathbf{T}_{t_i \rightarrow t_j}$.
	A direct computation of the optimal local space would therefore require to compute the evaluation of local solutions at time $t_j$ for all $N_D$ basis functions that span the local solution space at time $t_i$, thus evaluate the transfer operator $N_D$ times, and solve a dense generalized eigenproblem of dimension $N_D \times N_D$. As this becomes infeasible for large $N_D$, one would in general use Krylov subspace or randomized methods for the approximation of the optimal local spaces \cite{BuhSme18,HaMaTr11,MaRoTy11}.
	
	In Krylov subspace methods, the application of the transfer operator \cref{matrix_trans_op} would be implicitly passed to the eigenvalue solver. To calculate the $m$ leading eigenvectors of $\mathbf{T}_{t_i \rightarrow t_j}^\top \mathbf{T}_{t_i \rightarrow t_j}$ using, for instance, the implicitly restarted Arnoldi method (IRAM) from \cite{Lehoucq1998}, $\mathcal{O}(m)$ evaluations of $\mathbf{T}_{t_i \rightarrow t_j}$ and $\mathbf{T}_{t_i \rightarrow t_j}^\top$ are required in every iteration. While Krylov subspace methods can lead to more accurate approximations especially for slowly decaying singular values, randomized methods have the main advantage that they are inherently stable and amenable to parallelization \cite{HaMaTr11,Woo14}.
	
	To approximate the space spanned by the $m$ leading left singular vectors of the transfer operator via random sampling as described in \cref{subsec_quasi-optimal_spaces} in more detail, $m+s$ evaluations of the transfer operator are required. As randomized methods can outperform Krylov subspace methods even in the sequential setting (see, e.g., \cite{BuhSme18}), they are thus an appealing choice for the approximation of the optimal local spaces. For a more in-depth comparison of Krylov subspace and randomized methods, we refer, for instance, to \cite[section 6]{HaMaTr11}.
\end{remark}

\begin{remark}[Computational complexity] \label{remark_local_comp_complexity}
	The computational complexity of the local basis construction is clearly dominated by the evaluation of $\mathbf{T}_{t_i \rightarrow t_j}$ or $\mathbf{T}_{t_i \rightarrow t_j}^\top$ and thus the numerical solution of the local PDE, where we employ a sparse direct solver. For a standard FE discretization in two or three dimensions the factorizations of $(\mathbf{M} + \Delta_T \mathbf{A}_{i+1}),\ldots, (\mathbf{M} + \Delta_T \mathbf{A}_{j})\in\R^{N_D\times N_D}$ can be computed in $\mathcal{O}((j-i) N_D^{3/2})$ or $\mathcal{O}((j-i) N_D^{2})$ work \cite{GV13}. After factorizing, the computational complexity for each local solution trajectory of the PDE is $\mathcal{O}((j-i) N_D \log(N_D))$ or $\mathcal{O}((j-i) N_D^{4/3})$ \cite{GV13}. 
\end{remark}

\section{Generating reduced ansatz functions in parallel in time} \label{sec_random_spaces}

In this section we derive a randomized algorithm that provides an approximation to the \emph{discrete} solution of problem \cref{PDE} by exploiting techniques from randomized NLA \cite{BuhSme18,DrMaMu08,FrKaVe04,DerMah21,DriMah16,HaMaTr11} (see also motivation in \cref{subsection_motivation}). For this purpose, we first sketch in \cref{subsec_quasi-optimal_spaces} how a nearly optimal ansatz space for the approximation of the discrete solution in a single point of time can be generated in an efficient and parallel manner via random sampling. Subsequently, we present in \cref{subsec_algorithm} a randomized algorithm that constructs a reduced basis by performing several simulations of the PDE for only few time steps in parallel. The local simulations start at different time points that are randomly drawn from a data-dependent probability distribution whose choice is discussed in \cref{subsec_choice_prob}; random initial conditions are prescribed. The proposed algorithm thus enables to split and distribute the available computational budget over the entire time interval, facilitates an embarrassingly parallel computation, and is therefore well-suited to be used on modern computer architectures.

\subsection{Approximating the range of one transfer operator via random sampling}\label{subsec_quasi-optimal_spaces}

To motivate the random sampling strategy (cf. \cite{HaMaTr11}), suppose that we want to approximate the range of a large matrix of rank $m$. By multiplying the matrix with $m$ random vectors we draw $m$ random samples from the range of the matrix at high probability. As these $m$ samples are likely linearly independent thanks to the randomness, they span the range of the matrix with high probability. In cases where the rank of the matrix is unknown and to compensate for the fact that any of the drawn random vectors might lie in the null space of the matrix, we draw a small number $s$ of additional random vectors to ensure that the resulting $m+s$ samples most likely span the targeted $m$-dimensional subspace.

To construct a suitable approximation of the \emph{discrete} optimal local ansatz space $H^m_{t_j}$ \cref{optimal_discrete_space}, we thus prescribe $n=m+s$ random initial conditions at the local starting time point $t_i<t_j$, where the coefficient vectors of the corresponding FE functions are mutually independent normal random vectors $\mathbf{r}_1,\ldots,\mathbf{r}_n \sim N(\boldsymbol{0},(\mathbf{A}_{i}^\top \mathbf{A}_{i})^{-1})$. Recall that $\mathbf{A}_{i}$ denotes the stiffness matrix corresponding to time point $t_i$ \cref{FE_matrices}. Moreover, the oversampling parameter $s$ is typically not greater than $2$ or $3$ (cf. \cite{HaMaTr11,BuhSme18}).

Then, the matrix version of the transfer operator $\mathbf{T}_{t_i \rightarrow t_j}$ \cref{matrix_trans_op} is applied to the random vectors, meaning we solve the PDE locally on the time interval $(t_i,t_j)$ with initial conditions given by $\mathbf{r}_1,\ldots,\mathbf{r}_n$ and evaluate the solutions at the local end time $t_j$. We highlight that the computation of the $n$ local solutions is embarrassingly parallel. The space $H^n_{t_j,\text{rand}}$ is then spanned by the $n$ resulting local solutions evaluated at time $t_j$.

The following probabilistic a priori error bound shows that the space $H^n_{t_j,\text{rand}}$ yields an approximation that converges at a nearly optimal rate which is only slightly worse than the rate $\sigma^{(m+1)}_{t_i \rightarrow t_j}$ achieved by the optimal space $H^m_{t_j}$ (cf. \cref{matrix_representation}).

\begin{proposition}\label{probabilistic_a_priori} \textnormal{(Probabilistic a priori error bound, \cite[Proposition 3.2]{BuhSme18} based on \cite[Theorem 10.6]{HaMaTr11})}. 
	Let $\lambda_{min}^\mathbf{M}$ and $\lambda_{max}^\mathbf{M}$ denote the smallest and largest eigenvalue of $\mathbf{M}$ (cf. \cref{FE_matrices}). Moreover, we denote by $\sigma_{min}^{\mathbf{A}_i}$ and $\sigma_{max}^{\mathbf{A}_i}$ the smallest and largest singular value of $\mathbf{A}_i$. Then, for $n\geq 4$ it holds that
	\begin{align}\label{random_a_priori_result}
	\begin{split}
	&\mathbb{E}(\Vert T_{t_i\rightarrow t_j} - \proj_{H^n_{t_j,\text{rand}}} T_{t_i\rightarrow t_j}\Vert)\\ &\leq \frac{\sigma_{max}^{\mathbf{A}_i}\,\lambda_{max}^\mathbf{M}}{\sigma_{min}^{\mathbf{A}_i}\,\lambda_{min}^\mathbf{M}}  \min_{\substack{m+s=n\\ m\geq2,s\geq2}} \bigg[ \bigg(1 + \bigg(\frac{m}{s-1}\bigg)^{\mspace{-5mu}1/2}\,\bigg)\,\sigma^{(m+1)}_{t_i\rightarrow t_j} + \frac{e\sqrt{n}}{s} \bigg( \sum_{l>m} (\sigma^{(l)}_{t_i\rightarrow t_j})^2\bigg)^{\mspace{-5mu}1/2} \,\bigg].\hspace*{-0.2cm}
	\end{split}
	\end{align}
\end{proposition}

\begin{proof}
	By applying the Courant minimax principle the result follows directly from \cite[Proposition $3.2$]{BuhSme18}, which is based on \cite[Theorem 10.6]{HaMaTr11}. 
\end{proof}

As we can hope that the transfer operator has fast decaying singular values, the square root of the sum of squared singular values in the last term of \cref{random_a_priori_result} behaves often roughly as $\sigma^{(m+1)}_{t_i\rightarrow t_j}$. Hence, the error bound in  \cref{random_a_priori_result} decays for increasing $n$ if the singular values decay faster than $n^{-1/2}$. This is a valid assumption as the singular values often decay exponentially as can be seen in the numerical experiments (see, for instance, \cref{figure_stove_transfer_singular_vals}). 
For further details we refer to \cite{BuhSme18} where methods from randomized linear algebra \cite{HaMaTr11} have been used to approximate the optimal local approximation spaces in the elliptic setting.

\subsection{Randomized reduced basis generation algorithm}
\label{subsec_algorithm}

\cref{randomized_basis_generation} summarizes the embarrassingly parallel randomized basis generation. To provide a good approximation of the discrete solution of problem \cref{PDE}, we propose to randomly choose $\nrand\in \N$ time points in $\lb t_0,\ldots, t_{N_I-1}\rb$ according to the probability distribution $p$ that may be based on the time-dependent data functions of the PDE (see \cref{subsec_choice_prob} for details on the choice of $p$). Recall that the time interval $I$ is discretized via $N_I\in\N$ time points $0=t_0\leq ... \leq t_{N_I-1}=T$. To construct suitable ansatz functions in the chosen time points, we then apply the corresponding transfer operators to random initial conditions (cf. \cref{subsec_quasi-optimal_spaces}) as illustrated in \cref{fig_multiple_operators}. 

In detail, for each starting time point $t_i$, we draw in line \ref{draw_random_vector} a Gaussian random vector $\boldsymbol{\mathrm{rand\_{u_0}}}\sim N(\mathbf{0},(\mathbf{A}_i^\top \mathbf{A}_i)^{-1})$, where $\mathbf{A}_i$ is the stiffness matrix introduced in \cref{FE_matrices}. In line \ref{compute_snapshots}, we then compute the local solution of the PDE with initial condition $\boldsymbol{\mathrm{rand\_u_0}}$ for the respective local time interval employing the time stepping scheme \texttt{t\_steps} that is suitably chosen by the user.

Subsequently, we add the resulting local solution trajectories evaluated at the last $\nt-k+1$ $(k\leq \nt)$ time points to the snapshot matrix $\mathbf{S}$ in line \ref{collect_snapshots}. By choosing the parameter $k$ smaller than $\nt$, we sample from the ranges of multiple transfer operators simultaneously (see \cref{fig_multiple_operators} for an illustration), which often results in an improved approximation accuracy as the numerical experiments in \cref{numerical_experiments} show. For some guidance on the choice of $k$ and $\nt$ we refer to \cref{guidance_parameters}. Moreover, we highlight that the computations in lines \ref{loop} to \ref{add_initials} are embarrassingly parallel.

As the local PDEs are solved for $\nt (\ll N_{I})$ time steps, we discard any chosen time point smaller than or equal to $t_{\nt}$ in line \ref{remove_entry} and add the evolution of the initial conditions $\mathbf{u_0}$ for the first $\nt$ time steps to the snapshot matrix in line \ref{collect_snapshots_u0}.

\begin{algorithm2e}[t]
	\label[algo]{randomized_basis_generation}
	\DontPrintSemicolon
	\SetAlgoVlined
	\caption{Reduced basis generation via random sampling and SVD}
	\SetKwFunction{KwFn}{RandomizedReducedBasisGeneration}
	\SetKwProg{Fn}{Function}{:}{}
	\Fn{\KwFn{\textup{$\nrand$,$\nt$,$k$,$p$,\tol}}}{
		\SetKwInOut{KwIn}{Input}
		\SetKwInOut{KwOut}{Output} 
		\KwIn{number of randomly chosen time points $\nrand$, number of time steps for local PDE simulations $\nt$, starting time step for collecting snapshots $k\; (\leq \nt)$, (data-dependent) probability distribution $p$, time stepping scheme \texttt{t\_steps}, tolerance \tol}
		\KwOut{reduced basis $\mathbf{U}_\text{red}$ chosen according to tolerance \tol}
		$\mathrm{rand\_ints} \leftarrow $ draw $\nrand$ integers in $\lb 0,\ldots, N_I-1\rb$ according to $p$ \label{draw_ints} \;
		\If{$\mathrm{rand\_ints[i]} \leq \nt$ \label{check_size}}{ remove entry $\mathrm{rand\_ints[i]}$ from rand\_ints \label{remove_entry}
		}
		$\mathrm{endpoints} \leftarrow \mathrm{timegrid[rand\_ints]}$, \label{end_points} $\mathrm{startpoints} \leftarrow \mathrm{timegrid[rand\_ints-\nt]} $ \label{start_points}\;
		\tcp{initialize snapshot matrix}
		$\mathbf{S} \leftarrow \emptyset$ \label{empty_snapshots}\;
		\For{$i = 1,\ldots,\#\mathrm{startpoints}$\label{loop}}{		
			\tcp{draw random initial condition}
			$\boldsymbol{\mathrm{rand\_u_0}} \leftarrow \mathrm{Gaussian(size = N_D)}$ \label{draw_random_vector}\;
			\tcp{solve locally and store solution at time steps $k$ to $\nt$}
			$\boldsymbol{\mathrm{snapshots}} \leftarrow\hspace{-1pt} \mathrm{\texttt{t\_steps}(\mathrm{startpoints}[i],\hspace{-0.5pt} \mathrm{endpoints}[i],\hspace{-0.5pt} \nt,\hspace{-0.5pt}\boldsymbol{\mathrm{rand\_u_0})}}[:,k\hspace{-1pt}:\hspace{-1pt}\nt]$\label{compute_snapshots}\;
			$\mathbf{S} \leftarrow [\mathbf{S},\boldsymbol{\mathrm{snapshots}}]$\label{collect_snapshots}\;
		}
		\tcp{add representation of $\mathbf{u_0}$ for first $\nt$ time steps}
		$\boldsymbol{\mathrm{snapshots\_u_0}} \leftarrow \mathrm{\texttt{t\_steps}(0,\mathrm{timegrid[\nt]},\nt,\mathbf{u_0})}$ \label{add_initials}\;
		$\mathbf{S} \leftarrow [\mathbf{S},\boldsymbol{\mathrm{snapshots\_u_0}}]$\label{collect_snapshots_u0}\;
		\tcp{compute SVD of collected snapshots and cut using $\tol$}
		$\mathbf{U}_\text{red},\_\,, \_ \leftarrow \mathrm{svd}(\mathbf{S}, \tol)$\label{SVD_cut}\;
	}
\end{algorithm2e}

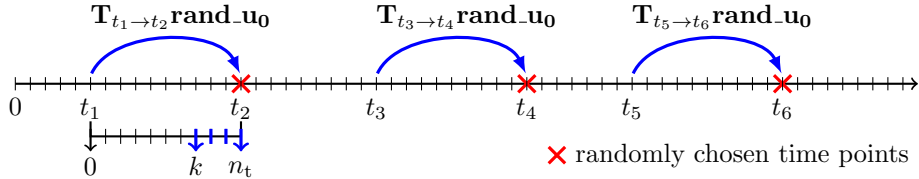
\begin{figure}
	\centering
	\begin{tikzpicture}
	\draw[thick, ->] (0,0) -- (12,0);
	\draw [arrow, bend angle = 65, bend left, very thick, blue] (0.95,0) to (3.05,0);
	\draw [arrow, bend angle = 65, bend left, very thick, blue] (4.75,0) to (6.85,0);
	\draw [arrow, bend angle = 65, bend left, very thick, blue] (8.15,0) to (10.25,0);
	\draw [red, very thick] plot [only marks,mark=x, mark options={scale=2.3}] coordinates{(3,0)};
	\draw [red, very thick] plot [only marks,mark=x, mark options={scale=2.3}] coordinates{(6.8,0)};
	\draw [red, very thick] plot [only marks,mark=x, mark options={scale=2.3}] coordinates{(10.2,0)};
	\draw (0,0) node[below=2pt] {$0$};
	\draw (1,0) node[below=2pt] {$t_1$};
	\draw (3,0) node[below=2pt] {$t_2$};
	\draw (4.8,0) node[below=2pt] {$t_3$};
	\draw (6.8,0) node[below=2pt] {$t_4$};
	\draw (8.2,0) node[below=2pt] {$t_5$};
	\draw (10.2,0) node[below=2pt] {$t_6$};
	\draw (2.2,0.6) node[above] {$\mathbf{T}_{t_1\rightarrow t_2} \boldsymbol{\mathrm{rand\_u_0}}$}; 
	\draw (6,0.6) node[above] {$\mathbf{T}_{t_3\rightarrow t_4} \boldsymbol{\mathrm{rand\_u_0}}$}; 
	\draw (9.4,0.6) node[above] {$\mathbf{T}_{t_5\rightarrow t_6} \boldsymbol{\mathrm{rand\_u_0}}$};
	\foreach \x in {0,0.2,...,11.8} \draw (\x cm, 2.5pt) -- (\x cm, -2.5pt); 
	\draw [thick] (1,-0.5) -- (1,-0.7) -- (3,-0.7) -- (3,-0.5);
	\foreach \x in {1,1.2,...,3.0} \draw (\x cm, -0.6) -- (\x cm, -0.8); 
	\draw[very thick, blue] (2.6, -0.6) -- (2.6, -0.8); 
	\draw[very thick, blue] (2.8, -0.6) -- (2.8, -0.8);
	\draw[thick, ->] (1,-0.6) -- (1,-0.9);
	\draw[very thick, ->,blue] (3,-0.6) -- (3,-0.9);
	\draw[very thick, ->, blue] (2.4,-0.6) -- (2.4,-0.9);
	\draw (1,-1.3) node[above=-1pt] {$0$};
	\draw (2.4,-1.3) node[above=-1pt] {$k$};
	\draw (3,-1.3) node[above=-2pt] {$\nt$};
	\draw [red, very thick] plot [only marks,mark=x, mark options={scale=2.2}] coordinates{(7.2,-0.95)};
	\node[label=right:randomly chosen time points] at (7.2,-0.95) {};
	\end{tikzpicture}
	\caption{Evaluation of multiple transfer operators at randomly chosen points in time for random initial conditions. Sketch of parameters $k$ and $\nt$ in \cref{randomized_basis_generation}.}
	\label{fig_multiple_operators}
\end{figure}

Finally, we compress all collected snapshots via an SVD in line \ref{SVD_cut} to construct the  reduced approximation space.

\begin{remark}[Choice of $\nrand$] \label{remark_adaptive_algorithm}
	In \cref{randomized_basis_generation} the user determines the number of random initial conditions $\nrand$ a priori, for instance, based on the knowledge that the employed computer architecture provides $\nrand$ parallel compute units. While we compress the snapshot matrix via an SVD, which is advantageous in cases where no error estimator is available or is very costly to evaluate, we conjecture that it might also be possible to develop an adaptive randomized algorithm that adaptively augments the reduced basis relying on a probabilistic a posteriori error estimator.
\end{remark}

\begin{remark}[Number of random initial conditions per time point]\label{remark_number_random_initials}
	In lines \ref{draw_random_vector}-\ref{compute_snapshots} of \cref{randomized_basis_generation} we propose to draw one random initial condition and thus compute one local basis function per chosen time point. In light of the probabilistic a priori error bound \cref{random_a_priori_result}, we conjecture that for transfer operators with fast decaying singular values this is sufficient in order to obtain a good approximation accuracy and we also observe this in the numerical experiments at least for the considered test cases.
	
	However, for transfer operators with more slowly decaying singular values, we suggest to draw multiple random initial conditions per chosen time point to enhance the quality of the approximation. In that case, the computation in line \ref{compute_snapshots} has to be split into (multiply) solving for random initial conditions and homogeneous right-hand side and (once) solving for homogeneous initial conditions and right-hand side. 
\end{remark}

\begin{remark}[Computational complexity of \cref{randomized_basis_generation}]
	Assuming that we employ the discretization introduced in \cref{sec_problem_setting,matrix_representation}, the complexity for computing the $\nrand +1$ solution trajectories in lines \ref{compute_snapshots} and \ref{add_initials} of \cref{randomized_basis_generation} is $\mathcal{O}((\nrand +1) \,\nt\, (N_D^{3/2}+N_D\log(N_D)))$ in two or $\mathcal{O}((\nrand +1)\, \nt\, (N_D^{2}+N_D^{4/3}))$ in three spatial dimensions (cf. \cref{remark_local_comp_complexity}). We highlight that the computations in lines \ref{compute_snapshots} and \ref{add_initials} are embarrassingly parallel as the local solution trajectories can be computed completely independently from each other. \cref{randomized_basis_generation} is thus well-suited to be used on modern computer architectures allowing for many parallel computations.
	
	Moreover, the computational complexity for compressing the collected snapshots via an SVD in line \ref{SVD_cut} is $\mathcal{O}(((\nt-k+1)\nrand + \nt) N_D^2)$. Alternatively, we can approximate the SVD of $\mathbf{S}$ via a randomized SVD \cite{HaMaTr11}. For details on the computational complexity of assembling the probability distribution $p$ we refer to \cref{subsec_choice_prob}.
\end{remark}

\subsection{Choice of probability distribution}
\label{subsec_choice_prob}

In this subsection we discuss how to choose the probability distribution $p$ that is used in \cref{randomized_basis_generation} for sampling points in time. As the behavior of the solution in time is influenced by the (time-dependent) data functions of the PDE, we guide the time point selection by the behavior of the data functions in time. To this end, we represent the time-dependent data functions as matrices, where each column corresponds to one time point and employ column subset selection techniques from randomized NLA \cite{DerMah21,DriMah16}. In detail, we focus on the following standard sampling strategies: uniform, squared norm, and leverage score sampling. While these methods are often used to construct low-rank matrix decompositions (cf., e.g., \cite{MahDri09}), we exploit them, as a key new contribution, for the purpose of time point selection.

In the following, we assume that the matrix $\mathbf{B}\in \R^{N_D \times N_I}$ represents the time-dependent data encoded in $\mathcal{A}$ or $\mathcal{F}$ (for instance, a coefficient function or source terms and boundary data that vary in time). In case of a time-dependent coefficient function, one option is to choose the entries of $\mathbf{B}$ equal to the value of the function in the respective space-time nodes. For a time-dependent $\mathcal{F}$, one could set the columns of $\mathbf{B}$ equal to the right-hand side vectors $\mathbf{F}_l$ for $0 \leq l \leq N_I-1$ (cf. \cref{FE_matrices}).

Next, we present the probability distributions (see, e.g., \cite{DerMah21,DriMah16} for an overview).

\textbf{Uniform sampling.} In the case of uniform sampling, we sample a point of time $t_i$  with probability $p_i = 1/N_I$ for $i=0,\ldots, N_I-1$, where $N_I$ denotes the number of time points determined by the partition $0=t_0\leq ... \leq t_{N_I-1}=T$ of the time interval $I$. The computational costs for constructing the probability distribution are zero.

\textbf{Squared norm sampling.} If we employ squared norm sampling \cite{FrKaVe04}, we select a point of time $t_i \in \lb t_0, \ldots ,t_{N_I-1} \rb $ with probability $p_i = \Vert \mathbf{B}[:,i] \Vert_2^2 / \Vert \mathbf{B} \Vert_F^2$, where $\Vert \cdot \Vert_F^2$ denotes the Frobenius norm and $ \mathbf{B}[:,i]$ is the $i$-th column of $\mathbf{B}$. The computation of the probabilities $p_i$ can be carried out in parallel and its complexity scales linearly in the number of non-zero entries of $\mathbf{B}$, e.g. for a dense matrix $\mathbf{B}$ the computational complexity is $\mathcal{O}(N_I N_D)$. 

\textbf{Leverage score sampling.} The leverage score sampling approach \cite{DrMaMu08} captures the statistical leverage of the columns of $\mathbf{B}$ on its best rank-$r$ approximation and preferably chooses columns which have a large influence on the best rank-$r$ fit of $\mathbf{B}$ \cite{MahDri09}. To this end, one computes the leading $r$ right singular vectors $\mathbf{v}_1,\ldots,\mathbf{v}_r$ of $\mathbf{B}$ and selects a column $i$ of $\mathbf{B}$ (a time point $t_i \in \lb t_0, \ldots ,t_{N_I-1} \rb $) with probability $p_i = 1/r \sum_{j=1}^{r} \mathbf{v}_j[i]^2$, where $\mathbf{v}_j[i]$ is the $i$-th entry of $\mathbf{v}_j$. The complexity for computing the SVD of $\mathbf{B}$ is $\mathcal{O}(r N_I N_D)$. Alternatively, we can approximate the SVD of $\mathbf{B}$ via a randomized SVD \cite{HaMaTr11}.

\subsubsection{Discussion and comparison of probability distributions}
\label{subsec_disc_choice_prob}

In the following, we discuss and compare the probability distributions introduced in \cref{subsec_choice_prob} with respect to their capability of detecting relevant time points of (heterogeneous) time-dependent data functions and computational costs. For a more general comparison, we refer to \cite{DriMah16}.

As no data-dependent information is incorporated, the uniform sampling approach might not detect relevant problem-specific features in time unless a large number of time points is drawn and can thus lead to poor results. For instance, a data matrix $\mathbf{B}$ with only one non-zero column would require to draw $\mathcal{O}(N_I)$ columns (time points) in order to detect the non-zero data at the single point in time with high probability. In contrast to uniform sampling, there exist error bounds for both the squared norm and the leverage score sampling approach (cf. \cref{supp:error_bounds}).

As illustrated for two different time-dependent data functions\footnote{To ensure reproducibility, the data functions and discretization parameters for Example 1 are listed in \cref{appendix_data_example_1}.} in \cref{figure_comp_frob_lev}, we can infer that the squared norm sampling approach might not detect parts of the data that has values on smaller scales or smaller temporal scales. In contrast, the leverage scores weight the heterogeneous parts of the data equally with the same expectation. As a result, the leverage score sampling approach more likely detects all dominant features of the heterogeneous data and might lead to a better approximation accuracy than the squared norm approach. Moreover, as can be expected from the definition of leverage scores, we observe from numerical experiments not included in this paper that leverage scores are capable of detecting repetitions in the data functions.

Nevertheless, the costs for computing the rank-$r$ leverage scores are $r$ times the costs of computing the squared norm sampling distribution (for a dense matrix $\mathbf{B}$). To reduce the computational complexity, one could employ a randomized SVD and parallelize computations. However, this potentially still leads to costs that are not negligible compared to the uniform and squared norm sampling approach.

Consequently, uniform (and squared norm) sampling can be advantageous compared to leverage score sampling if many parallel compute units are available and the data is, for instance, spread relatively uniform over the whole time interval. In that case, drawing a large number of time points will likely yield a good approximation accuracy without having to compute (or approximate) the SVD of a large matrix.

\begin{figure}
	\begin{tikzpicture}
	\begin{axis}[
	width=6cm,%0.85\textwidth,
	height=3.5cm,
	xmin=0,
	xmax=10,
	ymin=0,
	ymax=4.4,
	legend style={at={(1.01,1)},anchor=north west,font=\footnotesize},
	xtick={0,5,10},
	xticklabels = {0,$t$,10},
	ytick={},
	xtick style = {white},
	ytick style = {white},
	xlabel= (a),
	%ylabel= ,
	label style={font=\footnotesize, yshift = 3},
	tick label style={font=\footnotesize}  
	]
	\addplot[solid, black, thick] table[x index=0, y index=1] {data_figures/compare_frob_lev/grid_t+2_long_stoves.dat};
	\addplot[solid, black, thick] table[x index=0, y index=2] {data_figures/compare_frob_lev/grid_t+2_long_stoves.dat};
	\addplot[densely dotted, red, very thick] table[x index=0, y index=1, y expr={\thisrowno{1}*400}] {data_figures/compare_frob_lev/leverage_2_long_stoves.dat};
	\addplot[densely dashed, teal,very thick] table[x index=0, y index=1, y expr={\thisrowno{1}*400}] {data_figures/compare_frob_lev/frobenius_2_long_stoves.dat};
	\end{axis}
	\end{tikzpicture}
	\hspace*{-0.3cm}
	\begin{tikzpicture}
	\begin{axis}[
	width=6cm,%0.8\textwidth,
	height=3.5cm,
	xmin=0,
	xmax=10,
	ymin=0,
	ymax=1.12,
	legend style={at={(1.02,1)},anchor=north west,font=\footnotesize},
	xtick={0,5,10},
	xticklabels = {0,$t$,10},
	ytick={},
	xtick style = {white},
	ytick style = {white},
	xlabel= (b),
	label style={font=\footnotesize, yshift = 3},
	tick label style={font=\footnotesize}  
	]
	\addplot[solid, black,thick] table[x index=0, y index=1] {data_figures/compare_frob_lev/grid_t+long+short_stove.dat};
	\addplot[solid, black,thick] table[x index=0, y index=2] {data_figures/compare_frob_lev/grid_t+long+short_stove.dat};
	\addplot[densely dotted, red, very thick] table[x index=0, y index=1, y expr={\thisrowno{1}*15}] {data_figures/compare_frob_lev/leverage_long+short_stove.dat};
	\addplot[densely dashed, teal,very thick] table[x index=0, y index=1, y expr={\thisrowno{1}*100}] {data_figures/compare_frob_lev/frobenius_long+short_stove.dat};
	\legend{data in time, ,leverage scores, squared norms};
	\end{axis}
	\end{tikzpicture}
	\caption{\footnotesize Example 1: Rank-$2$ leverage score (LS) and squared norm (SN) probability distributions for two (spatially disjoint) signals with different values (a) or different temporal scales (b). Values of LS and SN are scaled with a factor of $400$ (a) or $15$ (LS) and $100$ (SN) (b).
	}
	\label{figure_comp_frob_lev}
\end{figure}
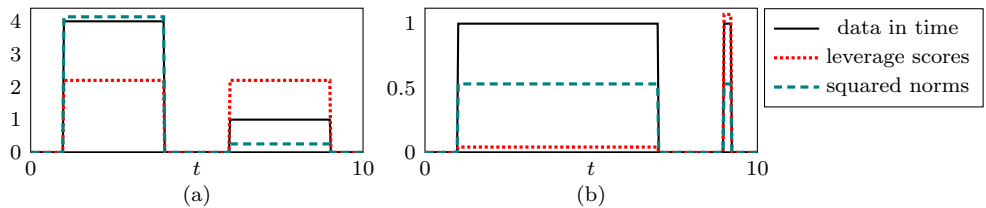

\subsubsection{Sampling from multiple probability distributions}
\label{subsubsec_multiple_probs}
If both $\mathcal{A}$ and $\mathcal{F}$ encode time-dependent data due to, for instance, a coefficient function and source terms that vary in time, one might want to include both data matrices $\mathbf{B}_\mathcal{A}$ and $\mathbf{B}_\mathcal{F}$ in the time point selection process. 

To this end, one option is to attach $\mathbf{B}_\mathcal{A}$ to $\mathbf{B}_\mathcal{F}$ and sample from the probability distribution computed from $[\mathbf{B}_\mathcal{A} \mathbf{B}_\mathcal{F}]\in \R^{N_D\times 2 N_I}$. However, if the data encoded in $\mathcal{F}$, for instance, has values on smaller scales compared to $\mathcal{A}$ (due to, e.g., high conductivity channels), employing the squared norm sampling approach one might sample solely from the part of the probability distribution that is associated with $\mathcal{A}$ and neglect the information encoded in $\mathcal{F}$ (cf. \cref{figure_comp_frob_lev} and the discussion in \cref{subsec_disc_choice_prob}). As leverage score sampling is based on the SVD of $[\mathbf{B}_\mathcal{A} \mathbf{B}_\mathcal{F}]$, the approach more likely detects the dominant modes encoded in both $\mathcal{A}$ and $\mathcal{F}$. Nevertheless, computing the SVD of the large matrix  $[\mathbf{B}_\mathcal{A} \mathbf{B}_\mathcal{F}]\in \R^{N_D\times 2 N_I}$ is more costly than computing the SVDs of $\mathbf{B}_\mathcal{A}\in \R^{N_D\times N_I}$ and $\mathbf{B}_\mathcal{F}\in \R^{N_D\times N_I}$ separately. Therefore, we propose to assemble the probability distributions associated with $\mathbf{B}_\mathcal{A}$ and $\mathbf{B}_\mathcal{F}$ separately and draw from both distributions simultaneously. Moreover, we note that in certain cases it might be necessary to also sample from data functions that are constant in time to achieve a good quality of approximation (e.g. in case of a constant advection field, cf. Experiment 3 in \cref{subsec_num_ex_2}).

\section{Numerical experiments}
\label{numerical_experiments} 

In this section we demonstrate the excellent approximation properties of the reduced basis generated via \cref{randomized_basis_generation}. In \cref{subsec_num_ex_1,subsec_num_ex_2} we first comprehensively test how the results depend on various parameters such as the number of chosen time points $\nrand$, the local oversampling size $\nt$, the number of collected snapshots for the SVD determined via $k$, or the probability distribution for drawing points in time. For this purpose, we consider both the linear heat equation and an advection-diffusion problem for time-dependent source terms. In particular, we demonstrate that the proposed method is able to tackle higher values of advection. Subsequently, we show in \cref{subsec_num_ex_4} that the randomized approach is well capable of approximating a problem with a time-dependent permeability coefficient that is rough with respect to both space and time using real-world data taken from the SPE10 benchmark problem \cite{SPE10}. Moreover, we demonstrate in \cref{subsec_num_ex_2,subsec_num_ex_4} how to sample from multiple probability distributions simultaneously. 

For the experiments, we employ the discretization introduced in \cref{sec_problem_setting}, use \cref{randomized_basis_generation} to generate the randomized reduced basis, and construct the reduced approximation via Galerkin projection as described in \cref{sec_problem_setting}. Hence, in what follows the term \textit{error} always refers to the error between the solution of \cref{eq_time_stepping} and its reduced approximation determined by solving the reduced problem \cref{eq_reduced_global_approximation}. Moreover, we prescribe homogeneous Dirichlet boundary conditions on $I\times \Sigma_D$ in all experiments. The source code to reproduce all results shown in this section is provided in \cite{code}.

\subsection{Stove problem}
\label{subsec_num_ex_1}

In this subsection, we consider the heat equation (\cref{PDE_ex} with $b\equiv c \equiv 0$) and investigate the following numerical experiment, which we refer to as Example 2: We choose $I=(0,10)$, $D=(0,1)^2$, $\Sigma_N=\emptyset$, and discretize the spatial domain $D$ with a regular quadrilateral mesh with mesh size $1/100$ in both directions. For the implicit Euler method, we use an equidistant time step size of $1/30$. Furthermore, we choose the initial condition $u_0(x,y)=\sum_{i=1}^{3} \sin(i \pi x) \sin(i \pi y)$, the coefficient $\kappa\equiv 1$, and the source term $f(t,x,y)=\sum_{i=1}^3 f_i(t)f_i(x,y)$ involving three spatially disjoint heat sources (stoves) that are turned on and off in time as illustrated in \cref{figure_stove_vals} (left and middle). \cref{figure_stove_vals} (right) shows the rank-$3$ leverage score probability distribution computed from the right-hand side matrix $\mathbf{F}$, whose columns are the right hand-side vectors $\mathbf{F}_{l}$ for $l=0,\ldots,300$ \cref{FE_matrices}. Unless stated otherwise, we use the rank-$3$ leverage score probability distribution to draw time points in \cref{randomized_basis_generation} for this example.

\begin{figure}
	\begin{tikzpicture}
	\begin{axis}[
	title={$f_1(x,y)$},
	width=3cm,
	height=3cm,
	xmin=0,
	xmax= 1,
	ymin=0,
	ymax=1,
	xlabel=$x$,
	ylabel=$y$,
	xtick = {0,0.5,1},
	ytick={0,0.5,1},
	tick style ={color=white},
	title style = {font=\footnotesize, yshift = -3},
	label style={font=\footnotesize},
	xlabel style = {yshift = 3},
	tick label style={font=\scriptsize}  
	]
	\draw[fill=lightgray] (0,0) rectangle (1,1);
	\draw[fill=teal] (0.2,0.2) rectangle (0.3,0.3);
	\end{axis}
	\end{tikzpicture}
	\hspace*{-0.5cm}
	\begin{tikzpicture}
	\begin{axis}[
	title={$f_2(x,y)$},
	title style = {font=\footnotesize, yshift = -3},
	width=3cm,
	height=3cm,
	xmin=0,
	xmax= 1,
	ymin=0,
	ymax=1,
	xlabel=$x$,
	ylabel=,
	xtick = {0,0.5,1},
	yticklabels={,,},
	tick style ={color=white},
	label style={font=\footnotesize},
	xlabel style = {yshift = 3},
	tick label style={font=\scriptsize}  
	]
	\draw[fill=lightgray] (0,0) rectangle (1,1);
	\draw[fill=red] (0.45,0.45) rectangle (0.55,0.55);
	\end{axis}
	\end{tikzpicture}
	\hspace*{-0.5cm}
	\begin{tikzpicture}
	\begin{axis}[
	title={$f_2(x,y)$},
	title style = {font=\footnotesize, yshift = -3},
	width=3cm,
	height=3cm,
	xmin=0,
	xmax= 1,
	ymin=0,
	ymax=1,
	xlabel=$x$,
	ylabel=,
	xtick = {0,0.5,1},
	yticklabels={,,},
	tick style ={color=white},
	label style={font=\footnotesize},
	xlabel style = {yshift = 3},
	tick label style={font=\scriptsize}  
	]
	\draw[fill=lightgray] (0,0) rectangle (1,1);
	\draw[fill=blue] (0.65,0.65) rectangle (0.8,0.8);
	\end{axis}
	\end{tikzpicture}
	\hspace*{-0.4cm}
	\begin{tikzpicture}
	\begin{axis}[
	title={$f_i(t)$},
	title style = {font=\footnotesize, yshift = -3},
	width=4cm,
	height=3cm,
	xmin=0,
	xmax= 10,
	ymin=0,
	ymax=31.5,
	legend style={at={(0.92,1)},anchor=north west,font=\footnotesize},
	xlabel=$t$,
	xtick = {0,5,10},
	xtick style = {white},
	ytick={10,20},
	ytick style = {white},
	xlabel style = {yshift = 3},
	label style={font=\footnotesize},
	tick label style={font=\scriptsize}  
	]
	\addplot[solid, teal, thick] table[x index=0, y index=1] {data_figures/stove_grid_t+f_vals.dat};
	\addplot[solid, red, thick] table[x index=0, y index=2] {data_figures/stove_grid_t+f_vals.dat};
	\addplot[solid, blue, thick] table[x index=0, y index=3] {data_figures/stove_grid_t+f_vals.dat};
	\legend{$f_1(t)$, $f_2(t)$, $f_3(t)$}
	\end{axis}
	\end{tikzpicture}
	\begin{tikzpicture}
	\begin{axis}[
	title={Leverage scores},
	title style = {font=\footnotesize, yshift = -2},
	width=4cm,
	height=3cm,
	xmin=0,
	xmax= 10,
	ymin=0,
	ymax=0.0072,
	xlabel=$t$,
	xtick = {0,5,10},
	xtick style = {white},
	ytick={0,0.007},
	ytick style = {white},
	label style={font=\footnotesize},
	xlabel style = {yshift = 3},
	tick label style={font=\scriptsize}  
	]
	\addplot[solid, purple, thick] table[x index=0, y index=1] {data_figures/stove_grid_t+leverage_scores.dat};
	\end{axis}
	\end{tikzpicture}
	\caption{\footnotesize Example 2: Source term $f(t,x,y)=\sum_{i=1}^{3} f_i(t) f_i(x,y)$ and corresponding rank-$3$ leverage score probability distribution. Gray equates to $0$, while green, red, and blue equate to $1$ (left).
	}
	\label{figure_stove_vals}
\end{figure}
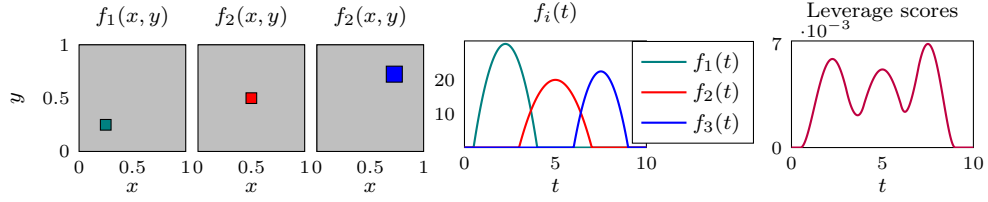

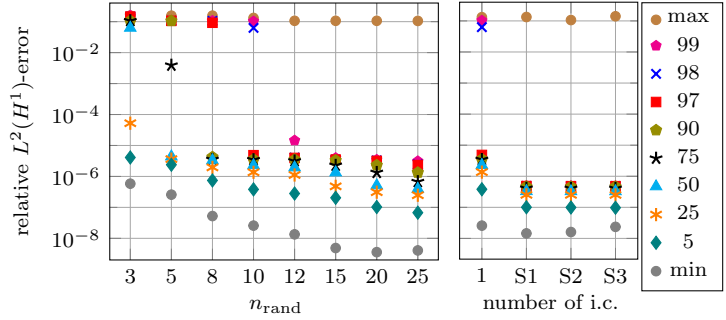
\begin{figure}
	\centering
	\begin{tikzpicture}
	\begin{semilogyaxis}[
	width=5.95cm,%0.85\textwidth,
	height=5cm,
	xmin=-0.5,
	xmax=7.5,
	ymin=2e-9,
	ymax=4e-1,
	legend style={at={(1.01,1)},anchor=north west,font=\footnotesize},
	grid=both,
	grid style={line width=.1pt, draw=gray!70},
	major grid style={line width=.2pt,draw=gray!70},
	xtick={0,1,2,3,4,5,6,7},
	ytick={1e-8,1e-6, 1e-4, 1e-2},
	minor ytick={1e-7, 1e-5, 1e-3, 1e-1},
	xticklabels={3,5,8,10,12,15,20,25},
	xlabel= $\textcolor{white}{()}\hspace{-5pt}\nrand$,
	xlabel style = {yshift = 3},
	ylabel= relative $L^2(H^1)$-error,
	label style={font=\footnotesize},
	tick label style={font=\footnotesize}  
	]
	\addplot[only marks, brown, mark=*, mark size=1.8pt] table[x expr=\coordindex, y index=0] {data_figures/stove_nrand_quantiles_rel_L2H1.dat};
	\addplot[only marks, magenta, mark=pentagon*, mark size=2pt] table[x expr=\coordindex, y index=1] {data_figures/stove_nrand_quantiles_rel_L2H1.dat};
	\addplot[only marks,blue, mark=x, mark size=2.5pt,thick] table[x expr=\coordindex, y index=2] {data_figures/stove_nrand_quantiles_rel_L2H1.dat};
	\addplot[only marks, red, mark=square*, mark size=1.8pt] table[x expr=\coordindex, y index=3] {data_figures/stove_nrand_quantiles_rel_L2H1.dat};
	%\addplot[only marks,violet,mark=asterisk, mark size=2.5pt,thick] table[x expr=\coordindex, y index=4] {data_figures/stove_nrand_quantiles_rel_L2H1.dat};
	\addplot[only marks, olive, mark=pentagon*, mark size=2.2pt] table[x expr=\coordindex, y index=5] {data_figures/stove_nrand_quantiles_rel_L2H1.dat};
	%\addplot[only marks,teal, mark=star, mark size=2.5pt, thick] table[x expr=\coordindex, y index=6] {data_figures/stove_nrand_quantiles_rel_L2H1.dat};
	%\addplot[only marks, brown, mark=square*, mark size=1.8pt] table[x expr=\coordindex, y index=7] {data_figures/stove_nrand_quantiles_rel_L2H1.dat};
	\addplot[only marks,black, mark=star,mark size=2.5pt, thick] table[x expr=\coordindex, y index=8] {data_figures/stove_nrand_quantiles_rel_L2H1.dat};
	\addplot[only marks,cyan,mark=triangle*, mark size = 2.5pt] table[x expr=\coordindex, y index=9] {data_figures/stove_nrand_quantiles_rel_L2H1.dat};
	\addplot[only marks, orange, mark=asterisk, mark size=2.5pt,thick] table[x expr=\coordindex, y index=10] {data_figures/stove_nrand_quantiles_rel_L2H1.dat};
	\addplot[only marks,teal, mark=diamond*, mark size=2.5pt] table[x expr=\coordindex, y index=11] {data_figures/stove_nrand_quantiles_rel_L2H1.dat};
	\addplot[only marks, gray, mark= oplus*, mark size=1.8pt] table[x expr=\coordindex, y index=12] {data_figures/stove_nrand_quantiles_rel_L2H1.dat};
	%\legend{max,99,98,97,90,75,50,25,5,min};
	\end{semilogyaxis}
	\end{tikzpicture}
	\hspace*{-0.2cm}
	\begin{tikzpicture}
	\begin{semilogyaxis}[
	width=3.95cm,%0.8\textwidth,
	height=5cm,
	xmin=-0.5,
	xmax=3.5,
	ymin=2.15e-9,
	ymax=4e-1,
	legend style={at={(1.02,1)},anchor=north west,font=\footnotesize},
	grid=both,
	grid style={line width=.1pt, draw=gray!70},
	major grid style={line width=.2pt,draw=gray!70},
	xtick={0,1,2,3},
	ytick={1e-8,1e-6, 1e-4, 1e-2},
	minor ytick={1e-7, 1e-5, 1e-3, 1e-1},
	yticklabels={,,},
	xticklabels={1,S1,S2,S3},
	xlabel= $\textcolor{white}{()}\hspace{-5pt}$number of i.c.,
	%ylabel= relative $L^2(H^1)$-error,
	label style={font=\footnotesize},
	xlabel style = {yshift = 3},
	tick label style={font=\footnotesize}  
	]
	\addplot[only marks, brown, mark=*, mark size=1.8pt] table[x expr=\coordindex, y index=0] {data_figures/stove_split_ninitials_quantiles_rel_L2H1.dat};
	\addplot[only marks, magenta, mark=pentagon*, mark size=2pt] table[x expr=\coordindex, y index=1] {data_figures/stove_split_ninitials_quantiles_rel_L2H1.dat};
	\addplot[only marks, blue, mark=x, mark size=2.5pt,thick] table[x expr=\coordindex, y index=2] {data_figures/stove_split_ninitials_quantiles_rel_L2H1.dat};
	\addplot[only marks, red, mark=square*, mark size=1.8pt] table[x expr=\coordindex, y index=3] {data_figures/stove_split_ninitials_quantiles_rel_L2H1.dat};
	%\addplot[only marks,magenta,mark=pentagon*, mark size=2.5pt] table[x expr=\coordindex, y index=4] {data_figures/stove_split_ninitials_quantiles_rel_L2H1.dat};
	\addplot[only marks, olive, mark=pentagon*, mark size=2.2pt] table[x expr=\coordindex, y index=5] {data_figures/stove_split_ninitials_quantiles_rel_L2H1.dat};
	%\addplot[only marks,teal, mark=star, mark size=2.5pt, thick] table[x expr=\coordindex, y index=6] {data_figures/stove_split_ninitials_quantiles_rel_L2H1.dat};
	%\addplot[only marks, red, mark=square*, mark size=1.8pt] table[x expr=\coordindex, y index=7] {data_figures/stove_split_ninitials_quantiles_rel_L2H1.dat};
	\addplot[only marks,black, mark=star,mark size=2.5pt, thick] table[x expr=\coordindex, y index=8] {data_figures/stove_split_ninitials_quantiles_rel_L2H1.dat};
	\addplot[only marks,cyan,mark=triangle*, mark size=2.5pt] table[x expr=\coordindex, y index=9] {data_figures/stove_split_ninitials_quantiles_rel_L2H1.dat};
	\addplot[only marks, orange, mark=asterisk, mark size=2.5pt,thick] table[x expr=\coordindex, y index=10] {data_figures/stove_split_ninitials_quantiles_rel_L2H1.dat};
	\addplot[only marks,teal, mark=diamond*, mark size=2.5pt] table[x expr=\coordindex, y index=11] {data_figures/stove_split_ninitials_quantiles_rel_L2H1.dat};
	\addplot[only marks, gray, mark= oplus*, mark size=1.8pt] table[x expr=\coordindex, y index=12] {data_figures/stove_split_ninitials_quantiles_rel_L2H1.dat};
	\legend{max,99,98,97,90,75,50,25,5,min};
	\end{semilogyaxis}
	\end{tikzpicture}
	\caption{\footnotesize Example 2: Quantiles of relative $L^2(I,H^1(D))$-error for $\nt = 15$, $k = 13$, $\tol = 10^{-8}$, and $100.000$ realizations for varying numbers of $\nrand$ (left) or $\nrand = 10$ and $1$, $2$, or $3$ random initial conditions (i.c.) per time point (right). Here, S indicates that local computations are performed separately for right-hand side and initial conditions (cf. \cref{remark_number_random_initials}).
	}
	\label{figure_stove_nrand+split_ninitials}
\end{figure}

\begin{figure}
	\centering
	\begin{tikzpicture}
	\begin{semilogyaxis}[
	width=6cm,%0.84\textwidth,
	height=5cm,
	xmin=0,
	xmax=15,
	ymin=1e-15,
	ymax=1e-3,
	legend style={at={(1.02,1)},anchor=north west,font=\footnotesize},
	grid=both,
	grid style={line width=.1pt, draw=gray!50},
	major grid style={line width=.2pt,draw=gray!50},
	xtick={0,2,4,6,8,10,12,14},
	minor xtick = {1,3,5,7,9},
	ytick={1e-14,1e-12,1e-10,1e-8,1e-6, 1e-4},
	minor ytick={1e-13,1e-11,1e-9, 1e-7, 1e-5, 1e-3},
	xlabel= $n$,
	ylabel = $\Vert T - \proj_{H^n_{\text{rand}}} T\Vert$,
	label style={font=\footnotesize},
	xlabel style = {yshift = 3},
	tick label style={font=\footnotesize}  
	]
	\addplot+[ thick, mark=x, mark repeat = 1, mark options={solid}, densely dotted, teal] table[x expr=\coordindex,y index=6]{data_figures/stove_transfer_singular_vals_nt=15.dat};
	\addplot+[ thick, mark=diamond, mark repeat = 1, mark options={solid}, densely dashed, blue] table[x expr=\coordindex,y index=5]{data_figures/stove_transfer_singular_vals_nt=15.dat};
	\addplot+[ thick, mark=star, mark repeat = 1, mark options={solid}, mark size = 2.4pt, densely dashdotted, orange] table[x expr=\coordindex,y index=4]{data_figures/stove_transfer_singular_vals_nt=15.dat};
	\addplot+[ thick, mark=asterisk, mark repeat = 1,mark size = 2.4pt, black, solid, mark options={solid}] table[x expr=\coordindex,y index=3]{data_figures/stove_transfer_singular_vals_nt=15.dat};
	\addplot+[ thick, mark=o, mark repeat = 1, densely dotted, brown, mark options={solid}] table[x expr=\coordindex,y index=2]{data_figures/stove_transfer_singular_vals_nt=15.dat};
	\addplot+[ thick, mark=triangle, mark repeat = 1, densely dashed, cyan, mark options={solid}] table[x expr=\coordindex,y index=1]{data_figures/stove_transfer_singular_vals_nt=15.dat};
	\addplot+[ thick, red, solid, mark = none] table[x expr=\coordindex,y index=0]{data_figures/stove_transfer_singular_vals_nt=15.dat};
	\legend{max, 95, 75, 50, 25, min, $\sigma^{(n+1)}_{t_i\rightarrow t_i+0.5}$}
	\end{semilogyaxis}
	\end{tikzpicture}
	\caption{\footnotesize Example 2: Singular value decay of transfer operator $T_{t_i \rightarrow t_i+ 15/30}$ ($0\leq i \leq 285$) and quantiles of projection error $\Vert T_{t_i\rightarrow t_i+0.5} - \proj_{H^n_{t_i+0.5,\text{rand}}} T_{t_i\rightarrow t_i+0.5}\Vert$ over basis size $n$ for $2.000$ realizations.
	}
	\label{figure_stove_transfer_singular_vals}
\end{figure}
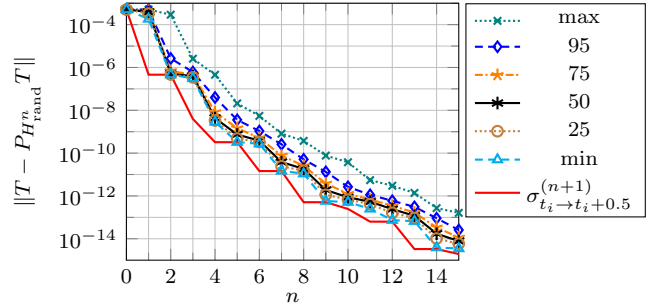

First, we investigate the influence of the number of drawn time points $\nrand$ on the approximation accuracy of the reduced basis. As the rank of the right-hand side is three, we choose $\nrand\geq 3$. In \cref{figure_stove_nrand+split_ninitials} (left) we observe that for $\nrand=3$ in $50\%$ of cases the error is of the order of $10^{-1}$ and only in $5\%$ of cases the error is below $10^{-5}$, while for $\nrand = 8$ or $\nrand = 10$ in $90\%$ or $97\%$ of cases the error is below $10^{-5}$. We thus infer that only a small amount of oversampling is necessary to detect all three stoves with high probability. If we choose $\nrand>10$ we see in \cref{figure_stove_nrand+split_ninitials} (left) that the approximation accuracy still slightly improves compared to smaller $\nrand$. Nevertheless, we recall that also the computational costs increase with increasing $\nrand$. For the following tests in this subsection we therefore choose $\nrand=10$ as a good trade-off between approximation accuracy and computational costs.

Next, we test how the approximation accuracy depends on the number of random initial conditions per drawn time point. We observe that for S$1$-S$3$ $99\%$ of the realizations yield an approximation with a relative $L^2(I,H^1(D))$-error below $10^{-6}$ and the approximation accuracy does not improve if we choose more than one random initial condition. This can be traced back to the very fast, exponential decay of the singular values of the transfer operator shown in \cref{figure_stove_transfer_singular_vals} (cf. \cref{probabilistic_a_priori}). If we draw only one random initial condition per chosen time point and do not separate the local computation for right-hand side and random initial condition, we observe in \cref{figure_stove_nrand+split_ninitials} (right) that $97\%$ of the realizations yield an approximation with a relative $L^2(I,H^1(D))$-error below $10^{-5}$. As the approximation accuracy is thus very good and the costs for computing local solutions are only half as much as in the separated approach (S$1$), we choose to not separate local computations for right-hand side and random initial condition in all other tests in this subsection.

We observe in \cref{figure_stove_transfer_singular_vals} that while the randomized spaces often provide an approximation that converges nearly with the optimal rate $\sigma^{(n+1)}_{t_i\rightarrow t_i+0.5}$, in some cases two or three basis vectors more are needed to guarantee the same approximation accuracy. This is in line with the predictions by theory (cf. \cref{probabilistic_a_priori}).

\begin{figure}
	\begin{minipage}{0.64\textwidth}
		\begin{tikzpicture}
		\begin{semilogyaxis}[
		width=4.5cm,%0.84\textwidth,
		height=5cm,
		xmin=-0.5,
		xmax=4.5,
		ymin=4e-9,
		ymax=4e-1,
		legend style={at={(1.01,1)},anchor=north west,font=\footnotesize},
		grid=both,
		grid style={line width=.1pt, draw=gray!70},
		major grid style={line width=.2pt,draw=gray!70},
		xtick={0,1,2,3,4},
		ytick={1e-8,1e-6, 1e-4, 1e-2},
		minor ytick={1e-7, 1e-5, 1e-3, 1e-1},
		xticklabels={11,12,13,14,15},
		xlabel style = {yshift = 3},
		xlabel= $\textcolor{white}{()}\hspace{-5pt} k$,
		ylabel= relative $L^2(H^1)$-error,
		label style={font=\footnotesize},
		tick label style={font=\footnotesize}  
		]
		\addplot[only marks, brown, mark=*, mark size=1.8pt] table[x expr=\coordindex, y index=0] {data_figures/stove_k_quantiles_rel_L2H1.dat};
		\addplot[only marks, magenta, mark=pentagon*, mark size=2pt] table[x expr=\coordindex, y index=1] {data_figures/stove_k_quantiles_rel_L2H1.dat};
		\addplot[only marks,blue, mark=x, mark size=2.5pt,thick] table[x expr=\coordindex, y index=2] {data_figures/stove_k_quantiles_rel_L2H1.dat};
		\addplot[only marks, red, mark=square*, mark size=1.8pt] table[x expr=\coordindex, y index=3] {data_figures/stove_k_quantiles_rel_L2H1.dat};
		%\addplot[only marks,violet,mark=asterisk, mark size=2.5pt,thick] table[x expr=\coordindex, y index=4] {data_figures/stove_k_quantiles_rel_L2H1.dat};
		\addplot[only marks, olive, mark=pentagon*, mark size=2.2pt] table[x expr=\coordindex, y index=5] {data_figures/stove_k_quantiles_rel_L2H1.dat};
		%\addplot[only marks,teal, mark=star, mark size=2.5pt, thick] table[x expr=\coordindex, y index=6] {data_figures/stove_k_quantiles_rel_L2H1.dat};
		%\addplot[only marks, brown, mark=square*, mark size=1.8pt] table[x expr=\coordindex, y index=7] {data_figures/stove_k_quantiles_rel_L2H1.dat};
		\addplot[only marks,black, mark=star,mark size=2.5pt, thick] table[x expr=\coordindex, y index=8] {data_figures/stove_k_quantiles_rel_L2H1.dat};
		\addplot[only marks,cyan,mark=triangle*, mark size = 2.5pt] table[x expr=\coordindex, y index=9] {data_figures/stove_k_quantiles_rel_L2H1.dat};
		\addplot[only marks, orange, mark=asterisk, mark size=2.5pt,thick] table[x expr=\coordindex, y index=10] {data_figures/stove_k_quantiles_rel_L2H1.dat};
		\addplot[only marks,teal, mark=diamond*, mark size=2.5pt] table[x expr=\coordindex, y index=11] {data_figures/stove_k_quantiles_rel_L2H1.dat};
		\addplot[only marks, gray, mark= oplus*, mark size=1.8pt] table[x expr=\coordindex, y index=12] {data_figures/stove_k_quantiles_rel_L2H1.dat};
		%\legend{max,99,98,97,90,75,50,25,5,min};
		\end{semilogyaxis}
		\end{tikzpicture}
		\hspace*{-0.2cm}
		\begin{tikzpicture}
		\begin{semilogyaxis}[
		width=4cm,%0.84\textwidth,
		height=5cm,
		xmin=-0.5,
		xmax=3.5,
		ymin=4e-9,
		ymax=4e-1,
		legend style={at={(1.02,1)},anchor=north west,font=\footnotesize},
		grid=both,
		grid style={line width=.1pt, draw=gray!70},
		major grid style={line width=.2pt,draw=gray!70},
		xtick={0,1,2,3},
		ytick={1e-8,1e-6, 1e-4, 1e-2},
		yticklabels={,,},
		minor ytick={1e-7, 1e-5, 1e-3, 1e-1},
		xticklabels={10,15,20,25},
		xlabel style = {yshift = 3},
		xlabel= $\textcolor{white}{()}\hspace{-5pt}\nt$,
		label style={font=\footnotesize},
		tick label style={font=\footnotesize}  
		]
		\addplot[only marks, brown, mark=*, mark size=1.8pt] table[x expr=\coordindex, y index=0] {data_figures/stove_nt_quantiles_rel_L2H1.dat};
		\addplot[only marks, magenta, mark=pentagon*, mark size=2pt] table[x expr=\coordindex, y index=1] {data_figures/stove_nt_quantiles_rel_L2H1.dat};
		\addplot[only marks,blue, mark=x, mark size=2.5pt,thick] table[x expr=\coordindex, y index=2] {data_figures/stove_nt_quantiles_rel_L2H1.dat};
		\addplot[only marks, red, mark=square*, mark size=1.8pt] table[x expr=\coordindex, y index=3] {data_figures/stove_nt_quantiles_rel_L2H1.dat};
		%\addplot[only marks,violet,mark=asterisk, mark size=2.5pt,thick] table[x expr=\coordindex, y index=4] {data_figures/stove_nt_quantiles_rel_L2H1.dat};
		\addplot[only marks, olive, mark=pentagon*, mark size=2.2pt] table[x expr=\coordindex, y index=5] {data_figures/stove_nt_quantiles_rel_L2H1.dat};
		%\addplot[only marks,teal, mark=star, mark size=2.5pt, thick] table[x expr=\coordindex, y index=6] {data_figures/stove_nt_quantiles_rel_L2H1.dat};
		%\addplot[only marks, brown, mark=square*, mark size=1.8pt] table[x expr=\coordindex, y index=7] {data_figures/stove_nt_quantiles_rel_L2H1.dat};
		\addplot[only marks,black, mark=star,mark size=2.5pt, thick] table[x expr=\coordindex, y index=8] {data_figures/stove_nt_quantiles_rel_L2H1.dat};
		\addplot[only marks,cyan,mark=triangle*, mark size = 2.5pt] table[x expr=\coordindex, y index=9] {data_figures/stove_nt_quantiles_rel_L2H1.dat};
		\addplot[only marks, orange, mark=asterisk, mark size=2.5pt,thick] table[x expr=\coordindex, y index=10] {data_figures/stove_nt_quantiles_rel_L2H1.dat};
		\addplot[only marks,teal, mark=diamond*, mark size=2.5pt] table[x expr=\coordindex, y index=11] {data_figures/stove_nt_quantiles_rel_L2H1.dat};
		\addplot[only marks, gray, mark= oplus*, mark size=1.8pt] table[x expr=\coordindex, y index=12] {data_figures/stove_nt_quantiles_rel_L2H1.dat};
		\legend{max,99,98,97,90,75,50,25,5,min};
		\end{semilogyaxis}
		\end{tikzpicture}
	\end{minipage}
	\begin{minipage}{0.35\textwidth}
		\begin{tikzpicture}
		\begin{axis}[
		width=4cm,%0.85\textwidth,
		height=5cm,
		xmin=-0.5,
		xmax=3.5,
		ymin=8.3,
		ymax=24.7,
		legend style={at={(1.02,1)},anchor=north west,font=\footnotesize},
		grid=both,
		grid style={line width=.1pt, draw=gray!70},
		major grid style={line width=.2pt,draw=gray!70},
		xtick={0,1,2,3},
		ytick={10,15,20},
		minor ytick={9,11,12,13,14,16,17,18,19,21,22,23,24},
		xticklabels={10,15,20,25},
		xlabel style = {yshift = 3},
		xlabel= $\textcolor{white}{()}\hspace{-5pt}\nt$,
		ylabel= reduced dimension,
		label style={font=\footnotesize},
		tick label style={font=\footnotesize}  
		]
		\addplot[only marks, brown, mark=*, mark size=1.8pt] table[x expr=\coordindex, y index=0] {data_figures/stove_nt_quantiles_red_sizes.dat};
		\addplot[only marks, magenta, mark=pentagon*, mark size=2pt] table[x expr=\coordindex, y index=1] {data_figures/stove_nt_quantiles_red_sizes.dat};
		%\addplot[only marks, blue, mark=x, mark size=2.5pt,thick] table[x expr=\coordindex, y index=2] {data_figures/stove_nt_quantiles_red_sizes.dat};
		%\addplot[only marks, red, mark=square*, mark size=1.8pt] table[x expr=\coordindex, y index=3] {data_figures/stove_nt_quantiles_red_sizes.dat};
		\addplot[only marks,olive,mark=square*, mark size=1.8pt] table[x expr=\coordindex, y index=4] {data_figures/stove_nt_quantiles_red_sizes.dat};
		%\addplot[only marks, orange, mark=diamond*, mark size=2.5pt] table[x expr=\coordindex, y index=5] {data_figures/stove_nt_quantiles_red_sizes.dat};
		%\addplot[only marks,teal, mark=star, mark size=2.5pt, thick] table[x expr=\coordindex, y index=6] {data_figures/stove_nt_quantiles_red_sizes.dat};
		%\addplot[only marks, red, mark=square*, mark size=1.8pt] table[x expr=\coordindex, y index=7] {data_figures/stove_nt_quantiles_red_sizes.dat};
		\addplot[only marks,black, mark=star,mark size=2.5pt, thick] table[x expr=\coordindex, y index=8] {data_figures/stove_nt_quantiles_red_sizes.dat};
		\addplot[only marks,cyan,mark=triangle*, mark size=2.5pt] table[x expr=\coordindex, y index=9] {data_figures/stove_nt_quantiles_red_sizes.dat};
		\addplot[only marks, orange, mark=asterisk, mark size=2.5pt,thick] table[x expr=\coordindex, y index=10] {data_figures/stove_nt_quantiles_red_sizes.dat};
		\addplot[only marks,teal, mark=diamond*, mark size=2.5pt] table[x expr=\coordindex, y index=11] {data_figures/stove_nt_quantiles_red_sizes.dat};
		\addplot[only marks, gray, mark= oplus*, mark size=1.8pt] table[x expr=\coordindex, y index=12] {data_figures/stove_nt_quantiles_red_sizes.dat};
		\legend{max,99,95,75,50,25,5,min};
		\end{axis}
		\end{tikzpicture}
	\end{minipage}
	\caption{\footnotesize Example 2: Quantiles of relative $L^2(I,H^1(D))$-error for $\nt=15$, varying numbers of $k$, $\nrand = 10$, $\tol = 10^{-8}$, and $100.000$ realizations (left). Quantiles of relative $L^2(I,H^1(D))$-error (middle) and reduced dimension (right) for varying numbers of $\nt$, $k = \nt-2$, $\nrand=10$, $\tol = 10^{-8}$, and $100.000$ realizations.
	}
	\label{figure_stove_k+nt+red_sizes}
\end{figure}
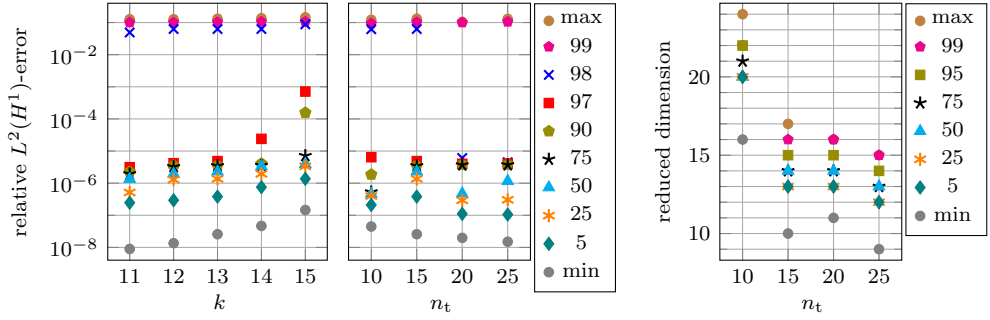

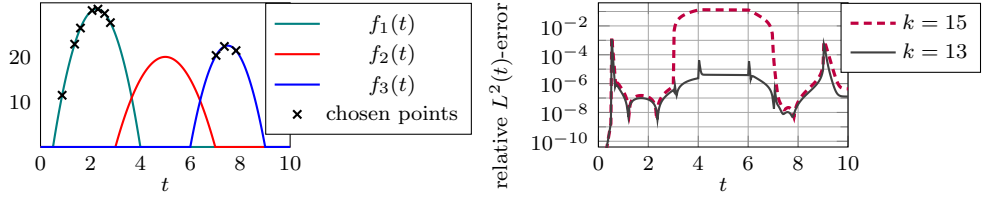
\begin{figure}
	\begin{tikzpicture}
	\begin{axis}[
	width=4.9cm,%0.85\textwidth,
	height=3.5cm,
	xmin=0,
	xmax= 10,
	ymin=0,
	ymax=32,
	legend style={at={(0.9,1)},anchor=north west,font=\footnotesize},
	xlabel=$t$,
	xtick = {0,2,4,6,8,10},
	xtick style = {white},
	ytick={10,20},
	ytick style = {white},
	xlabel= $t$,
	label style={font=\footnotesize},
	xlabel style = {yshift = 4},
	tick label style={font=\footnotesize}  
	]
	\addplot[solid, teal, thick] table[x index=0, y index=1] {data_figures/stove_grid_t+f_vals.dat};
	\addplot[solid, red, thick] table[x index=0, y index=2] {data_figures/stove_grid_t+f_vals.dat};
	\addplot[solid, blue, thick] table[x index=0, y index=3] {data_figures/stove_grid_t+f_vals.dat};
	\addplot[only marks,black, mark=x,mark size=2.2pt, thick] table[x index = 0, y index=1] {data_figures/example_k/chosen_points+f_vals.dat};
	\legend{$f_1(t)$, $f_2(t)$, $f_3(t)$, chosen points}
	\end{axis}
	\end{tikzpicture}
	\begin{tikzpicture}
	\begin{semilogyaxis}[
	width=4.9cm,%0.8\textwidth,
	height=3.5cm,
	xmin=0,
	xmax=10,
	ymin=4e-11,
	ymax=4e-1,
	legend style={at={(0.97,1)},anchor=north west,font=\footnotesize},
	grid=both,
	grid style={line width=.1pt, draw=gray!70},
	major grid style={line width=.2pt,draw=gray!70},
	xtick={0,2,4,6,8,10},
	ytick={1e-10,1e-8,1e-6, 1e-4, 1e-2},
	minor ytick={1e-11,1e-9,1e-7, 1e-5, 1e-3, 1e-1},
	xlabel= $t$,
	ylabel= relative $L^2(t)$-error,
	xlabel style = {yshift = 4},
	ylabel style = {xshift = -8},
	label style={font=\footnotesize},
	tick label style={font=\footnotesize}  
	]
	\addplot[densely dashed,purple,  very thick] table[x index = 0, y index = 2]{data_figures/example_k/grid_t+L2_errors_in_time.dat};
	\addplot[solid ,darkgray, thick] table[x index = 0, y index = 1]{data_figures/example_k/grid_t+L2_errors_in_time.dat};
	\legend{$k=15$, $k=13$}
	\end{semilogyaxis}
	\end{tikzpicture}
	\caption{\footnotesize Example 2: $\nrand=10$ randomly chosen points (left) and relative $L^2(t)$-error of corresponding reduced approximation for $\nt=15$, $\tol = 10^{-8}$ and $k=15$ vs. $k=13$ (right).
	}
	\label{figure_stove_comp_k}
\end{figure}

Moreover, we investigate how the number of collected snapshots determined via the parameter $k$ influences the approximation accuracy of the reduced basis. In \cref{figure_stove_k+nt+red_sizes} (left) we observe that the approximation quality significantly improves if we collect not only solution snapshots at local end time points ($k=15$), but at the (locally) last two or three time points ($k=14$ or $k=13$) as, for instance, $75\%$ of realizations have a relative $L^2(I,H^1(D))$-error below $10^{-5}$ for $k=15$, while for $k=14$ or $k=13$ already $90\%$ or $97\%$ of realizations have a relative $L^2(I,H^1(D))$-error below $10^{-5}$. \cref{figure_stove_comp_k} illustrates an explanation for this phenomenon. In the particular realization only time points in the first and last stove are drawn from the probability distribution. However, by choosing $k=13$ we include a solution snapshot that also detects the second stove and can thus significantly decrease the error in the time interval $(3,7)$ for this particular realization compared to $k=15$ as shown in \cref{figure_stove_comp_k} (right). For $k=12$ or $k=11$ we observe in \cref{figure_stove_k+nt+red_sizes} (left) only very slight improvements in the approximation quality compared to $k=13$. We therefore choose $k=\nt-2$ in all other tests in this subsection as a trade-off between approximation quality and size of the snapshot matrix and thus computational costs for its SVD.

Subsequently, we test how the approximation accuracy depends on the local oversampling size $\nt$. For all tested oversampling sizes, and in particular already for a small oversampling size of $\nt=10$, we observe in \cref{figure_stove_k+nt+red_sizes} (middle) that $97\%$ of realizations have a relative $L^2(I,H^1(D))$-error below $10^{-5}$ ($98\%$ for $\nt=20$ and $\nt=25$). Nevertheless, we see in \cref{figure_stove_k+nt+red_sizes} (right) that for $\nt=10$ the reduced basis is significantly larger compared to $\nt=15\,(20,25)$ as, for instance, in $95\%$ of cases the reduced dimension is larger than or equal to $20$ for $\nt=10$, while for $\nt=15, 20, 25$ it is smaller than or equal to $15$ in $95\%$ of cases. This can be explained by the fact that for smaller $\nt$ the randomness of the initial conditions has a larger influence compared to larger $\nt$ due to the exponential decay behavior of (local) solutions in time. In \cref{figure_stove_k+nt+red_sizes} (middle) we observe that for $\nt=20$ or $\nt=25$ the quality of approximation slightly improves compared to $\nt=15$, but we recall that also the computational costs increase with increasing $\nt$. For all other tests in this subsection we therefore choose $\nt=15$ as a good trade-off between approximation quality, computational costs, and size of the reduced basis.

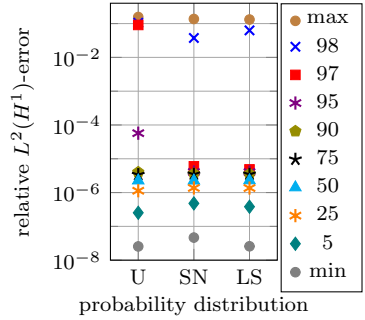
\begin{figure}
	\centering
	\begin{tikzpicture}
	\begin{semilogyaxis}[
	width=3.8cm,%0.8\textwidth,
	height=5cm,
	xmin=-0.5,
	xmax=2.5,
	ymin=1e-8,
	ymax=4e-1,
	legend style={at={(1.02,1)},anchor=north west,font=\footnotesize},
	grid=both,
	grid style={line width=.1pt, draw=gray!70},
	major grid style={line width=.2pt,draw=gray!70},
	xtick={0,1,2,3},
	ytick={1e-8,1e-6, 1e-4, 1e-2},
	minor ytick={1e-7, 1e-5, 1e-3, 1e-1},
	xticklabels={U, SN, LS},
	xlabel= probability distribution,
	xlabel style = {yshift = 2},
	ylabel= relative $L^2(H^1)$-error,
	label style={font=\footnotesize},
	tick label style={font=\footnotesize}  
	]
	\addplot[only marks, brown, mark=*, mark size=1.8pt] table[x expr=\coordindex, y index=0] {data_figures/stove_unif+frob+lev_quantiles_rel_L2H1.dat};
	%\addplot[only marks, magenta, mark=pentagon*, mark size=2pt] table[x expr=\coordindex, y index=1] {data_figures/stove_unif+frob+lev_quantiles_rel_L2H1.dat};
	\addplot[only marks, blue, mark=x, mark size=2.5pt,thick] table[x expr=\coordindex, y index=2] {data_figures/stove_unif+frob+lev_quantiles_rel_L2H1.dat};
	\addplot[only marks, red, mark=square*, mark size=1.8pt] table[x expr=\coordindex, y index=3] {data_figures/stove_unif+frob+lev_quantiles_rel_L2H1.dat};
	\addplot[only marks,violet,mark=asterisk, mark size=2.5pt,thick] table[x expr=\coordindex, y index=4] {data_figures/stove_unif+frob+lev_quantiles_rel_L2H1.dat};
	\addplot[only marks, olive, mark=pentagon*, mark size=2.2pt] table[x expr=\coordindex, y index=5] {data_figures/stove_unif+frob+lev_quantiles_rel_L2H1.dat};
	%\addplot[only marks,teal, mark=star, mark size=2.5pt, thick] table[x expr=\coordindex, y index=6] {data_figures/stove_unif+frob+lev_quantiles_rel_L2H1.dat};
	%\addplot[only marks, red, mark=square*, mark size=1.8pt] table[x expr=\coordindex, y index=7] {data_figures/stove_unif+frob+lev_quantiles_rel_L2H1.dat};
	\addplot[only marks,black, mark=star,mark size=2.5pt, thick] table[x expr=\coordindex, y index=8] {data_figures/stove_unif+frob+lev_quantiles_rel_L2H1.dat};
	\addplot[only marks,cyan,mark=triangle*, mark size=2.5pt] table[x expr=\coordindex, y index=9] {data_figures/stove_unif+frob+lev_quantiles_rel_L2H1.dat};
	\addplot[only marks, orange, mark=asterisk, mark size=2.5pt,thick] table[x expr=\coordindex, y index=10] {data_figures/stove_unif+frob+lev_quantiles_rel_L2H1.dat};
	\addplot[only marks,teal, mark=diamond*, mark size=2.5pt] table[x expr=\coordindex, y index=11] {data_figures/stove_unif+frob+lev_quantiles_rel_L2H1.dat};
	\addplot[only marks, gray, mark= oplus*, mark size=1.8pt] table[x expr=\coordindex, y index=12] {data_figures/stove_unif+frob+lev_quantiles_rel_L2H1.dat};
	\legend{max,98,97,95,90,75,50,25,5,min};
	\end{semilogyaxis}
	\end{tikzpicture}
	\caption{\footnotesize Example 2: Quantiles of relative $L^2(I,H^1(D))$-error for $\nrand = 10$, $\nt = 15$, $k = 13$, $\tol = 10^{-8}$, $100.000$ realizations, and uniform (U), squared norm (SN), or leverage score (LS) probability distribution.
	}
	\label{figure_stove_comp_probs}
\end{figure}

\begin{figure}
	\begin{tikzpicture}
	\begin{semilogyaxis}[
	width=6.8cm,%0.8\textwidth,
	height=4.5cm,
	xmin=0,
	xmax=10,
	ymin=1e-10,
	ymax=1e-0,
	legend style={at={(0.1,1.15)},anchor=north west,font=\footnotesize},
	grid=both,
	grid style={line width=.1pt, draw=gray!70},
	major grid style={line width=.2pt,draw=gray!70},
	xtick={0,2,4,6,8,10},
	ytick={1e-10,1e-8,1e-6, 1e-4, 1e-2,1e-0},
	minor ytick={1e-11,1e-9,1e-7, 1e-5, 1e-3, 1e-1},
	xlabel= $t$,
	ylabel= relative $L^2(t)$-error,
	label style={font=\footnotesize},
	xlabel style = {yshift=4},
	tick label style={font=\footnotesize}  
	]
	\addplot[densely dashdotted,purple,  very thick] table[x index = 0, y index = 1]{data_figures/stove_compare_POD_random.dat};
	\addplot[densely dotted ,blue, very thick] table[x index = 0, y index = 3]{data_figures/stove_compare_POD_random_new.dat};
	\addplot[solid ,black, thick] table[x index = 0, y index = 2]{data_figures/stove_compare_POD_random_new.dat};
	\addplot[densely dashed,teal, very thick] table[x index = 0, y index = 1]{data_figures/stove_compare_POD_random_new.dat};
	\legend{POD, LS 95, LS 50, LS 5}
	\end{semilogyaxis}
	\end{tikzpicture}
	\begin{tikzpicture}
	\begin{semilogyaxis}[
	width=6.8cm,%0.8\textwidth,
	height=4.5cm,
	xmin=0,
	xmax=10,
	ymin=1e-10,
	ymax=1e-0,
	legend style={at={(0.1,1.15)},anchor=north west,font=\footnotesize},
	grid=both,
	grid style={line width=.1pt, draw=gray!70},
	major grid style={line width=.2pt,draw=gray!70},
	xtick={0,2,4,6,8,10},
	ytick={1e-10,1e-8,1e-6, 1e-4, 1e-2},
	minor ytick={1e-11,1e-9,1e-7, 1e-5, 1e-3, 1e-1},
	yticklabels={,,},
	xlabel= $t$,
	label style={font=\footnotesize},
	xlabel style = {yshift=4},
	tick label style={font=\footnotesize}  
	]
	\addplot[densely dashdotted,purple,  very thick] table[x index = 0, y index = 1]{data_figures/stove_compare_POD_random.dat};
	\addplot[densely dotted ,blue, very thick] table[x index = 0, y index = 6]{data_figures/stove_compare_POD_random_new.dat};
	\addplot[solid ,black, thick] table[x index = 0, y index = 5]{data_figures/stove_compare_POD_random_new.dat};
	\addplot[densely dashed ,teal, very thick] table[x index = 0, y index = 4]{data_figures/stove_compare_POD_random_new.dat};
	\legend{POD, U 95, U 50, U 5}
	\end{semilogyaxis}
	\end{tikzpicture}
	\caption{\footnotesize Example 2: $5$, $50$, and $95\%$ quantiles of relative $L^2(t)$-error for $\nrand=10$, $\nt=15$, $k=13$, $\tol = 10^{-8}$, $100.000$ realizations, and leverage score (LS) or uniform (U) probability distribution vs. relative $L^2(t)$-error for POD on solution trajectory of first $165$ of $300$ time steps with tolerance $10^{-8}$ (POD).
	}
	\label{figure_stove_comp_POD}
\end{figure}
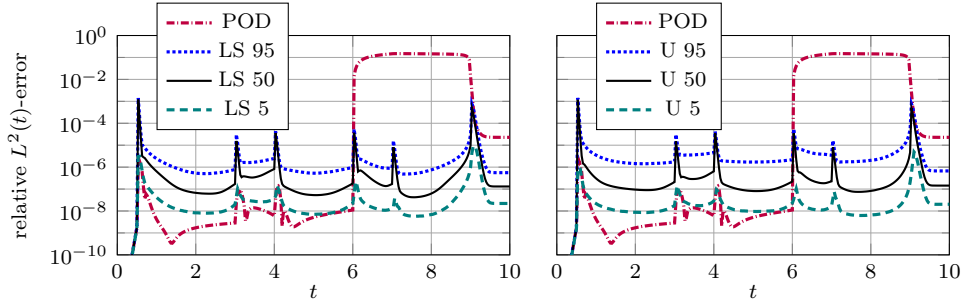

Next, we investigate how the choice of the probability distribution influences the approximation accuracy for the considered test case. To this end, we employ either the uniform, squared norm, or leverage score probability distribution as introduced in \cref{subsec_choice_prob} for drawing time points in \cref{randomized_basis_generation}. Both squared norms and leverage scores are computed from the right-hand side matrix $\mathbf{F}$. In \cref{figure_stove_comp_probs} we observe that the squared norm and the leverage score probability distribution yield a comparably good approximation quality and in both cases $97\%$ of realizations have a relative $L^2(I,H^1(D))$-error below $10^{-5}$. The uniform sampling approach also achieves a good approximation quality as, for instance, $95\%$ ($90\%$) of realizations have a relative $L^2(I,H^1(D))$-error below $10^{-4}$ ($10^{-5}$). Therefore, the results shown in \cref{figure_stove_comp_probs} indicate that for data that is spread over almost the whole time interval, aside from the leverage scores, the uniform or squared norm sampling approach can also lead to very good approximation results. Especially if the computer architecture allows for many parallel computations, one can thus save expenses for computing an SVD of a potentially large data matrix and instead draw a large number of time points from the uniform (or squared norm) probability distribution as discussed in \cref{subsec_disc_choice_prob}.

As the POD is a well-established tool for compressing and reducing time trajectories, we compare in \cref{figure_stove_comp_POD} quantiles of the relative $L^2(t)$-approximation error for both the uniform and the leverage score sampling approach with the relative $L^2(t)$-error of the approximation via POD on the solution trajectory of the first $165$ of $300$ time steps. In this way, we compare the standard POD approach with the randomized approach based on the same computational budget of $(\nrand +1)\cdot \nt = (10+1)\cdot 15 = 165$ time steps. While we also have to compute the SVD of the solution trajectory for the POD and the SVD of the data matrix for the leverage score sampling approach, we here focus on equaling the budget based on the time stepping, which likely dominates the computational costs in complex applications. We observe in \cref{figure_stove_comp_POD} that the POD approach, in contrast to the randomized approach in at least $95\%$ of cases, is not able to detect the third stove and thus yields a much larger relative approximation error in the time interval $(6,9)$ compared to both the uniform and the leverage score sampling approach. Moreover, in the randomized approach the local PDE simulations can be performed in parallel, while for the POD the first $165$ time instances of the solution trajectory have to be computed sequentially.

\subsection{Advection-diffusion problem}
\label{subsec_num_ex_2}

\begin{figure}
	\centering
	\begin{tikzpicture}
	\begin{semilogyaxis}[
	width=5cm,%0.84\textwidth,
	height=5cm,
	xmin=0,
	xmax=25,
	ymin=3e-15,
	ymax=1e-3,
	legend style={at={(1.05,1)},anchor=north west,font=\footnotesize},
	grid=both,
	grid style={line width=.1pt, draw=gray!50},
	major grid style={line width=.2pt,draw=gray!50},
	xtick={0,5,10,15,20,25},
	minor xtick = {},
	ytick={1e-20,1e-16,1e-14,1e-12,1e-10, 1e-8, 1e-6, 1e-4},
	minor ytick={1e-21,1e-19,1e-18,1e-17,1e-15,1e-14,1e-13,1e-11,1e-10,1e-9, 1e-7,1e-6, 1e-5, 1e-3},
	xlabel= $n$,
	ylabel = $\sigma^{(n+1)}_{t_i \rightarrow t_i+ 0.1}$,
	label style={font=\footnotesize},
	xlabel style = {yshift = 3},
	tick label style={font=\footnotesize}  
	]
	\addplot+[thick, solid, red, mark = none] table[x index = 0,y index=1]{data_figures/advection_transfer_singular_vals_nt=10.dat};
	\addplot+[ thick, mark=o, mark repeat = 2, mark size = 1.6pt, mark options={solid}, densely dashed, black] table[x index = 0,y index=2]{data_figures/advection_transfer_singular_vals_nt=10.dat};
	\addplot+[ thick, mark=asterisk, mark repeat = 2,mark size = 2.4pt, blue, densely dotted, mark options={solid}] table[x index = 0,y index=3]{data_figures/advection_transfer_singular_vals_nt=10.dat};
	\addplot+[ thick, mark=diamond, mark repeat = 2, densely dashdotted, cyan, mark options={solid}] table[x index = 0,y index=4]{data_figures/advection_transfer_singular_vals_nt=10.dat};
	\addplot+[very thick, teal, densely dotted, mark = none] table[x index = 0,y index=5]{data_figures/advection_transfer_singular_vals_nt=10.dat};
	\legend{$b_1=0$,$b_1=10$,$b_1=25$,$b_1=50$,$b_1=100$}
	\end{semilogyaxis}
	\end{tikzpicture}
	\caption{\footnotesize Example 3a: Singular value decay of transfer operator $T_{t_i \rightarrow t_i+ 0.1}$ ($0\leq i \leq 491$) for different values of advection in $x$-direction and a constant diffusion of $1$.
	}
	\label{figure_advection_transfer_singular_vals}
\end{figure}
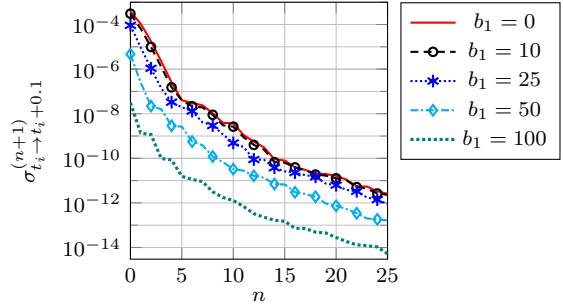

\begin{figure}
	\begin{minipage}{13cm}
		\centering
		{\footnotesize initial conditions at time point $t_i$}\\
		\vspace{1pt}
		\includegraphics[width = 6cm]{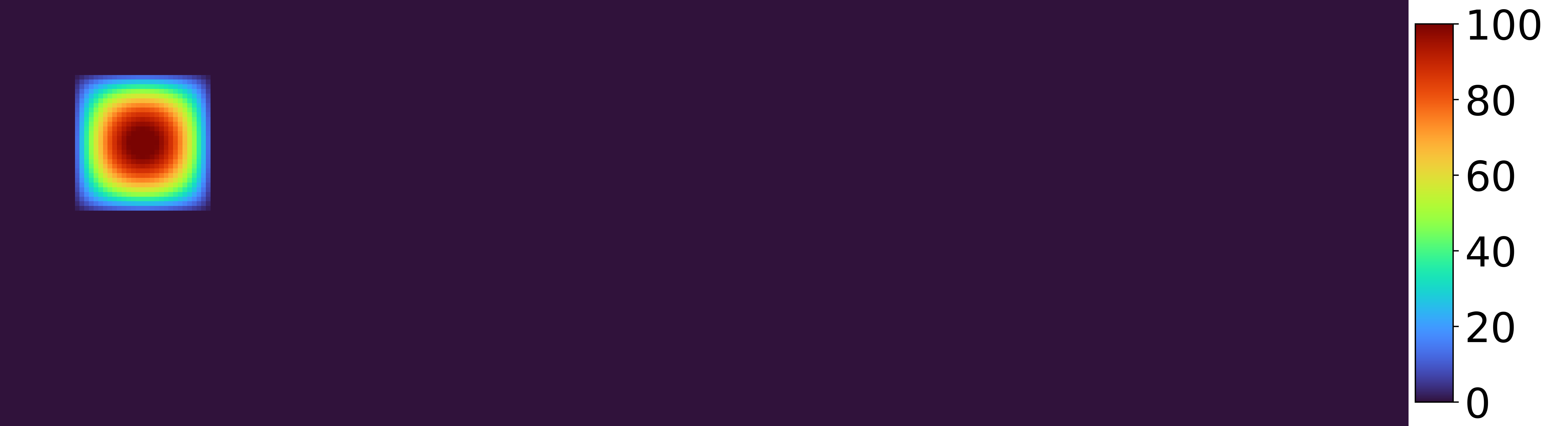}
		\hspace{2pt}
		\includegraphics[width = 6cm]{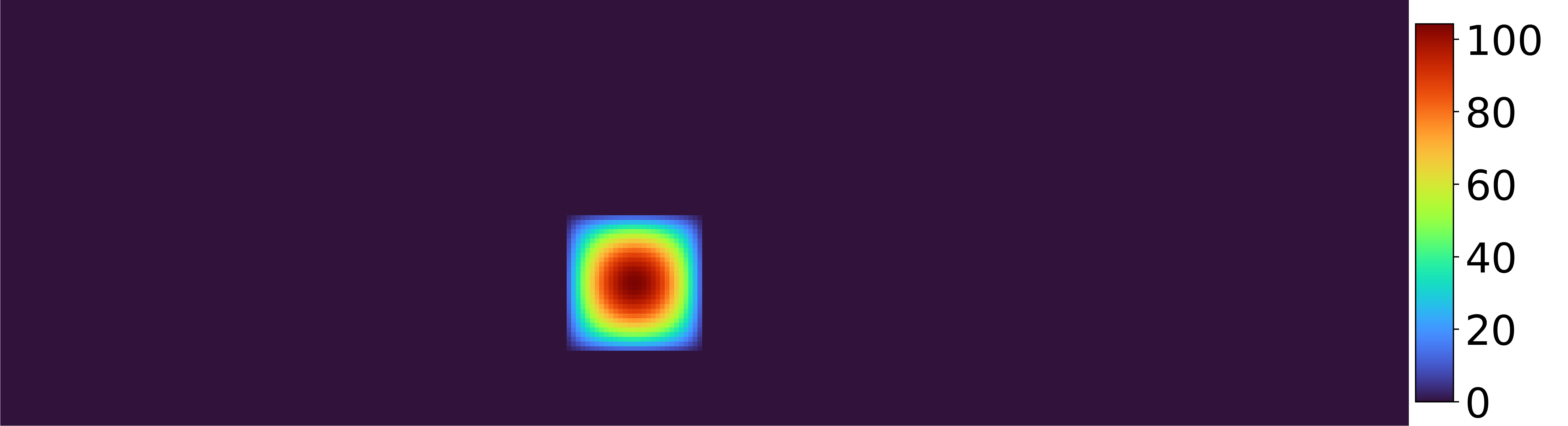}\\
		\vspace{3pt}
	\end{minipage}
	\begin{minipage}{13cm}
		\centering
		{\footnotesize solution at time point $t_i+0.1$ for $a=0$}\\
		\vspace{1pt}
		\includegraphics[width = 6cm]{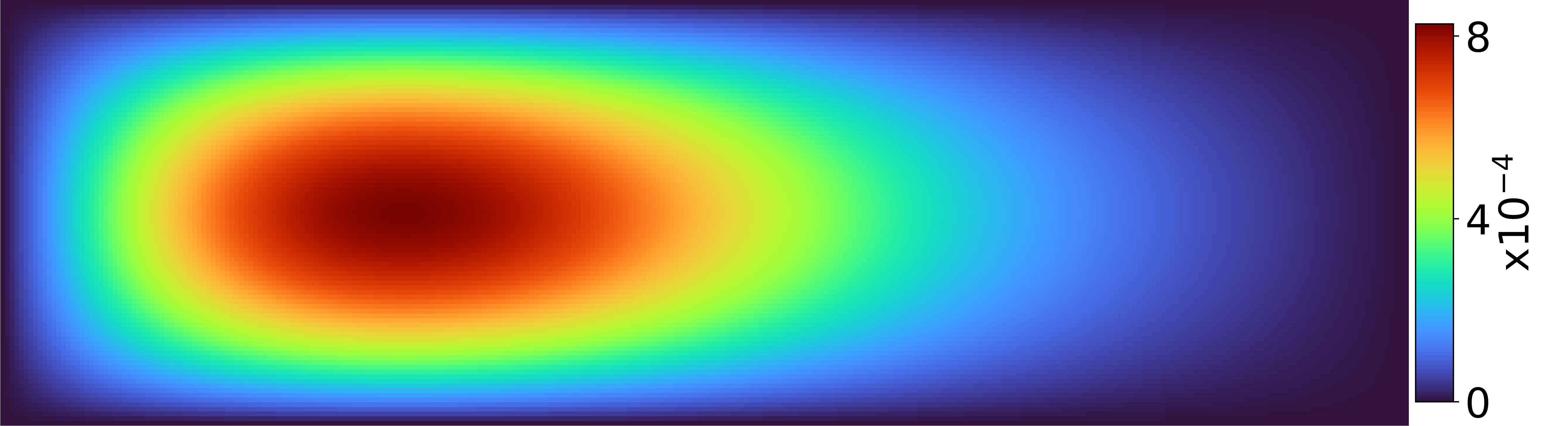}
		\hspace{2pt}
		\includegraphics[width = 6cm]{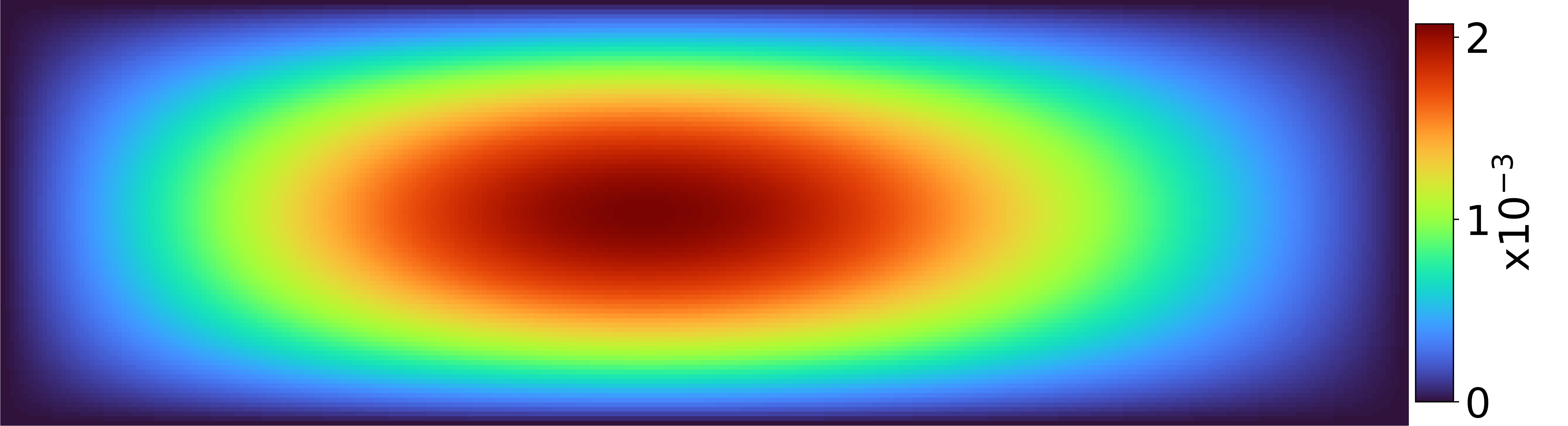}\\
		\vspace{3pt}
	\end{minipage}
	\begin{minipage}{13cm}
		\centering
		{\footnotesize solution at time point $t_i+0.1$ for $a=100$}\\
		\vspace{1pt}
		\includegraphics[width = 6cm]{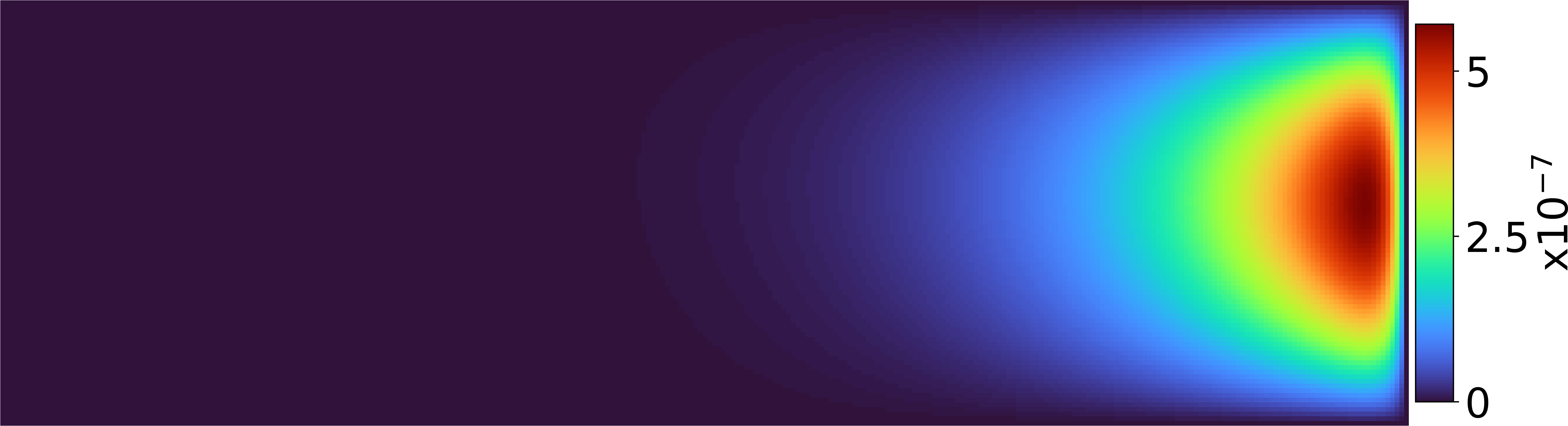}
		\hspace{2pt}
		\includegraphics[width = 6cm]{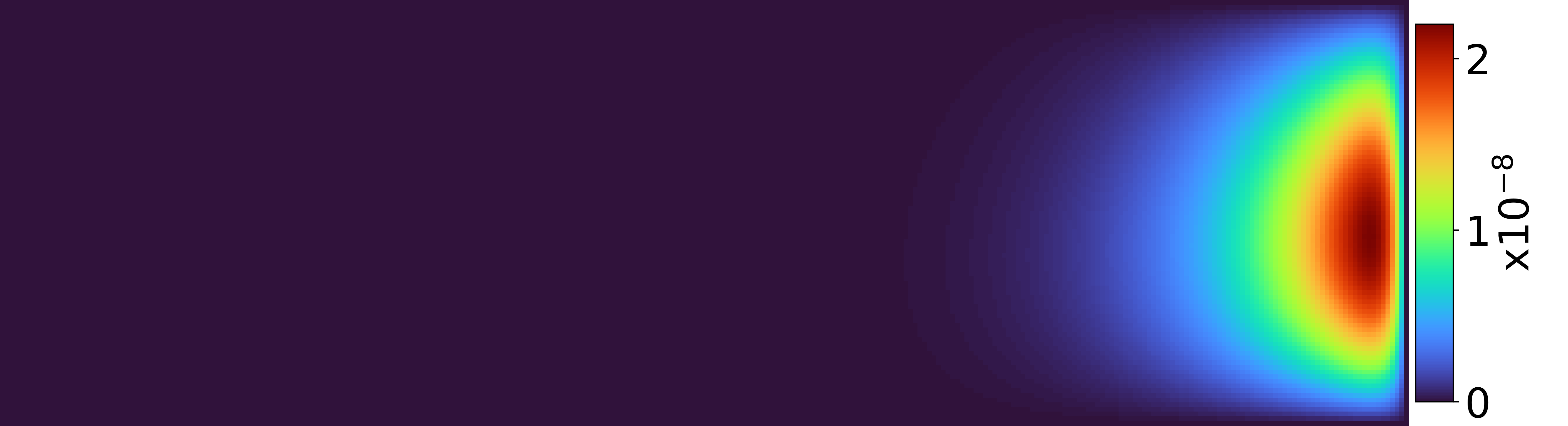}
	\end{minipage}
	\vspace{2pt}
	\caption{\footnotesize Example 3a: Local solutions for different initial conditions and advection $a=0$ vs. $a=100$.}
	\label{figure_advection_example_solutions}
\end{figure}

Here, we consider an advection-diffusion problem (\cref{PDE_ex} with $ c \equiv 0$) and first investigate the following numerical experiment, which we refer to as Example 3a: We choose $I=(0,5)$, $D=(0,1)\times (0,0.3)$, $\Sigma_N=\emptyset$, and discretize the spatial domain $D$ with a regular quadrilateral mesh with mesh size $1/300$ in both directions. For the implicit Euler method, we use an equidistant time step size of $1/100$. Furthermore, we consider the advection field $b=(b_1,\, 0)^\top $ for varying constants $b_1 \in \lb 0,10,25,50,100\rb$ and choose the conductivity coefficient $\kappa \equiv 1$. First, we observe in \cref{figure_advection_transfer_singular_vals} a very fast, exponential decay of the singular values of the transfer operators mapping arbitrary initial conditions at time point $t_i$ to the local solution at time point $t_i + 0.1$. Moreover, we see that the singular values are smaller for higher values of advection. This can be traced back to the fact that for a higher value of advection (i.e. $b_1=100$) initial conditions at time $t_i$ move quickly in $x$-direction and the corresponding local solutions at time $t_i+0.1$ have support only in the right half of the spatial domain close to the right boundary, see \cref{figure_advection_example_solutions} (bottom). In contrast, the plots in \cref{figure_advection_example_solutions} (middle) show that in the absence of advection (i.e. $b_1=0$) the local solutions corresponding to different initial conditions differ more from each other compared to the case $b_1=100$ and spread over the whole spatial domain. We can thus infer for this test case that the range of the transfer operator (i.e. the local solution space at a point of time) is extremely low-dimensional for higher values of advection and therefore well amenable to approximation.

\begin{figure}
	\begin{minipage}{13cm}
		\centering
		{\footnotesize \hspace{-0.5cm} solution at $t=0.5$ \hspace{4cm} solution at $t=1.8$}\\
		\vspace{2pt}
		\includegraphics[width = 6cm]{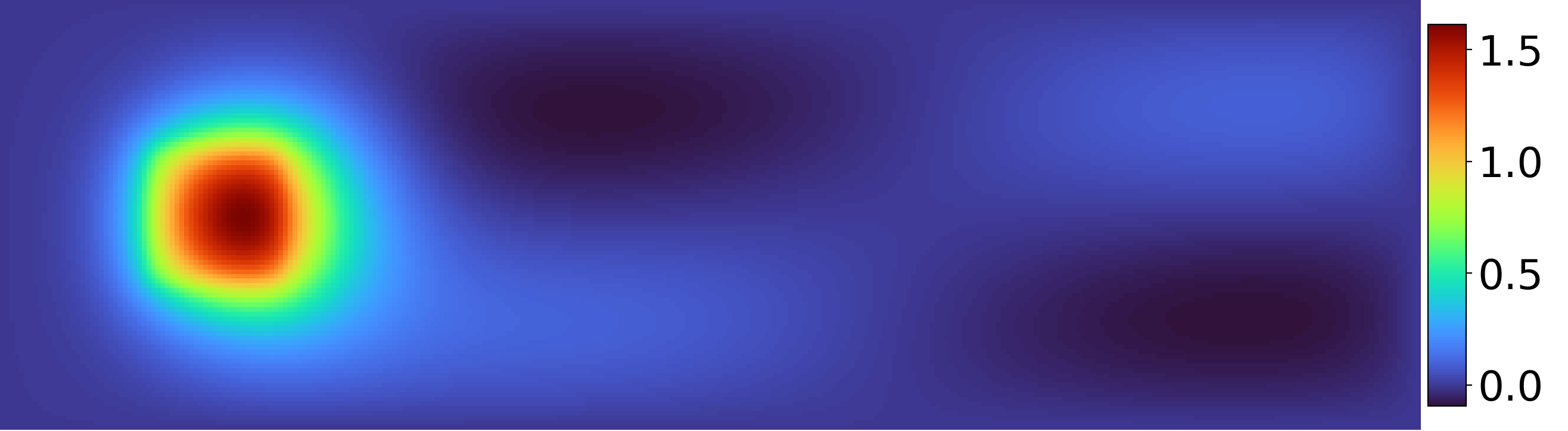}
		\hspace{2pt}
		\includegraphics[width = 6cm]{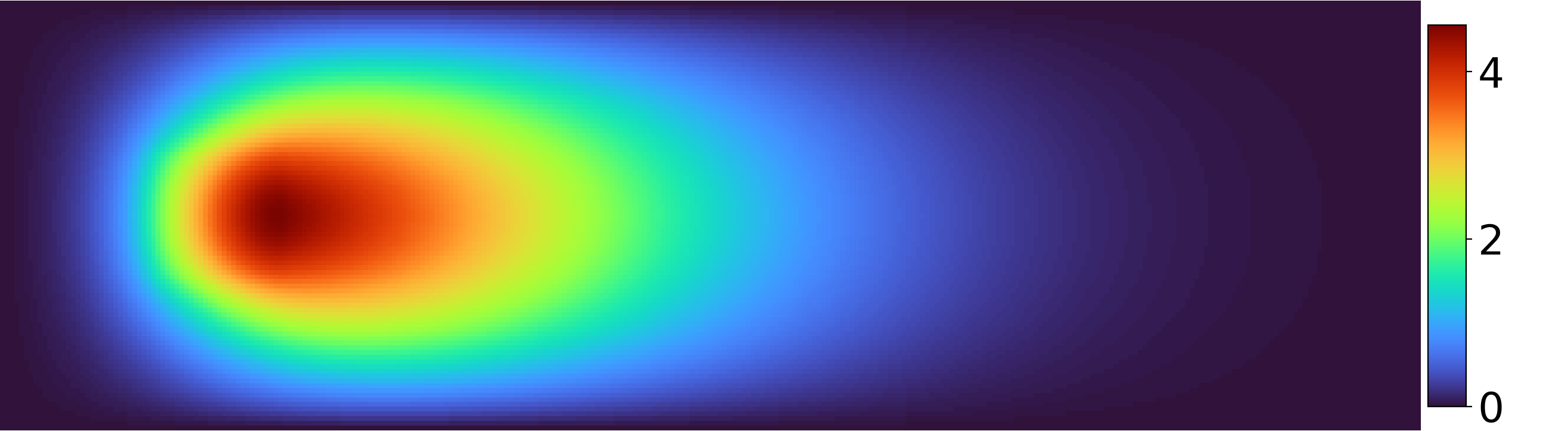}\\
		\vspace{3pt}
	\end{minipage}
	\begin{minipage}{13cm}
		\centering
		{\footnotesize \hspace{-0.6cm} solution at $t=3$ \hspace{4.2cm} solution at $t=5$}\\
		\vspace{2pt}
		\includegraphics[width = 6cm]{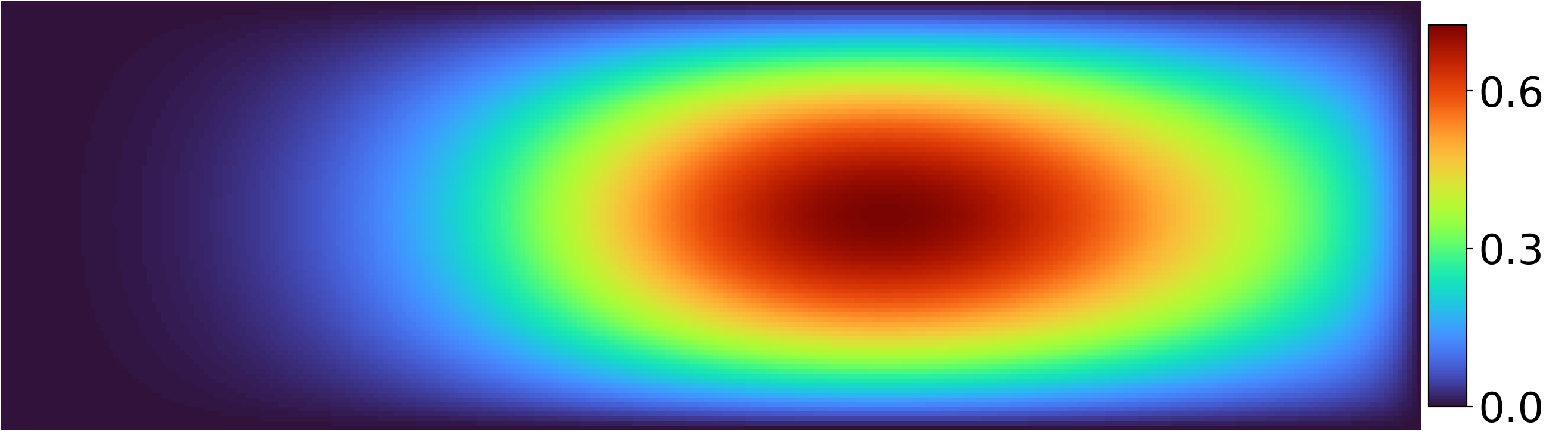}
		\hspace{2pt}
		\includegraphics[width = 6cm]{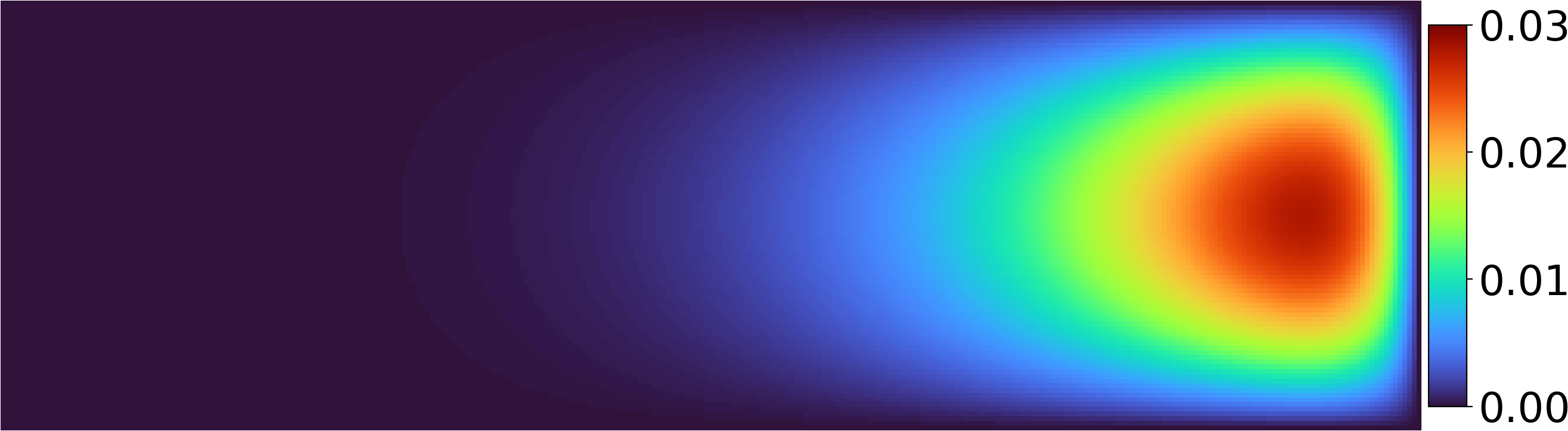}\\
		\vspace{3pt}
	\end{minipage}
	\vspace{2pt}
	\caption{\footnotesize Example 3b: Solution evaluated at different points in time. The shades in the middle and right part of the spatial domain at time point $t=0.5$ are the vanishing initial conditions.}
	\label{figure_advection_solutions}
\end{figure}

\begin{figure}
	\centering
	\begin{tikzpicture}
	\begin{axis}[
	title={$f(x,y)$},
	width=6cm,
	height=3cm,
	xmin=0,
	xmax=1,
	ymin=0,
	ymax=0.3,
	xlabel=$x$,
	ylabel=$y$,
	xtick = {0,1},
	ytick={0,0.3},
	tick style ={color=white},
	title style = {font=\footnotesize, yshift = -3},
	label style={font=\footnotesize},
	xlabel style = {yshift = 8},
	ylabel style = {yshift = -8},
	tick label style={font=\scriptsize}  
	]
	\draw[fill=lightgray] (0,0) rectangle (1,0.3);
	\draw[fill=teal] (0.1,0.1) rectangle (0.2,0.2);
	\end{axis}
	\end{tikzpicture}
	\begin{tikzpicture}
	\begin{axis}[
	title={$f(t)$},
	title style = {font=\footnotesize, yshift = -3},
	width=4cm,
	height=3cm,
	xmin=0,
	xmax= 5,
	ymin=0,
	ymax=42,
	xlabel=$t$,
	xtick = {0,5},
	xtick style = {white},
	ytick={20,40},
	ytick style = {white},
	xlabel style = {yshift = 8},
	label style={font=\footnotesize},
	tick label style={font=\scriptsize}  
	]
	\addplot[solid, teal, thick] table[x index=0, y index=1] {data_figures/advection_grid_t+_rhs_t+leverage_scores.dat};
	\end{axis}
	\end{tikzpicture}
	\begin{tikzpicture}
	\begin{axis}[
	title={Leverage scores},
	title style = {font=\footnotesize, yshift = -2},
	width=4cm,
	height=3cm,
	xmin=0,
	xmax= 5,
	ymin=0,
	ymax=0.012,
	xlabel=$t$,
	xtick = {0,5},
	xtick style = {white},
	ytick={0,0.01},
	ytick style = {white},
	label style={font=\footnotesize},
	xlabel style = {yshift = 8},
	tick label style={font=\scriptsize}  
	]
	\addplot[solid, purple, thick] table[x index=0, y index=2] {data_figures/advection_grid_t+_rhs_t+leverage_scores.dat};
	\end{axis}
	\end{tikzpicture}
	\caption{\footnotesize Example 3b: Source term $f(t,x,y)= f(t) f(x,y)$ and corresponding rank-$1$ leverage score probability distribution (identical to squared norms). Gray equates to $0$, green equates to $1$ (left).
	}
	\label{figure_advection_rhs}
\end{figure}
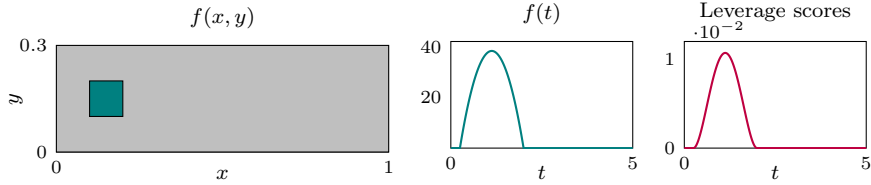

Next, we consider an advection-diffusion problem with a solution that is constantly moving in time as shown in \cref{figure_advection_solutions}, which we refer to as Example 3b. We choose $I$, $D$, $\Sigma_N$, and the discretization as in Example 3a. Moreover, we choose the advection field $b=(0.3,\, 0)^\top $, the conductivity coefficient $\kappa = 0.01$, the initial condition $u_0(x,y)=\sum_{i=1}^{3} \sin(i \pi x) \sin(i \pi y)$, and the source term $f(t,x,y)= f(t) f(x,y)$ as depicted in \cref{figure_advection_rhs} (left). For drawing time points in \cref{randomized_basis_generation}, we use the rank-$1$ leverage score probability distribution (see \cref{figure_advection_rhs} (right)) computed from the right-hand side matrix $\mathbf{F}$, whose columns are the right hand-side vectors $\mathbf{F}_{l}$ for $l=0,\ldots,500$ \cref{FE_matrices}. We emphasize that the latter distribution is identical to the squared norm probability distribution for this test case. As the solution is constantly moving in time (cf. \cref{figure_advection_solutions}), we additionally sample uniformly from the global time grid to capture the global advection in time. We denote the number of chosen time points by $\nrand^\text{rhs}$ and $\nrand^\text{advec}$, respectively.

First, we observe in \cref{figure_advection_global_singular_vals+nt} (left) that the singular values of the transfer operator decay exponentially, but more slowly compared to Example 2, 3a, and 4. This is mainly due to the lower diffusion compared to Example 2, 3a, and 4, which results in a slower propagation of (local) solutions in time. However, we see that the singular values decrease faster with an increasing local oversampling size and also observe in \cref{figure_advection_global_singular_vals+nt} (right) that the approximation accuracy increases with increasing $\nt$: While for $\nt=15$ $95\%$ of realizations yield an approximation with a relative $L^2(I,H^1(D))$-error above $10^{-2}$, for $\nt=30$ or $\nt=45$ already $99\%$ of realizations have a relative error below $10^{-3}$ or $2\cdot 10^{-4}$. For all other tests in this subsection we therefore choose $\nt = 30$ as a good trade-off between approximation quality and computational costs.

\begin{figure}
	\centering
	\begin{tikzpicture}
	\begin{semilogyaxis}[
	width=4.5cm,%0.84\textwidth,
	height=5cm,
	xmin=0,
	xmax=35,
	ymin=3e-4,
	ymax=1e-0,
	legend style={at={(1.02,1)},anchor=north west,font=\footnotesize},
	grid=both,
	grid style={line width=.1pt, draw=gray!40},
	major grid style={line width=.2pt,draw=gray!70},
	xtick={0,5,10,15,20,25,30,35},
	minor xtick = {},
	ytick={1e-4,1e-3,1e-2,1e-1,1e-0},
	xlabel= $n$,
	label style={font=\footnotesize},
	xlabel style = {yshift = 3},
	tick label style={font=\footnotesize}  
	]
	\addplot+[thick, solid, purple, mark = none] table[x index = 0,y index=1]{data_figures/advection_global_transfer_singular_vals_nt=15-30-45.dat};
	\addplot+[ thick, mark=asterisk, mark repeat = 2,mark size = 2.4pt, blue, densely dotted, mark options={solid}] table[x index = 0,y index=2]{data_figures/advection_global_transfer_singular_vals_nt=15-30-45.dat};
	\addplot+[ thick, mark=diamond, mark repeat = 2, densely dashdotted, cyan, mark options={solid}] table[x index = 0,y index=3]{data_figures/advection_global_transfer_singular_vals_nt=15-30-45.dat};
	\legend{$\sigma^{(n+1)}_{t_i \rightarrow t_i+ 0.15}$,$\sigma^{(n+1)}_{t_i \rightarrow t_i+ 0.3}$,$\sigma^{(n+1)}_{t_i \rightarrow t_i+ 0.45}$}
	\end{semilogyaxis}
	\end{tikzpicture}
	\hspace{0.1cm}
	\begin{tikzpicture}
	\begin{semilogyaxis}[
	width=3.8cm,
	height=5cm,
	xmin=-0.5,
	xmax=2.5,
	ymin=1e-5,
	ymax=1e-1,
	legend style={at={(1.02,1)},anchor=north west,font=\footnotesize},
	grid=both,
	grid style={line width=.1pt, draw=gray!40},
	major grid style={line width=.2pt,draw=gray!70},
	xtick={0,1,2,3},
	ytick={1e-5,1e-4, 1e-3, 1e-2,1e-1,1e-0},
	xticklabels={15, 30, 45},
	xlabel= $\nt$,
	xlabel style = {yshift = 3},
	ylabel= relative $L^2(H^1)$-error,
	label style={font=\footnotesize},
	tick label style={font=\footnotesize}  
	]
	\addplot[only marks, brown, mark=*, mark size=1.8pt] table[x expr=\coordindex, y index=0] {data_figures/advection_quantiles_nt=15-30-45_rel_L2H1.dat};
	\addplot[only marks, blue, mark=x, mark size=2.5pt,thick] table[x expr=\coordindex, y index=1] {data_figures/advection_quantiles_nt=15-30-45_rel_L2H1.dat};
	\addplot[only marks,red, mark=pentagon*, mark size=2pt] table[x expr=\coordindex, y index=2] {data_figures/advection_quantiles_nt=15-30-45_rel_L2H1.dat};
	\addplot[only marks, black,mark=star, mark size=2.5pt,thick] table[x expr=\coordindex, y index=3] {data_figures/advection_quantiles_nt=15-30-45_rel_L2H1.dat};
	\addplot[only marks,cyan,mark=triangle*, mark size=2.5pt] table[x expr=\coordindex, y index=4] {data_figures/advection_quantiles_nt=15-30-45_rel_L2H1.dat};
	\addplot[only marks, orange, mark=asterisk, mark size=2.5pt,thick] table[x expr=\coordindex, y index=5] {data_figures/advection_quantiles_nt=15-30-45_rel_L2H1.dat};
	\addplot[only marks,teal, mark=diamond*, mark size=2.5pt] table[x expr=\coordindex, y index=6] {data_figures/advection_quantiles_nt=15-30-45_rel_L2H1.dat};
	\addplot[only marks, gray, mark= oplus*, mark size=1.8pt] table[x expr=\coordindex, y index=7] {data_figures/advection_quantiles_nt=15-30-45_rel_L2H1.dat};
	\legend{max,99,95,75,50,25,5,min};
	\end{semilogyaxis}
	\end{tikzpicture}
	\caption{\footnotesize Example 3b: Singular value decay of transfer operators $T_{t_i \rightarrow t_i+ 0.15}$, $0\leq i \leq 486$, $T_{t_i \rightarrow t_i+ 0.3}$, $0\leq i \leq 471$, and $T_{t_i \rightarrow t_i+ 0.45}$, $0\leq i \leq 466$, (left). Quantiles of relative $L^2(I,H^1(D))$-error for varying numbers of $\nt$, $k = \nt-6$, $\nrand^\text{rhs}=10$, $\nrand^\text{advec}=20$, $\tol = 10^{-8}$, and $10.000$ realizations (right). Singular values decay slower due to lower diffusion compared to Example 2, 3a, and 4.
	}
	\label{figure_advection_global_singular_vals+nt}
\end{figure}
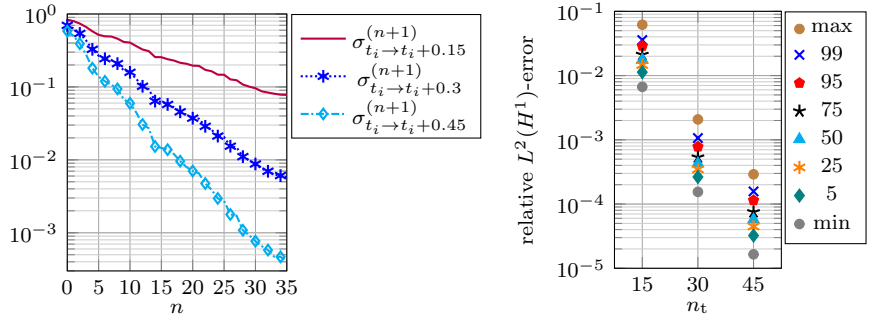

Next, we investigate how the approximation accuracy is influenced by the number of time points selected from the uniform sampling distribution and how many samples are required in order to sufficiently capture the advection in the global problem. In \cref{figure_advection_global_nrand_coeff+k} (left) we see that for $\nrand^\text{advec}= 0$ or $\nrand^\text{advec}=5$ in $95\%$ of cases the relative $L^2(I,H^1(D))$-error is above $2\cdot 10^{-2}$ or $7\cdot 10^{-3}$, while for $\nrand^\text{advec}= 20$ the relative error is already below $10^{-3}$ in $99\%$ of cases. If we choose $\nt>20$ we see in \cref{figure_advection_global_nrand_coeff+k} (left) that the approximation accuracy still slightly improves compared to smaller $\nrand^\text{advec}$. We therefore choose $\nrand^\text{advec}=20$ in all other tests in this subsection.
	
Finally, we test how the approximation accuracy depends on the number of collected snapshots determined via $k$. In \cref{figure_advection_global_nrand_coeff+k} (right) we observe that the approximation quality significantly improves if we collect not only solution snapshots at local end time points ($k = 30$), but at the (locally) last three, five, or seven time points ($k = 28,26,24$) as, for instance, for $k=30$ all realizations have a relative $L^2(I,H^1(D))$-error above $3\cdot 10^{-2}$, while for $k=28$ or $k=24$ already $99\%$ of realizations have a relative error below $4\cdot 10^{-3}$ or $10^{-3}$. This can be explained by the fact that by choosing a smaller $k$ the propagation of the solution can be captured more accurately. For $k = 22$ or $k = 20$ we observe in \cref{figure_advection_global_nrand_coeff+k} (right) only very slight improvements in the approximation quality compared to $k = 24$. We therefore choose $k = \nt-6$ in all other tests in this subsection.

\begin{figure}
	\begin{tikzpicture}
	\begin{semilogyaxis}[
	width=6.5cm,
	height=5cm,
	xmin=-0.5,
	xmax=8.5,
	ymin=3e-5,
	ymax=1e-1,
	legend style={at={(1.02,1)},anchor=north west,font=\footnotesize},
	grid=both,
	grid style={line width=.1pt, draw=gray!40},
	major grid style={line width=.2pt,draw=gray!70},
	xtick={0,1,2,3,4,5,6,7,8},
	ytick={1e-5,1e-4, 1e-3, 1e-2,1e-1,1e-0},
	xticklabels={0,5,10,15,20,25,30,35,40},
	xlabel= $\nrand^\text{advec}$,
	xlabel style = {yshift = 4},
	ylabel= relative $L^2(H^1)$-error,
	label style={font=\footnotesize},
	tick label style={font=\footnotesize}  
	]
	\addplot[only marks, brown, mark=*, mark size=1.8pt] table[x expr=\coordindex, y index=0] {data_figures/advection_quantiles_nrand_coeff_rel_L2H1.dat};
	\addplot[only marks, blue, mark=x, mark size=2.5pt,thick] table[x expr=\coordindex, y index=1] {data_figures/advection_quantiles_nrand_coeff_rel_L2H1.dat};
	\addplot[only marks,red, mark=pentagon*, mark size=2pt] table[x expr=\coordindex, y index=2] {data_figures/advection_quantiles_nrand_coeff_rel_L2H1.dat};
	\addplot[only marks, black,mark=star, mark size=2.5pt,thick] table[x expr=\coordindex, y index=3] {data_figures/advection_quantiles_nrand_coeff_rel_L2H1.dat};
	\addplot[only marks,cyan,mark=triangle*, mark size=2.5pt] table[x expr=\coordindex, y index=4] {data_figures/advection_quantiles_nrand_coeff_rel_L2H1.dat};
	\addplot[only marks, orange, mark=asterisk, mark size=2.5pt,thick] table[x expr=\coordindex, y index=5] {data_figures/advection_quantiles_nrand_coeff_rel_L2H1.dat};
	\addplot[only marks,teal, mark=diamond*, mark size=2.5pt] table[x expr=\coordindex, y index=6] {data_figures/advection_quantiles_nrand_coeff_rel_L2H1.dat};
	\addplot[only marks, gray, mark= oplus*, mark size=1.8pt] table[x expr=\coordindex, y index=7] {data_figures/advection_quantiles_nrand_coeff_rel_L2H1.dat};
	\end{semilogyaxis}
	\end{tikzpicture}
	\begin{tikzpicture}
	\begin{semilogyaxis}[
	width=5cm,
	height=5cm,
	xmin=-0.5,
	xmax=5.5,
	ymin=1e-4,
	ymax=2e-1,
	legend style={at={(1.05,1)},anchor=north west,font=\footnotesize},
	grid=both,
	grid style={line width=.1pt, draw=gray!40},
	major grid style={line width=.2pt,draw=gray!70},
	xtick={0,1,2,3,4,5},
	ytick={1e-5,1e-4, 1e-3, 1e-2,1e-1,1e-0},
	xticklabels={20,22,24,26,28,30},
	xlabel= $k$,
	label style={font=\footnotesize},
	tick label style={font=\footnotesize}  
	]
	\addplot[only marks, brown, mark=*, mark size=1.8pt] table[x expr=\coordindex, y index=0] {data_figures/advection_quantiles_k_rel_L2H1.dat};
	\addplot[only marks, blue, mark=x, mark size=2.5pt,thick] table[x expr=\coordindex, y index=1] {data_figures/advection_quantiles_k_rel_L2H1.dat};
	\addplot[only marks,red, mark=pentagon*, mark size=2pt] table[x expr=\coordindex, y index=2] {data_figures/advection_quantiles_k_rel_L2H1.dat};
	\addplot[only marks, black,mark=star, mark size=2.5pt,thick] table[x expr=\coordindex, y index=3] {data_figures/advection_quantiles_k_rel_L2H1.dat};
	\addplot[only marks,cyan,mark=triangle*, mark size=2.5pt] table[x expr=\coordindex, y index=4] {data_figures/advection_quantiles_k_rel_L2H1.dat};
	\addplot[only marks, orange, mark=asterisk, mark size=2.5pt,thick] table[x expr=\coordindex, y index=5] {data_figures/advection_quantiles_k_rel_L2H1.dat};
	\addplot[only marks,teal, mark=diamond*, mark size=2.5pt] table[x expr=\coordindex, y index=6] {data_figures/advection_quantiles_k_rel_L2H1.dat};
	\addplot[only marks, gray, mark= oplus*, mark size=1.8pt] table[x expr=\coordindex, y index=7] {data_figures/advection_quantiles_k_rel_L2H1.dat};
	\legend{max,99,95,75,50,25,5,min};
	\end{semilogyaxis}
	\end{tikzpicture}
	\caption{\footnotesize Example 3b: Quantiles of relative $L^2(I,H^1(D))$-error for varying numbers of $\nrand^\text{advec}$, $\nt=30$, $k=24$, $\nrand^\text{rhs}=10$, $\tol = 10^{-8}$, and $10.000$ realizations (left). Quantiles of relative $L^2(I,H^1(D))$-error for $\nt=30$, varying numbers of $k$, $\nrand^\text{rhs}=10$, $\nrand^\text{advec}=20$, $\tol = 10^{-8}$, and $10.000$ realizations (right).
	}
	\label{figure_advection_global_nrand_coeff+k}
\end{figure}
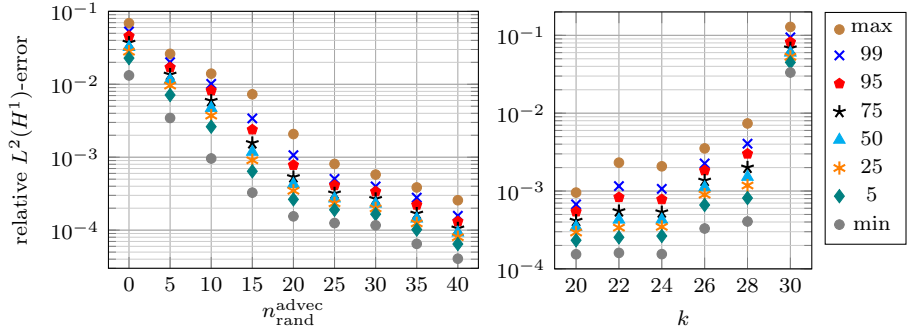

The computational costs of the randomized approach exceed the costs of the POD for this particular test case. However, we highlight that the computation of the local solution trajectories in \cref{randomized_basis_generation} is embarrassingly parallel while for the POD the global solution trajectory has to be computed in a sequential manner. Depending on the employed computer architecture the randomized approach can thus still lead to a significant speed-up compared to the standard POD approach.

\subsection{Problem with a time-dependent permeability coefficient}
\label{subsec_num_ex_4}

In this subsection we consider a numerical experiment including the real-world permeability coefficient $\kappa_0$ taken from the SPE10 benchmark problem \cite{SPE10}, see \cref{figure_coeff_solutions_SPE10}. We refer to this experiment as Example $4$. In \cref{figure_coeff_solutions_SPE10} we observe that the solution trajectory of the problem is quite complex due to different configurations and combinations of permeability and inflow into the domain depicted in \cref{figure_rhs_channels_SPE10}.

In detail, we consider the heat equation (\cref{PDE_ex} with $b\equiv c \equiv 0$), choose $I=(0,10)$ and $D=(0,2.2)\times(0,0.6)$ with Dirichlet and Neumann boundary as shown in \cref{figure_rhs_channels_SPE10} (left) and discretize the spatial domain $D$ with a regular quadrilateral mesh with mesh size $1/100$ in both directions. For the implicit Euler method, we use an equidistant time step size of $1/50$. We impose Neumann boundary conditions $g_N(t,x,y)=g_N(t)$ for $(x,y)\in (0.4,1.8)\times \lb 0.6\rb$ modeling a time-dependent inflow as depicted in \cref{figure_rhs_channels_SPE10} (middle) and $g_N=0$ elsewhere in $I\times\Sigma_N$. The permeability coefficient $\kappa(t,x,y) = \kappa_0(x,y) + \kappa_1(t) \cdot \kappa_1(x,y) + \kappa_2(t) \cdot \kappa_2(x,y)$ is given by a sum of the permeability field $\kappa_0$ from \cite{SPE10} shown in \cref{figure_coeff_solutions_SPE10} (top, left) and high conductivity channels $\kappa_1$ and $\kappa_2$ that are turned on and off in time as depicted in \cref{figure_rhs_channels_SPE10} (middle). Moreover, the initial conditions are given by $u_0(x,y)= 1$ for $(x,y)\in (0.5,0.7)\times (0.3,0.4)$ and $u_0(x,y)= 0$ else. As both the permeability $\kappa$ and the inflow $g_N$ vary in time, we sample from two probability distributions simultaneously to draw time points in \cref{randomized_basis_generation} (cf. the discussion in \cref{subsubsec_multiple_probs}). For this purpose, we employ the rank-$3$ leverage score probability distribution computed from the matrix, whose columns contain the values of $\kappa(t_l,\cdot,\cdot)$ for all spatial elements at time point $t_l$ for $l=0,\ldots,500$. In addition, we use the rank-$1$ leverage score probability distribution computed from the right hand-side matrix $\mathbf{F}$, whose columns are the right hand-side vectors $\mathbf{F}_{l}$ for $l=0,\ldots,500$ \cref{FE_matrices}. Both distributions are depicted in \cref{figure_rhs_channels_SPE10} (right).

\begin{figure}
	\begin{minipage}{13cm}
		{\footnotesize \hspace{1.5cm} $\kappa_0(x,y)$\hspace{2.4cm} solution at $t=0.1$\hspace{2.1cm} solution at $t=2$}
		\vspace{0.03cm}
		\\
		\includegraphics[width=4.25cm, height= 1.5cm]{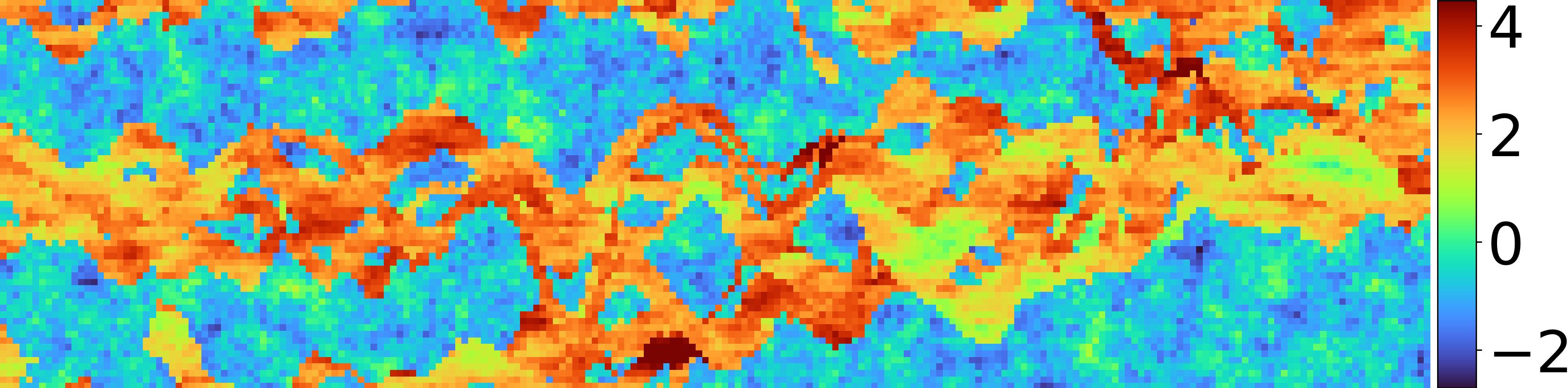}
		\includegraphics[width=4.25cm, height= 1.5cm]{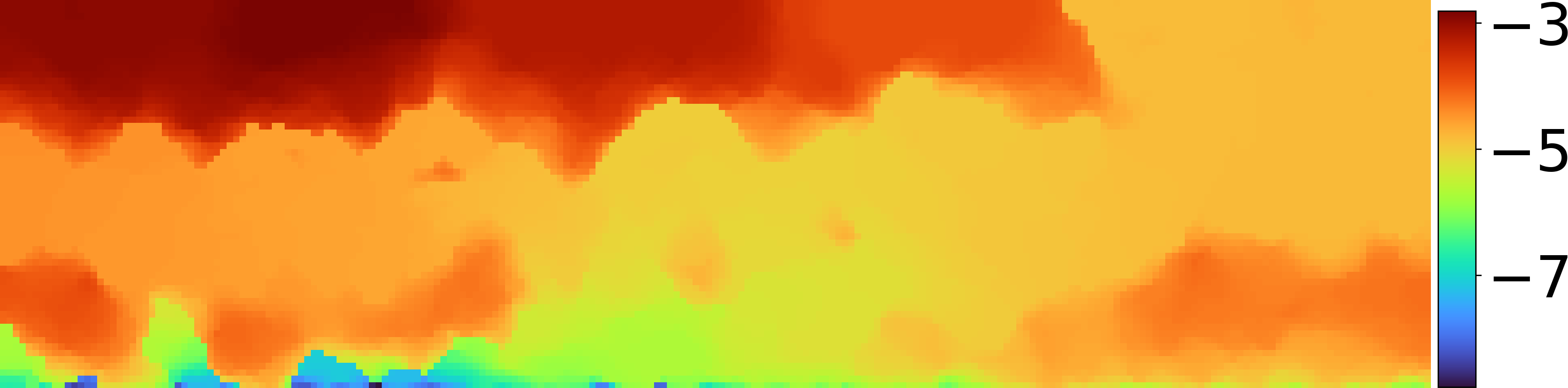}
		\includegraphics[width=4.25cm, height= 1.5cm]{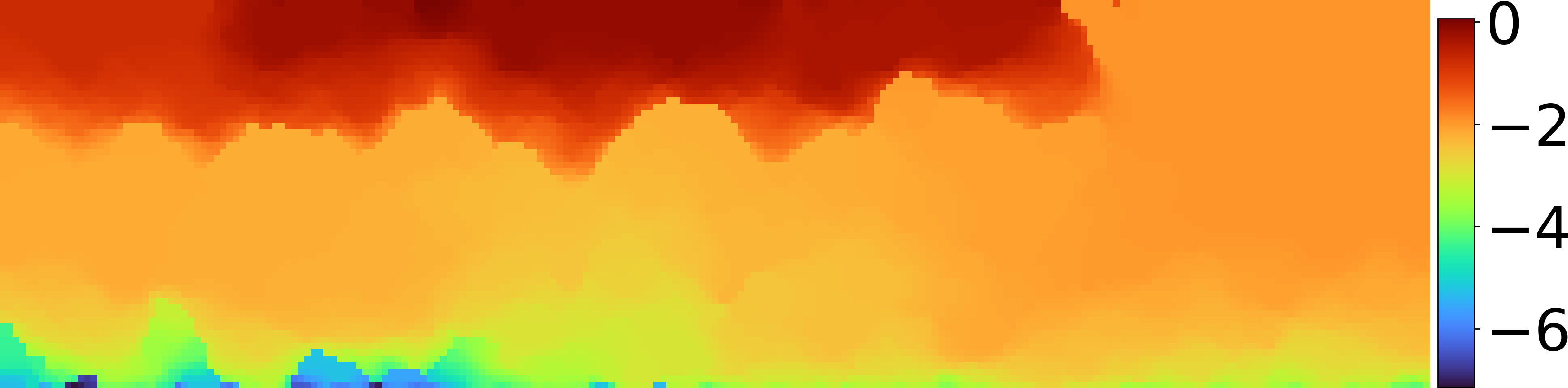}		
	\end{minipage}
	\vspace{0.05cm}\\
	\begin{minipage}{13cm}
		{\footnotesize \hspace{0.75cm} solution at $t=4$\hspace{2cm} solution at $t=7.6$\hspace{2.05cm} solution at $t=9.5$}
		\vspace{0.1cm}
		\\
		\includegraphics[width=4.25cm, height= 1.5cm]{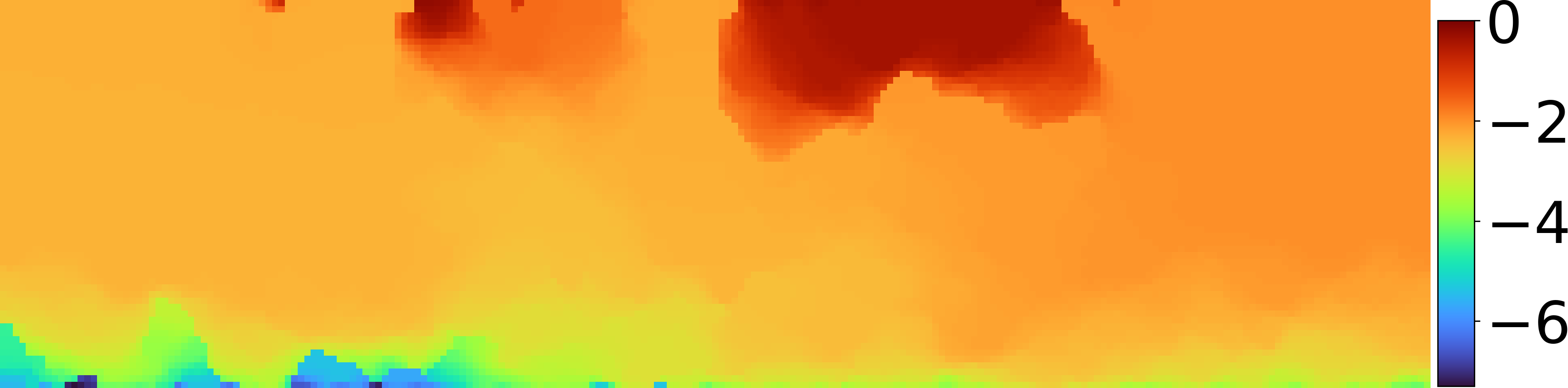}
		\includegraphics[width=4.25cm, height= 1.5cm]{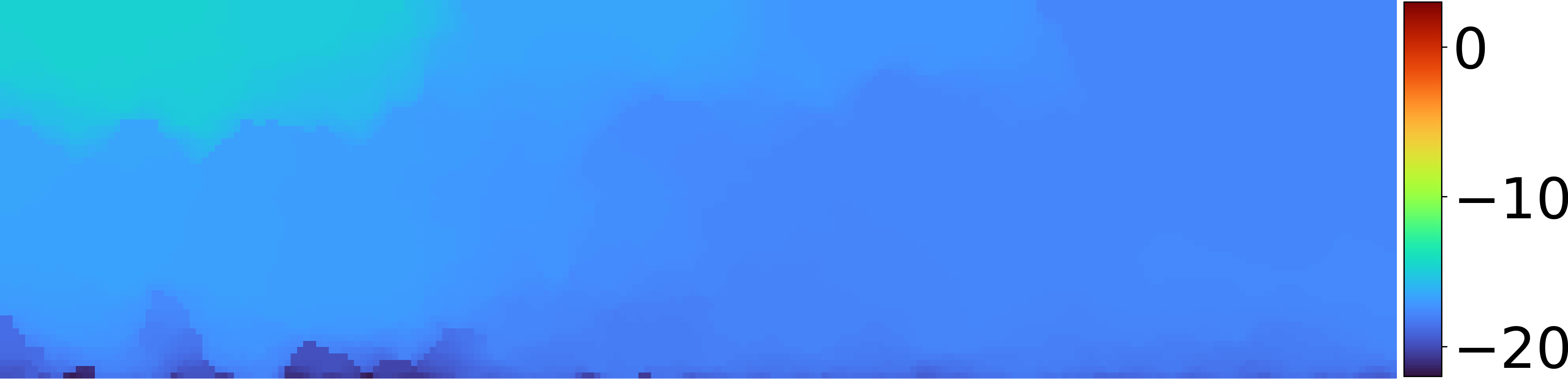}
		\includegraphics[width=4.25cm, height= 1.5cm]{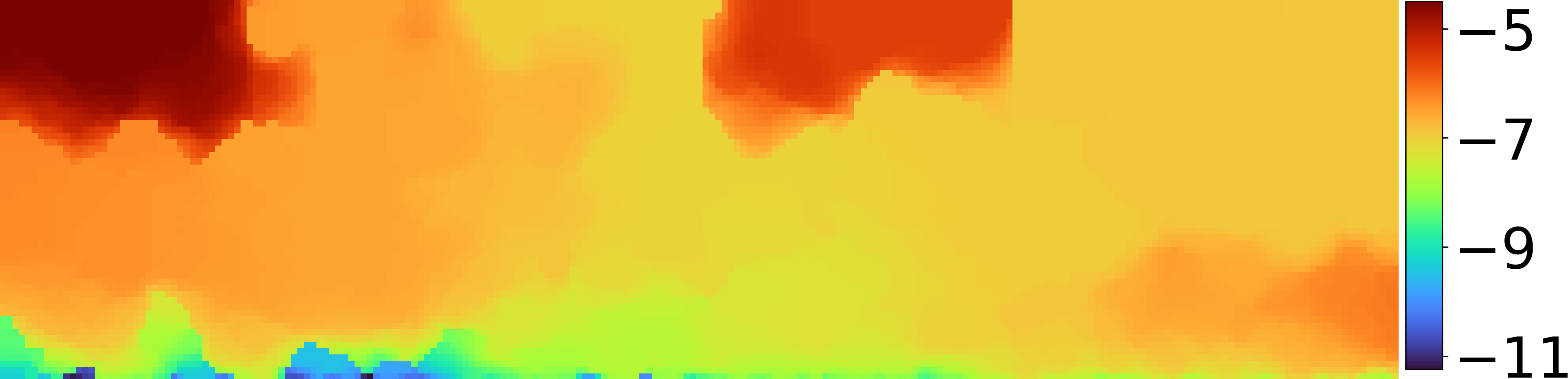}		
	\end{minipage}
	\vspace{2pt}
	\caption{\footnotesize Example 4: Permeability field $\kappa_0$ from \cite{SPE10} and solution evaluated at different points in time, plotted in logarithmic values to the base of $10$.}
	\label{figure_coeff_solutions_SPE10}
\end{figure}

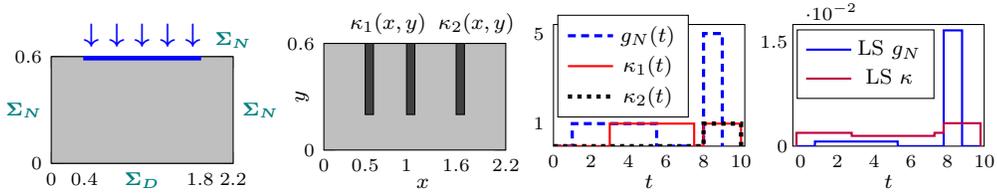
\begin{figure}
	\hspace{-0.25cm}
	\begin{tikzpicture}
	\begin{axis}[
	name = MyAxis,
	width=4cm,
	height=3cm,
	xmin=0,
	xmax= 2.2,
	ymin=0,
	ymax=0.6,
	xtick = {0,0.4,1.1,1.8,2.2},
	xticklabels = {0,0.4,{\textcolor{teal}{$\boldsymbol{\Sigma_D}$}},1.8,2.2},
	ytick={0,0.3,0.6},
	yticklabels={0,{\textcolor{teal}{$\boldsymbol{\Sigma_N}$}},0.6},
	x label style={yshift=4,xshift=3},
	y label style={yshift=-12},
	title style = {font=\footnotesize, yshift = -4},
	label style={font=\footnotesize},
	tick label style={font=\scriptsize}  
	]
	\draw[fill=lightgray] (0,0) rectangle (2.2,0.6);
	\draw[blue, fill=blue, thick] (0.4,0.585) rectangle (1.8,0.61);
	\end{axis}
	\node[above , align=center] at (MyAxis.north) { \textcolor{blue}{$\boldsymbol{\downarrow \,\, \downarrow \,\, \downarrow \,\, \downarrow \,\, \downarrow}$}};
	\node[above, align=center] at (MyAxis.north east) {\scriptsize \textcolor{teal}{$\boldsymbol{\Sigma_N}$}};
	\node[right, align=center] at (MyAxis.east) { \scriptsize  \textcolor{teal}{$\boldsymbol{\Sigma_N}$}};
	\end{tikzpicture}
	\hspace{-0.35cm}
	\begin{tikzpicture}
	\begin{axis}[
	title = {$\quad\;\;\kappa_1(x,y)\;\; \kappa_2(x,y)$},
	width=4cm,
	height=3cm,
	xmin=0,
	xmax= 2.2,
	ymin=0,
	ymax=0.6,
	xlabel = $x$,
	ylabel = $y$,
	xtick = {0,0.5,1,1.6,2.2},
	ytick={0,0.6},
	x label style={yshift=4,xshift=3},
	xtick style = {color=white},
	y label style={yshift=-15},
	title style = {font=\footnotesize, yshift = -6},
	label style={font=\footnotesize},
	tick label style={font=\scriptsize} 
	]
	\draw[fill=lightgray] (0,0) rectangle (2.2,0.6);
	\draw[fill=darkgray] (0.5,0.2) rectangle (0.6,0.6);
	\draw[fill=darkgray] (1,0.2) rectangle (1.1,0.6);
	\draw[fill=darkgray] (1.6,0.2) rectangle (1.7,0.6);
	\end{axis}
	\end{tikzpicture}
	\hspace{-0.275cm}
	\begin{tikzpicture}
	\begin{axis}[
	width=4.125cm,
	height=3.2cm,
	xmin=0,
	xmax= 10.2,
	ymin=0,
	ymax=5.4,
	xtick = {},
	ytick={1,5},
	xlabel = $t$,
	x label style={yshift=4},
	xtick style = {white},
	ytick style = {white},
	title style = {font=\footnotesize, yshift = -4},
	legend style={at={(0.02,1.08)},anchor=north west,font=\footnotesize},
	label style={font=\footnotesize},
	tick label style={font=\scriptsize}  
	]
	\addplot[const plot, no marks,  densely dashed, very thick, blue] coordinates {(0,0) (1,1) (5.5,0) (8,5) (9,0)} node[above,pos=.57]{};
	\addplot[const plot, no marks, thick, red] coordinates {(0,0) (3,1) (7.5,0) (8,1) (10,0)} node[above,pos=.57]{};
	\addplot[const plot, no marks, dotted, ultra thick, black] coordinates {(0,0) (8,0) (8,1) (10,1) (10,0)} node[above,pos=.57]{};
	\legend{$g_N(t)$,$\kappa_1(t)$,$\kappa_2(t)$},
	\end{axis}
	\end{tikzpicture}	
	\hspace{-0.35cm}
	\begin{tikzpicture}
	\begin{axis}[
	name=plot1,
	width = 4.125cm,
	height=3.2cm,
	xmin=-0.2,
	xmax= 10.2,
	ymin=0,
	ymax=0.0175,
	ytick={0,0.015},
	xlabel = $t$,
	x label style={yshift=4},
	xtick style = {white},
	ytick style = {white},
	title style = {font=\footnotesize, yshift = -4},
	legend style={at={(0.02,0.96)},anchor=north west,font=\footnotesize},
	label style={font=\footnotesize},
	tick label style={font=\scriptsize,xshift=1.5}  
	]
	\addplot[const plot, no marks, thick, blue] coordinates {(0,0) (1,0.00066622) (5.5,0) (8,0.01665556) (9,0) (10,0)} node[above,pos=.57]{};
	\addplot[const plot, no marks, thick, purple] coordinates {(0,0) (0,0.00191571) (3,0.00147493) (7.5,0.00191571) (8,0.00330033) (10,0)} node[above,pos=.57]{};
	\legend{LS $g_N$, LS $\kappa$}
	\end{axis}
	\end{tikzpicture}
	\caption{\footnotesize Example 4: Dirichlet and Neumann boundary $\Sigma_D$, $\Sigma_N$ (left), high conductivity chan-\\nels $ \kappa_1(t)\cdot \kappa_1(x,y) + \kappa_2(t)\cdot \kappa_2(x,y)$, and Neumann boundary data $g_N(t,x,y) = g_N(t)$ for $(x,y)\in (0.4,1.8)\times \lb 0.6 \rb$, $g_N(t,x,y) = 0$ else (middle). Dark gray equates to $10^3$, light gray to $0$. Rank-$1$ ´ (LS) associated with $g_N$ and rank-$3$ LS corresponding to $\kappa = \kappa_0 + \kappa_1 + \kappa_2$ (right).}
	\label{figure_rhs_channels_SPE10}
\end{figure}

First, we test how the approximation accuracy depends on the local oversampling size $\nt$. For all tested sizes of $\nt$, and especially for a small oversampling size of $\nt=10$, we observe in \cref{figure_SPE10_quantiles_POD} (left) that in $88\%$ of cases the relative $L^2(I,H^1(D))$-error is below $2\cdot 10^{-2}$ and the algorithm succeeds in detecting all different configurations of the time-dependent data functions and thus all different shapes of the solution. Nevertheless, as already observed for Example 2, we see in \cref{figure_SPE10_quantiles_POD} (middle) that for $\nt=10$ the reduced basis is significantly larger compared to $\nt=15\,(20,25)$ as, for instance, in $95\%$ of cases the reduced dimension is larger than or equal to $59$ for $\nt=10$, while for $\nt=15\, (20, 25)$ it is smaller than or equal to $50\, (41, 36)$ in $95\%$ of cases. The results thus confirm the findings from \cref{subsec_num_ex_1} for this test case and in the following we choose $\nt=15$.

Next, we compare the relative $L^2(t)$-approximation error for one realization of \cref{randomized_basis_generation} with the relative $L^2(t)$-error of the approximation via POD on the solution trajectory of the first $315$ of $500$ time steps. As already motivated in \cref{subsec_num_ex_1}, we focus on equaling the computational budget based on the time stepping ($(\nrand +1)\cdot \nt = (20+1)\cdot 15 = 315$). In \cref{figure_SPE10_quantiles_POD} (right) we observe that the POD yields a much larger error in the time interval $(8,10)$ compared to the randomized approach and is not able to detect the high conductivity channel $\kappa_2$. The randomized approach thus outperforms the POD for this test case even in the sequential setting.

As we observe that the singular values of the transfer operators do not decay as fast as in case of Example 2, see \cref{figure_SPE10_svals_quantiles} (left) (and cf. \cref{figure_stove_transfer_singular_vals} (right) and the discussion in \cref{subsec_num_ex_1}), we investigate how the approximation quality depends on the number of random initial conditions per chosen time point. In \cref{figure_SPE10_svals_quantiles} (right) we see that the accuracy of the approximation improves if we choose two (S2) instead of one (1,$\,$S1) random initial condition per drawn time point. For three or more random initial conditions (S3-S5), the approximation accuracy is at a comparable level with S2. Moreover, we can alternatively improve the approximation quality for this test case by increasing the number of drawn time points $\nrand$. If we employ the same computational budget as for S2, but drawn only one random initial condition for $30$ instead of $20$ drawn time points (S1$*$), we observe that in $97\%$ of cases the relative $L^2(I,H^1(D))$-error is below $2\cdot 10^{-2}$. Using the same computational budget we thus achieve that the algorithm succeeds to detect all different configurations of the time-dependent data functions in $97\%$ instead of $88\%$ of cases.

\begin{figure}
	\hspace{-0.25cm}
	\begin{tikzpicture}
	\begin{semilogyaxis}[
	width=3.5cm,
	height=5cm,
	xmin=-0.5,
	xmax=3.5,
	ymin=1.3e-7,
	ymax=3e-1,
	legend style={at={(1.01,1)},anchor=north west,font=\footnotesize},
	grid=both,
	grid style={line width=.1pt, draw=gray!70},
	major grid style={line width=.2pt,draw=gray!70},
	xtick={0,1,2,3},
	ytick={1e-8,1e-6, 1e-4, 1e-2},
	minor ytick={1e-7, 1e-5, 1e-3, 1e-1},
	xticklabels={10,15,20,25},
	xlabel style = {yshift = 3},
	ylabel style = {yshift = -4},
	xlabel= $\nt$,
	ylabel= relative $L^2(H^1)$-error,
	label style={font=\footnotesize},
	tick label style={font=\footnotesize}  
	]
	\addplot[only marks, brown, mark=*, mark size=1.8pt] table[x expr=\coordindex, y index=0] {data_figures/SPE10_quantiles_rel_L2H1.dat};
	\addplot[only marks, blue, mark=x, mark size=2.5pt,thick] table[x expr=\coordindex, y index=1] {data_figures/SPE10_quantiles_rel_L2H1.dat};
	\addplot[only marks, red, mark=square*, mark size=1.7pt] table[x expr=\coordindex, y index=2] {data_figures/SPE10_quantiles_rel_L2H1.dat};
	\addplot[only marks,violet,mark=asterisk, mark size=2.5pt,thick] table[x expr=\coordindex, y index=3] {data_figures/SPE10_quantiles_rel_L2H1.dat};
	\addplot[only marks, olive, mark=pentagon*, mark size=2.2pt] table[x expr=\coordindex, y index=4] {data_figures/SPE10_quantiles_rel_L2H1.dat};
	\addplot[only marks,black, mark=star,mark size=2.5pt, thick] table[x expr=\coordindex, y index=5] {data_figures/SPE10_quantiles_rel_L2H1.dat};
	\addplot[only marks,cyan,mark=triangle*, mark size=2.5pt] table[x expr=\coordindex, y index=6] {data_figures/SPE10_quantiles_rel_L2H1.dat};
	\addplot[only marks, orange, mark=asterisk, mark size=2.5pt,thick] table[x expr=\coordindex, y index=7] {data_figures/SPE10_quantiles_rel_L2H1.dat};
	\addplot[only marks,teal, mark=diamond*, mark size=2.5pt] table[x expr=\coordindex, y index=8] {data_figures/SPE10_quantiles_rel_L2H1.dat};
	\addplot[only marks, gray, mark= oplus*, mark size=1.8pt] table[x expr=\coordindex, y index=9] {data_figures/SPE10_quantiles_rel_L2H1.dat};
	\legend{max\hspace*{-4pt},97,88,75,60,50,40,25,5,min\hspace*{-4pt}};
	\end{semilogyaxis}
	\end{tikzpicture}
	\hspace{-0.125cm}
	\begin{tikzpicture}
	\begin{axis}[
	width=3.5cm,
	height=5cm,
	xmin=-0.5,
	xmax=3.5,
	ymin=19,
	ymax=67,
	legend style={at={(1.02,1)},anchor=north west,font=\footnotesize},
	grid=both,
	grid style={line width=.1pt, draw=gray!70},
	major grid style={line width=.2pt,draw=gray!70},
	xtick={0,1,2,3},
	ytick={20,30,40,50,60},
	minor ytick={15,25,35,45,55,65},
	xticklabels = {10,15,20,25},
	xlabel style = {yshift = 3},
	ylabel style = {yshift = -2},
	xlabel= $\nt$,
	ylabel= reduced dimension,
	label style={font=\footnotesize},
	tick label style={font=\footnotesize}  
	]
	\addplot[only marks, brown, mark=*, mark size=1.8pt] table[x expr=\coordindex, y index=0] {data_figures/SPE10_quantiles_red_sizes.dat};
	\addplot[only marks,magenta,mark=square*, mark size=1.7pt] table[x expr=\coordindex, y index=1] {data_figures/SPE10_quantiles_red_sizes.dat};
	\addplot[only marks,black, mark=star,mark size=2.5pt, thick] table[x expr=\coordindex, y index=2] {data_figures/SPE10_quantiles_red_sizes.dat};
	\addplot[only marks,cyan,mark=triangle*, mark size=2.5pt] table[x expr=\coordindex, y index=3] {data_figures/SPE10_quantiles_red_sizes.dat};
	\addplot[only marks, orange, mark=asterisk, mark size=2.5pt,thick] table[x expr=\coordindex, y index=4] {data_figures/SPE10_quantiles_red_sizes.dat};
	\addplot[only marks,teal, mark=diamond*, mark size=2.5pt] table[x expr=\coordindex, y index=5] {data_figures/SPE10_quantiles_red_sizes.dat};
	\addplot[only marks, gray, mark= oplus*, mark size=1.8pt] table[x expr=\coordindex, y index=6] {data_figures/SPE10_quantiles_red_sizes.dat};
	\legend{max\hspace*{-4pt},95,75,50,25,5,min\hspace*{-4pt}};
	\end{axis}
	\end{tikzpicture}
	\hspace{-0.125cm}
	\begin{tikzpicture}
	\begin{semilogyaxis}[
	width=5cm,
	height=5cm,
	xmin=0,
	xmax=10,
	ymin=1e-12,
	ymax=2.5e0,
	legend style={at={(0.02,0.5)},anchor=north west,font=\footnotesize},
	grid=both,
	grid style={line width=.1pt, draw=gray!70},
	major grid style={line width=.2pt,draw=gray!70},
	xtick={0,2,4,6,8,10},
	ytick={1e-16,1e-14,1e-12,1e-10,1e-8,1e-6, 1e-4, 1e-2,1e-0},
	minor ytick={1e-11,1e-9,1e-7, 1e-5, 1e-3, 1e-1},
	xlabel= $t$,
	ylabel= relative $L^2(t)$-error,
	ylabel style = {yshift = -4},
	label style={font=\footnotesize},
	xlabel style = {yshift=4},
	tick label style={font=\footnotesize}  
	]
	\addplot[densely dashdotted,purple,  very thick] table[x index = 0, y index = 1]{data_figures/SPE10_compare_POD_random.dat};
	\addplot[solid ,darkgray, thick] table[x index = 0, y index = 2]{data_figures/SPE10_compare_POD_random.dat};
	\legend{POD, random}
	\end{semilogyaxis}
	\end{tikzpicture}
	\caption{\footnotesize Example 4: Quantiles of relative $L^2(I,H^1(D))$-error (left) and reduced dimension (middle) for varying numbers of $\nt$, $k = \nt-2$, $\nrand=20\,(10+10)$, $\tol = 10^{-8}$, and $25.000$ realizations. Relative $L^2(t)$-error for one realization of \cref{randomized_basis_generation} for $\nt=15$ (random) vs. POD on solution trajectory of first $315$ of $500$ time steps with tolerance $10^{-8}$ (POD) (right).
	}
	\label{figure_SPE10_quantiles_POD}
\end{figure}
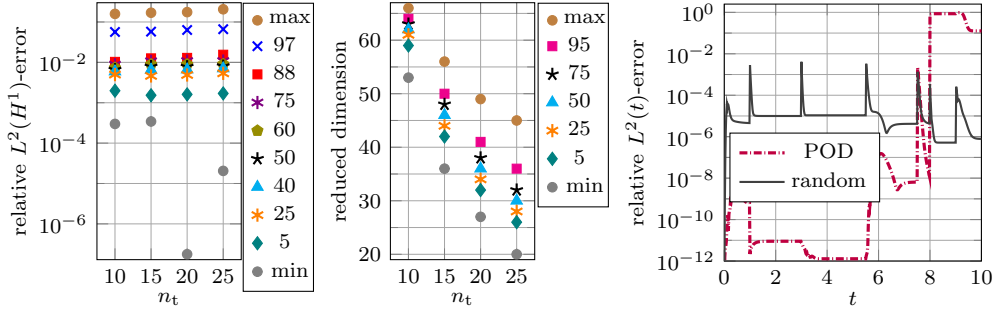

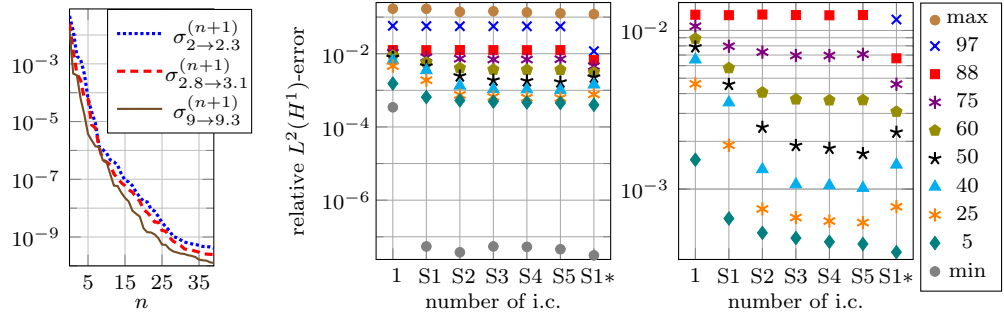
\begin{figure}
	\hspace*{-0.2cm}
	\begin{tikzpicture}
	\begin{semilogyaxis}[
	width=3.5cm,
	height=5cm,
	xmin=0,
	xmax=39,
	ymin=1e-10,
	ymax=8e-2,
	legend style={at={(0.26,1)},anchor=north west,font=\footnotesize},
	grid=both,
	grid style={line width=.1pt, draw=gray!50},
	major grid style={line width=.2pt,draw=gray!50},
	xtick={5,15,25,35},
	%minor xtick = {1,3,5,7,9},
	ytick={1e-9,1e-7, 1e-5,1e-3},
	minor ytick={1e-10,1e-8, 1e-6, 1e-4,1e-2,1e-1},
	xlabel= $n$,
	label style={font=\footnotesize},
	xlabel style = {yshift = 3},
	tick label style={font=\footnotesize}  
	]
	\addplot+[very thick,densely dotted, mark=none] table[x expr=\coordindex,y index=0]{data_figures/SPE10_transfer_singular_vals_nt=15.dat};
	\addplot+[very thick,densely dashed, mark=none] table[x expr=\coordindex,y index=1]{data_figures/SPE10_transfer_singular_vals_nt=15.dat};
	\addplot+[thick, mark=none] table[x expr=\coordindex,y index=2]{data_figures/SPE10_transfer_singular_vals_nt=15.dat};
	\legend{\hspace*{-0.05cm}$\sigma^{(n+1)}_{2\rightarrow 2.3}$\hspace*{-0.15cm}, \hspace*{-0.05cm}$\sigma^{(n+1)}_{2.8\rightarrow 3.1}$\hspace*{-0.15cm},\hspace*{-0.05cm}$\sigma^{(n+1)}_{9\rightarrow 9.3}$\hspace*{-0.15cm}}
	\end{semilogyaxis}
	\end{tikzpicture}
	\hspace*{0.1cm}
	\begin{tikzpicture}
	\begin{semilogyaxis}[
	width=4.7cm,
	height=5cm,
	xmin=-0.5,
	xmax=6.5,
	ymin=2.4e-8,
	ymax=2.5e-1,
	legend style={at={(1.02,1)},anchor=north west,font=\footnotesize},
	grid=both,
	grid style={line width=.1pt, draw=gray!70},
	major grid style={line width=.2pt,draw=gray!70},
	xlabel= number of i.c.,
	label style={font=\footnotesize},
	xlabel style = {yshift = 3},
	xtick={0,1,2,3,4,5,6},
	ytick={1e-8,1e-6, 1e-4, 1e-2},
	minor ytick={1e-7, 1e-5, 1e-3, 1e-1},
	xticklabels={1,S1,S2,S3,S4,S5,$\text{S1}*$},
	ylabel= relative $L^2(H^1)$-error,
	ylabel style={yshift=-2},
	tick label style={font=\footnotesize,xshift=1.5}  
	]
	\addplot[only marks, brown, mark=*, mark size=1.8pt] table[x expr=\coordindex, y index=0] {data_figures/SPE10_quantiles_rel_L2H1_S.dat};
	\addplot[only marks, blue, mark=x, mark size=2.5pt,thick] table[x expr=\coordindex, y index=1] {data_figures/SPE10_quantiles_rel_L2H1_S.dat};
	\addplot[only marks, red, mark=square*, mark size=1.7pt] table[x expr=\coordindex, y index=2] {data_figures/SPE10_quantiles_rel_L2H1_S.dat};
	\addplot[only marks,violet,mark=asterisk, mark size=2.5pt,thick] table[x expr=\coordindex, y index=3] {data_figures/SPE10_quantiles_rel_L2H1_S.dat};
	\addplot[only marks, olive, mark=pentagon*, mark size=2.2pt] table[x expr=\coordindex, y index=4] {data_figures/SPE10_quantiles_rel_L2H1_S.dat};
	\addplot[only marks,black, mark=star,mark size=2.5pt, thick] table[x expr=\coordindex, y index=5] {data_figures/SPE10_quantiles_rel_L2H1_S.dat};
	\addplot[only marks,cyan,mark=triangle*, mark size=2.5pt] table[x expr=\coordindex, y index=6] {data_figures/SPE10_quantiles_rel_L2H1_S.dat};
	\addplot[only marks, orange, mark=asterisk, mark size=2.5pt,thick] table[x expr=\coordindex, y index=7] {data_figures/SPE10_quantiles_rel_L2H1_S.dat};
	\addplot[only marks,teal, mark=diamond*, mark size=2.5pt] table[x expr=\coordindex, y index=8] {data_figures/SPE10_quantiles_rel_L2H1_S.dat};
	\addplot[only marks, gray, mark= oplus*, mark size=1.8pt] table[x expr=\coordindex, y index=9] {data_figures/SPE10_quantiles_rel_L2H1_S.dat};
	\end{semilogyaxis}
	\end{tikzpicture}
	\hspace{-0.4cm}
	\begin{tikzpicture}
	\begin{semilogyaxis}[
	width=4.7cm,
	height=5cm,
	xmin=-0.5,
	xmax=6.5,
	ymin=3.6e-4,
	ymax=1.5e-2,
	legend style={at={(1.03,1)},anchor=north west,font=\footnotesize},
	grid=both,
	grid style={line width=.1pt, draw=gray!70},
	major grid style={line width=.2pt,draw=gray!70},
	xtick={0,1,2,3,4,5,6},
	ytick={1e-3,1e-2,1e-1},
	xlabel= number of i.c.,
	label style={font=\footnotesize},
	xlabel style = {yshift = 3},
	%minor ytick={1e-7, 1e-5, 1e-3, 1e-1},
	xticklabels={1,S1,S2,S3,S4,S5,$\text{S1}*$},
	ylabel style={ yshift=-2},
	tick label style={font=\footnotesize}  
	]
	\addplot[only marks, brown, mark=*, mark size=1.8pt] table[x expr=\coordindex, y index=0] {data_figures/SPE10_quantiles_rel_L2H1_S.dat};
	\addplot[only marks, blue, mark=x, mark size=2.5pt,thick] table[x expr=\coordindex, y index=1] {data_figures/SPE10_quantiles_rel_L2H1_S.dat};
	\addplot[only marks, red, mark=square*, mark size=1.7pt] table[x expr=\coordindex, y index=2] {data_figures/SPE10_quantiles_rel_L2H1_S.dat};
	\addplot[only marks,violet,mark=asterisk, mark size=2.5pt,thick] table[x expr=\coordindex, y index=3] {data_figures/SPE10_quantiles_rel_L2H1_S.dat};
	\addplot[only marks, olive, mark=pentagon*, mark size=2.2pt] table[x expr=\coordindex, y index=4] {data_figures/SPE10_quantiles_rel_L2H1_S.dat};
	\addplot[only marks,black, mark=star,mark size=2.5pt, thick] table[x expr=\coordindex, y index=5] {data_figures/SPE10_quantiles_rel_L2H1_S.dat};
	\addplot[only marks,cyan,mark=triangle*, mark size=2.5pt] table[x expr=\coordindex, y index=6] {data_figures/SPE10_quantiles_rel_L2H1_S.dat};
	\addplot[only marks, orange, mark=asterisk, mark size=2.5pt,thick] table[x expr=\coordindex, y index=7] {data_figures/SPE10_quantiles_rel_L2H1_S.dat};
	\addplot[only marks,teal, mark=diamond*, mark size=2.5pt] table[x expr=\coordindex, y index=8] {data_figures/SPE10_quantiles_rel_L2H1_S.dat};
	\addplot[only marks, gray, mark= oplus*, mark size=1.8pt] table[x expr=\coordindex, y index=9] {data_figures/SPE10_quantiles_rel_L2H1_S.dat};
	\legend{max,97,88,75,60,50,40,25,5,min};
	\end{semilogyaxis}
	\end{tikzpicture}
	\caption{\footnotesize Example 4: Singular values of transfer operators (left). Quantiles of relative $L^2(I,H^1(D))$-error for $\nt=15$, $k=13$, $\nrand=20\,(10+10)$, $\tol=10^{-8}$, $25.000$ realizations, and $1-5$ random initial conditions (i.c.) per time point (right). S indicates that local computations are performed separately for right-hand side and initial conditions, $*$ indicates that $\nrand=30\,(15+15)$.
	}
	\label{figure_SPE10_svals_quantiles}
\end{figure}

\subsection{Choice of parameters in \cref{randomized_basis_generation}}
\label{guidance_parameters}
Based on the experiments above we give the following guidance on how to choose the parameters $\nt$, $k$, and $\nrand$: An appropriate choice for the local oversampling size $\nt$ seems to be $10-15$, while for $\nt=10$ the resulting reduced basis is larger compared to $\nt=15$ and the user can balance computational costs depending on the respective application. In case of coefficients implying a slow spreading or propagation of the solution, a larger $\nt$, for instance, $\nt=30$, might be more favorable. In addition, we propose $k=\nt-2$ as a reasonable choice for diffusion problems, while for advection-dominated problems a smaller $k$, such as $k=\nt-6$, is more favorable. The choice of drawn time points $\nrand$ can be based on the number of available parallel compute units. Moreover, in certain cases it is necessary to also sample from data functions that are constant in time (e.g. a constant advection field) to achieve a good approximation accuracy.

\section{Conclusions}
\label{conclusions}

To tackle time-dependent problems with heterogeneous time-dependent coefficients, we have proposed a randomized algorithm that constructs a reduced approximation space in time by solving several local problems in time in parallel. Based on techniques from randomized NLA \cite{DerMah21,DriMah16,HaMaTr11} points in time are drawn from a data-driven probability distribution and the PDE is solved locally in time using these points as end points with random initial conditions. The approach allows for local error control \cite{BuhSme18} and the computation of the local basis functions is embarrassingly parallel.

The numerical experiments demonstrate that the proposed algorithm can outperform the POD even in the sequential setting for complex problems with heterogeneous time-dependent data functions and is also well capable of tackling higher values of advection. Moreover, we have observed that leverage scores are capable of detecting multiscale features in the data functions.

\appendix

\section{Compactness of the transfer operator in time for the advection-diffusion-reaction problem} \label{supp_section_ex}
	In this section we prove a Caccioppoli inequality and compactness of the transfer operator in time (cf. \cref{transfer_operator}) for the advection-diffusion-reaction problem \cref{PDE_ex} introduced in \cref{sec_problem_setting}. The Caccioppoli inequality in \cref{prop:Caccioppoli} is closely linked to the exponential decay behavior of solutions of the PDE in time (cf. the discussion in \cref{subsection_motivation}) and allows to bound the $L^2(D)$-norm of solutions evaluated at a point of time in terms of their $L^2(I,L^2(D))$-norm. The second key ingredient that we use to prove compactness of the transfer operator in time in \cref{compact_transfer_op} is the compactness theorem of Aubin-Lions \cite[Corollary 5]{Sim86}, which states that the embedding $\left\lbrace v \in L^2(I,H^1_{0}( D)) \mid v_t \in L^2(I,H^{-1}( D))\right\rbrace \hookrightarrow L^{2}(I,L^2(D)) $ is compact. The combination of a Caccioppoli-type inequality with a suitable compactness theorem is usually used to show compactness of the transfer operator; see also \cite{BabLip11,MaSch21,SchSme20,SmePat16,TadPat18}. 
	
	We may consider the following weak formulation (for a proof of well-posedness see, e.g., \cite{Wloka82,SchSte09}): Find the solution $u \in  W^{1,2,2}(I,H^1_{0}( D),H^{-1}( D))\mspace{-2mu}:=\mspace{-2mu} \lbrace v \in L^2(I,H^1_{0}( D))\mspace{-2mu}\mid $ $v_t \in L^2(I,H^{-1}( D)) \rbrace$ such that $u(0)=u_0$ in $L^2( D)$ and
	\begin{align*}
	\begin{split}
	&\int_{I} \la u_t(t),\psi(t)\ra_{H^1_0( D)} dt + \int_{I} (\kappa(t) \nabla u (t), \nabla \psi(t))_{L^2( D)}dt + \int_{I} (b(t) \cdot \nabla u (t), \psi(t))_{L^2( D)}dt\\
	&  + \int_{I} (c(t) u(t),\psi(t))_{L^2(D)} dt\; = \int_{I} \la f(t), \psi(t)\ra_{H^1_0( D)} dt \qquad \forall\,\psi \in L^2(I,H^1_0(D)).
	\end{split}
	\end{align*}
	
	The data functions are introduced in \cref{sec_problem_setting}. In particular, the source terms and initial conditions are given by $f \in L^{2}(I,H^{-1}( D))$ and $u_0 \in L^{2}( D)$. To simplify notations, we assume homogeneous Dirichlet boundary conditions on $I\times\partial D$. However, the theory analogously applies to non-homogeneous Dirichlet boundary conditions.
	
	We highlight that the embedding $W^{1,2,2}(I,H^1_{0}( D),H^{-1}( D)) \hookrightarrow C^{0}(\bar{I},L^{2}(D))$ is not compact unless additional regularity is assumed. In fact, the space 
	\begin{align*}
	W^{1,\infty,r}(I,H^1_{0}( D),H^{-1}( D)):= \left\lbrace v \in L^\infty(I,H^1_{0}( D)) \mid v_t \in L^r(I,H^{-1}( D)) \right\rbrace
	\end{align*}
	embeds compactly in $ C^{0}(\bar{I},L^{2}(D))$ for $r>1$ \cite[Corollary 5]{Sim86}. In contrast, we prove compactness in \cref{supp_section_ex} requiring significantly less regularity. First, the compactness theorem of Aubin-Lions \cite[Corollary 5]{Sim86} states that the embedding $W^{1,2,2}(I,H^1_{0}( D),$ $H^{-1}( D)) \hookrightarrow L^{2}(I,L^2(D))$ is compact. Next, the Caccioppoli inequality that we prove in \cref{prop:Caccioppoli} is the key ingredient that facilitates the restriction to a point of time as the inequality bounds the $L^2(D)$-norm of (local) solutions evaluated at a point of time in terms of their $L^2(I,L^2(D))$-norm. Consequently, we show in \cref{compact_transfer_op} that the space of (local) solutions contained in the generalized Sobolev space $W^{1,2,2}(I,H^1_{0}( D),H^{-1}( D))$ embeds compactly in $C^{0}(\bar{I},L^{2}(D))$.

	In the following, we denote by $t^* \in I$ the local end time point to simplify notations in the proofs. We then consider the local time interval $(s,t^*) \subseteq I$ and seek local solutions 
	$u_\text{loc} \in W^{1,2,2}((s,t^*),H^1_{0}( D),H^{-1}( D))$ with initial conditions $u_\text{loc}(s,\cdot)\in L^2(D)$ such that for all $\psi \in L^2((s,t^*),H^1_0(D))$
	\begin{align}\label{eq:weak_local_advec}
	&\int_{s}^{t^*} \la (u_\text{loc})_t(t),\psi(t)\ra_{H^1_0( D)} dt + \int_{s}^{t^*} (\kappa(t) \nabla u_\text{loc} (t), \nabla \psi(t))_{L^2( D)}dt \\
	+ \mspace{-5mu}\int_{s}^{t^*}\mspace{-12mu} (b(t) \cdot& \nabla u_\text{loc} (t), \psi(t))_{L^2( D)}dt  +\mspace{-7mu} \int_{s}^{t^*} \mspace{-12mu}(c(t) u_\text{loc}(t),\psi(t))_{L^2(D)} dt =\mspace{-7mu} \int_{s}^{t^*} \mspace{-12mu}\la f(t), \psi(t)\ra_{H^1_0( D)}dt.\nonumber
	\end{align}

	\begin{proposition}[Caccioppoli inequality in time]\label{prop:Caccioppoli}
		Let $w$ satisfy \cref{eq:weak_local_advec} with $f \equiv 0$ and arbitrary initial conditions $w(s,\cdot)\in L^2(D)$. Then, we have that
		\begin{equation}\label{eq:Caccioppoli_advec}
		\|w(t^*,\cdot)\|_{L^{2}(D)}^{2}  \leq \frac{2}{(t^*-s)} \|w\|_{L^{2}((s,t^*),L^{2}(D))}^{2}.
		\end{equation}
	\end{proposition}
	
	\begin{proof}
		The first paragraph closely follows the proof of Proposition 3.1 in \cite{SchSme20}.
		Since $w$ satisfies \cref{eq:weak_local_advec}, we can choose $\psi = v \varphi$ for arbitrary $v \in H_0^1(D)$ and $\varphi \in C_0^\infty((s,t^*))$ as a test function in \cref{eq:weak_local_advec}. As $\varphi \in C_0^\infty((s,t^*))$ is chosen arbitrarily, the fundamental lemma of calculus of variations yields that
		$\langle w_t(t),v\ra_{H^1_0(D)}+(\kappa(t) \nabla w(t), \nabla v )_{L^2(D)} + (b(t) \cdot \nabla w (t), v)_{L^2( D)} + (c(t) w(t),v)_{L^2(D)} = 0$ for all $v\in H^1_0(D)$ and almost every $t \in (s,t^*)$. Next, we introduce a cut-off function $\eta \in C^{1}((s,t^*))$ that satisfies $0 \leq \eta \leq 1$, $\eta(s)=0$, $\eta(t^*)=1$, and $|\eta_{t} | \leq \frac{1}{(t^*-s)}$. In the following, we want to use $w\eta^2$ as a test function. To enable rearranging the part of the weak formulation that includes the time derivative, we approximate $w$ by a sequence $w_{n} \in C_0^\infty((s,t^*),H^1_0(D))$ such that $w_{n}$ converges strongly to $w$ in $L^2((s,t^*),H^1_0(D))$. Then, for almost every $t \in (s,t^*)$ and each $n\in \N$ we have that
		\begin{align*}
		&\la w_t(t),w_n(t)\eta^{2}(t) \ra_{H^1_0(D)} + ( \kappa(t) \nabla w(t), \nabla w_n(t)\eta^{2}(t))_{L^2(D)}\\
		& +  (b(t) \cdot \nabla w (t),w_n(t)\eta^{2}(t))_{L^2( D)} + (c(t) w(t),w_n(t)\eta^{2}(t))_{L^2(D)} = 0.
		\end{align*}
		Integrating over the time interval $(s,t^*)$ yields 
		\begin{align}\label{eq:aux_advec}
		\begin{split}
		&\int_s^{t^*} \mspace{-4mu}\la w_t(t),w_n(t)\eta^{2}(t) \ra_{H^1_0(D)} dt +\mspace{-2mu} \int_s^{t^*}\mspace{-4mu} ( \kappa(t) \nabla w(t), \nabla w_n(t)\eta^{2}(t))_{L^2(D)} dt\\
		& + \mspace{-2mu}  \int_s^{t^*} \mspace{-4mu}(b(t) \cdot \nabla w (t),w_n(t)\eta^{2}(t))_{L^2( D)} dt +\mspace{-2mu} \int_s^{t^*} \mspace{-4mu} (c(t) w(t),w_n(t)\eta^{2}(t))_{L^2(D)} dt = 0.
		\end{split}
		\end{align}
		Using integration by parts and $\eta(s)=0$, we rewrite the first term in \cref{eq:aux_advec} as follows:
		\begin{align*}
		&\int_s^{t^*} \la w_t(t),w_n(t)\eta^{2}(t)\ra_{H^1_0(D)} \, dt\\ 
		=& - \int_{s}^{t^*} (w(t), (w_{n}(t))_{t} \eta^{2}(t))_{L^{2}(D)}\,dt - \int_{s}^{t^*} (w(t),w_{n}(t) 2\eta(t) \eta_{t}(t))_{L^{2}(D)} \, dt \\
		& + (w(t^*),w_n(t^*) \eta^2(t^*))_{L^2(D)}\\
		=& - \int_{s}^{t^*} (w(t) \eta(t), (w_{n}(t))_{t} \eta(t))_{L^{2}(D)}\,dt - 2 \int_{s}^{t^*} (w(t) \eta(t) ,w_{n}(t) \eta_{t}(t))_{L^{2}(D)} \, dt \\
		& +  (w(t^*)\eta(t^*),w_n(t^*) \eta(t^*))_{L^2(D)}\\
		=&  \int_{s}^{t^*} \la (w(t)\eta(t))_{t}, w_{n}(t)\eta(t) \ra_{H^1_0(D)} \, dt - \int_{s}^{t^*} (w(t)\eta(t) ,w_{n}(t) \eta_{t}(t))_{L^{2}(D)} \, dt.
		\end{align*}
		Letting $n$ go to $\infty$ thus yields
		\begin{align}\label{eq:aux_advec_2}
		\begin{split}
		&\int_s^{t^*} \la (w(t)\eta(t))_{t},w(t)\eta(t)\ra  dt - \int_{s}^{t^*} (w(t)\eta(t) ,w(t) \eta_{t}(t))_{L^{2}(D)} dt \\ 
		& + \int_s^{t^*} ( \kappa(t) \nabla w(t) , \nabla w(t) \eta^{2}(t))_{L^2(D)} dt +  \int_s^{t^*} (b(t) \cdot \nabla w (t),w(t)\eta^{2}(t))_{L^2( D)} dt\mspace{-10mu}\\
		& + \int_s^{t^*}  (c(t) w(t),w(t)\eta^{2}(t))_{L^2(D)} dt   = 0.
		\end{split}
		\end{align}
		Employing the assumption $c -\frac{1}{2}\nabla \cdot b \geq 0$ and Gauss's theorem we then conclude that 
		\begin{align*}
		&\int_s^{t^*} (b(t) \cdot \nabla w (t),w(t)\eta^{2}(t))_{L^2( D)} dt  + \int_s^{t^*}  (c(t) w(t),w(t)\eta^{2}(t))_{L^2(D)} dt\\
		= &\; \frac{1}{2} \int_{s}^{t^*} \int_D \nabla \cdot (b (w \eta)^2)\, dx\, dt + \int_{s}^{t^*} \int_D (c - \frac{1}{2} \nabla \cdot b) (w\eta)^2\, dx \,dt\\
		\geq &\; \frac{1}{2} \int_{s}^{t^*} \int_{\partial D} b (w \eta)^2 n\, dx\, dt\; = \; 0.
		\end{align*}
		Finally, we exploit that $\int_s^{t^*} ( \kappa(t) \nabla w(t) , \nabla w(t) \eta^{2}(t))_{L^2(D)} dt \geq 0$, the properties of the cut-off function $\eta$ and the identity $\int_s^{t^*} \la (w(t)\eta(t))_{t},w(t)\eta(t)\ra \, dt = \frac{1}{2} \Vert w(t^*)\eta(t^*) \Vert_{L^2(D)}^2$ to infer that
		\begin{align*}
		\|w( t^*,\cdot)\|_{L^{2}(D)}^{2} \leq \frac{2}{(t^*-s)} \|w\|_{L^{2}((s,t^*),L^{2}(D))}^{2}.
		\end{align*}
	\end{proof}

	\begin{proposition}\label{compact_transfer_op}
		The transfer operator $\mathcal{T}_{s\rightarrow t^*}$ in \cref{transfer_operator} is compact.
	\end{proposition}
	
	\begin{proof}
		Let $(\xi_{n})_{n\in \mathbb{N}}$ be a bounded sequence in $L^{2}(D)$. We denote by $(w_{n})_{n\in \mathbb{N}} \subset W^{1,2,2}((s,t^*),H^1_{0}( D),H^{-1}( D))$ the corresponding sequence of solutions of problem \cref{eq:weak_local_advec} with initial conditions $w_n(s) = \xi_{n}$ in $L^2(D)$ and right-hand side $f\equiv 0$, obtaining $\|w_{n}\|_{L^2((s,t^*),H^{1}_{0}(D))} \leq C$ for a constant $0<C<\infty$. Then, there exists a subsequence $(w_{n_l})_{l\in\N}$ and a limit function $w \in L^2((s,t^*),H^1_{0}(D))$ such that the subsequence converges weakly to $w$ in $L^2((s,t^*),H^{1}_{0}(D))$. Thanks to this weak convergence and the Hahn-Banach theorem we infer that $w \in W^{1,2,2}((s,t^*),H^1_{0}(D),H^{-1}(D))$ and that $w$ solves the local PDE \cref{eq:weak_local_advec} for $f\equiv 0$ (cf. proofs of Lemma 3.2 and Theorem A.1 in \cite{SchSme20}) and initial conditions $w(s,\cdot) \in L^{2}(D)$ (due to the embedding $W^{1,2,2}((s,t^*),H^1_0(D),H^{-1}(D))\hookrightarrow C^0([s,t^*],L^2(D))$). Here, we moreover employ that the sequence $((w_{n_l})_t)_{l\in\N}$ converges weakly-$*$ to $w_t$ in $L^2((s,t^*),H^{-1}(D))$ (cf. proof of Theorem A.1 in \cite{SchSme20}).
		As there also holds that $(w_{n_l})_{l\in\N} \subseteq W^{1,2,2}((s,t^*),$ $H^1_{0}(D),H^{-1}(D))$, the compactness theorem of Aubin-Lions \cite[Corollary 5]{Sim86} yields a subsequence $(w_{n_{l_m}})_{m\in\N}$ which converges strongly to $w$ in $L^2((s,t^*),L^2(D))$. Since the sequence $e_{n_{l_{m}}}:=w - w_{n_{l_m}}$ thus solves \cref{eq:weak_local_advec} with $f\equiv 0$ and some initial conditions in $L^{2}(D)$, we may invoke \cref{prop:Caccioppoli} to infer that
		\begin{align*}
		\|w(t^*\mspace{-1mu}, \cdot)\mspace{-2mu} - \mspace{-2mu}w_{n_{l_m}}(t^*\mspace{-1mu}, \cdot)\|_{L^{2}(D)}^{2} \mspace{-4mu}= \mspace{-4mu}\| e_{n_{l_{m}}}( t^*\mspace{-1mu}, \cdot) \|_{L^{2}(D)}^{2}\mspace{-4mu} \leq \mspace{-4mu}\frac{2}{(t^*-s)} \|e_{n_{l_{m}}}\|_{L^{2}((s,t^*),L^{2}(D))}^{2} \mspace{-4mu}\longrightarrow \mspace{-4mu}0.
		\end{align*}
	\end{proof}

\section{Error bounds for squared norm and leverage score sampling} \label{supp:error_bounds}

	In this subsection,	we state and briefly discuss error bounds for the squared norm and leverage score sampling approach introduced in \cref{subsec_choice_prob}. 
	
	The following theorem from \cite{FrKaVe04} gives an additive error bound on the approximation quality of the squared norm sampling approach.
	
	\begin{theorem} $\textnormal{(\cite[Theorem 2]{FrKaVe04}).}$ \label{theorem_error_squared_norm}
		Let $\mathbf{C}\in \R^{N_D \times m}$ be a sample of $m$ columns of $\mathbf{B}$ using the squared norm probability distribution and let $\mathbf{C}^\dagger$ denote the Moore-Penrose inverse of $\mathbf{C}$. Then, with probability at least $0.9$ it holds that
		\begin{align}\label{a_priori_squared_norm}
		\Vert \mathbf{B} -  \mathbf{C}\mathbf{C}^\dagger \mathbf{B} \Vert_F^2 \leq \Vert \mathbf{B} - \mathbf{B}_r \Vert_F^2 + \frac{10 \mspace{1mu}r}{m} \Vert \mathbf{B} \Vert_F^2.
		\end{align}
		Here, $\mathbf{B}_r$ denotes the best rank-$r$ approximation to $\mathbf{B}$.
	\end{theorem}

	The subsequent theorem from \cite{DrMaMu08} gives a multiplicative error bound on the approximation quality of the leverage score sampling approach.
	
	\begin{theorem} $\textnormal{(\cite[Theorem 3]{DrMaMu08}).}$ \label{theorem_error_leverage_scores}
		Let $\varepsilon \in (0,1]$ and let $\mathbf{C}\in \R^{N_D \times m}$ be a sample of $m=3200\,r^2 / \varepsilon^2$ columns of $\mathbf{B}$ using the leverage score probability distribution. Then, with probability at least $0.7$ it holds that
		\begin{align}\label{a_priori_leverage_scores}
		\Vert \mathbf{B} -  \mathbf{C}\mathbf{C}^\dagger \mathbf{B} \Vert_F \leq (1+\varepsilon) \Vert \mathbf{B} - \mathbf{B}_r \Vert_F.
		\end{align}
		Here, $\mathbf{C}^\dagger$ denotes the Moore-Penrose inverse of $\mathbf{C}\mspace{2mu}$ and $\mathbf{B}_r$ denotes the best rank-$r$ approximation to $\mathbf{B}$.
	\end{theorem}
	
	The observations and discussion in \cref{subsec_disc_choice_prob} are in line with the error bounds stated in \cref{theorem_error_squared_norm,theorem_error_leverage_scores}: The rank-$r$ best approximation error in \cref{a_priori_squared_norm,a_priori_leverage_scores} is given by $\Vert \mathbf{B} - \mathbf{B}_r \Vert_F = (\sum_{i=r+1}^{R}\sigma_{i}^2)^{1/2}$ (Eckart-Young theorem e.g. in \cite{GV13}), where $\sigma_{1}\geq \sigma_{2} \geq \ldots$ denote the singular values of $\mathbf{B}$ and $R\leq \min \lb N_D, N_I \rb$ denotes its rank. Moreover, the Frobenius-norm in the last term of \cref{a_priori_squared_norm} is determined by $\Vert \mathbf{B}\Vert_F = (\sum_{i=1}^R \sigma_{i}^2 )^{1/2}$. Consequently, error bound \cref{a_priori_squared_norm} indicates that in the squared norm sampling approach the number $m$ of selected columns needs to be large in order to detect structures that correspond to singular values $\sigma_i$ for which $\sigma_i/\sigma_j$ with $i\in \lb r+1,\ldots,R\rb$ and $j\in \lb 1,\ldots,r\rb$ is small. As $m$ appears in the denominator in the last term of \cref{a_priori_squared_norm}, it has to be chosen as $m \approx 10 \mspace{1mu} r\mspace{1mu} (\sum_{i=1}^{R}\sigma_{i}^2 ) / (\sum_{i=r+1}^R \sigma_{i}^2)$ to achieve an error that is of the order of the best rank-$r$ approximation to $B$. In contrast, the bound for the leverage score sampling approach in \cref{a_priori_leverage_scores} is of the order $(\sum_{i=r+1}^{R}\sigma_{i}^2)^{1/2}$ for a number of selected columns that is independent of the size of the singular values of $\mathbf{B}$.

	\section{Data functions and parameters for Example 1}
	\label{appendix_data_example_1}
	To ensure reproducibility, we list here data functions and discretization parameters corresponding to Example 1 in \cref{subsec_disc_choice_prob}. We choose $I=(0,10)$, $D=(0,1)^2$, $\Sigma_N = \emptyset$, and discretize the spatial domain $D$ with a regular quadrilateral mesh with mesh size $1/50$ in both directions. For the implicit Euler method, we use an equidistant time step size of $1/30$. Furthermore, the two source terms are given by $f_i(t,x,y)= \sum_{j=1}^2 f_{i,j}(t) f_{j}(x,y)$ with $f_{1,1}(t)=4\,\mathbbm{1}_{(t\geq 1)}\mathbbm{1}_{(t \leq 4)}$, $f_{1,2}(t) = \mathbbm{1}_{(t\geq 6)}\mathbbm{1}_{(t \leq 9)}$, $f_{2,1}(t)= \mathbbm{1}_{(t\geq 1)}\mathbbm{1}_{(t \leq 7)}$, $f_{2,2}(t)=\mathbbm{1}_{(t\geq 9)}\mathbbm{1}_{(t \leq 9.2)}$, $f_1(x,y)=\mathbbm{1}_{(x \geq 0.2)} \mathbbm{1}_{(x \leq 0.3)} \mathbbm{1}_{(y \geq 0.2)} \mathbbm{1}_{(y \leq 0.3)}$, and $f_2(x,y) = \mathbbm{1}_{(x \geq 0.7)} \mathbbm{1}_{(x \leq 0.8)} \mathbbm{1}_{(y \geq 0.7)} \mathbbm{1}_{(y \leq 0.8)}$. For computing the probability distributions, we choose $\mathbf{B}=\mathbf{F}$, i.e. the corresponding right-hand side matrix \cref{FE_matrices}.

\section*{Acknowledgments}{The authors thank Dr. Alexander Heinlein for providing us with the data file of the permeability field $\kappa_0$ used in Example 4. Moreover, we thank Dr. Christian Himpe for discussions regarding system and control theory.}

\bibliographystyle{siamplain}

\end{document}